\documentclass[11pt,reqno]{amsart}

\usepackage{geometry}
\usepackage{amsmath,amsthm,amssymb,amsfonts}

\usepackage[utf8]{inputenc} 
\usepackage[T1]{fontenc}    
\usepackage[hidelinks]{hyperref} 
\hypersetup{colorlinks=true,linkcolor=red,citecolor=blue}
\usepackage[capitalise]{cleveref}

\usepackage{comment}
\usepackage{subcaption}
\usepackage{wrapfig}
\usepackage{bm}
\usepackage{url}            
\usepackage{booktabs}       
\usepackage{nicefrac}       
\usepackage{microtype}      

\usepackage{bbm}
\usepackage{cite}
\usepackage{psfrag}
\usepackage{cite}
\usepackage{color,soul}
\usepackage{breakcites}     
\usepackage{bigints}

\usepackage{enumitem}
\setlist{leftmargin=1.6em}

\newtheorem{theorem}{Theorem}
\newtheorem{proposition}{Proposition}

\newtheorem{lemma}{Lemma} 
\newtheorem{corollary}{Corollary}

\theoremstyle{definition}
\newtheorem{definition}{Definition}
\newtheorem{example}{Example}
\theoremstyle{remark}
\newtheorem{remark}{Remark}

\usepackage{graphicx}
\usepackage{xcolor}
\usepackage{svg}

\usepackage{algorithm}
\usepackage[noend]{algpseudocode}

\makeatletter
\def\subsubsection{\@startsection{subsubsection}{3}%
  \z@{.5\linespacing\@plus.7\linespacing}{-.5em}%
  {\normalfont\bfseries}}
\makeatother

\newcommand{\Cov}[0]{\mathrm{Cov}}
\newcommand{\Var}[0]{\mathrm{Var}}

\newcommand{\ind}[0]{\mathbbm{1}}

\newcommand{\calA}[0]{\mathcal{A}}

\newcommand{\calC}[0]{\mathcal{C}}
\newcommand{\calF}[0]{\mathcal{F}}

\newcommand{\calH}[0]{\mathcal{H}}

\newcommand{\calM}[0]{\mathcal{M}}
\newcommand{\calP}[0]{\mathcal{P}}

\newcommand{\calT}[0]{\mathcal{T}}
\newcommand{\calX}[0]{\mathcal{X}}

\newcommand{\supp}{\mathrm{spt}}
\newcommand{\E}[0]{\mathbb{E}}

\newcommand{\R}[0]{\mathbb{R}}

\newcommand{\D}{\mathbb{D}}

\newcommand{\Prob}[0]{\mathbb{P}}

\newcommand{\frakP}{\mathcal P}

\newcommand{\vertiii}[1]{{\left\vert\kern-0.25ex\left\vert\kern-0.25ex\left\vert #1 
    \right\vert\kern-0.25ex\right\vert\kern-0.25ex\right\vert}}

\newcommand{\sS}{\mathsf{S}}

\newcommand{\BB}{\mathbb{B}}

\newcommand{\DD}{\D}

\newcommand{\GG}{\mathbb{G}}

\newcommand{\LL}{\mathbb{L}}

\newcommand{\NN}{\mathbb{N}}

\newcommand{\RR}{\mathbb{R}}

\newcommand{\vasti}{\bBigg@{3.5 }}
\newcommand{\vast}{\bBigg@{4}}
\newcommand{\Vast}{\bBigg@{5}}
\newcommand{\Vastt}{\bBigg@{7}}

\makeatother

\newcommand{\be}{\begin{equation}}
\newcommand{\ee}{\end{equation}}
\newcommand{\ba}{\begin{align}}
\newcommand{\ea}{\end{align}}
\newcommand{\baa}{\begin{align*}}
\newcommand{\eaa}{\end{align*}}

\newif\ifedit

\DeclareMathOperator{\interior}{int}

\DeclareMathOperator{\inte}{int}

\DeclareMathOperator{\lin}{lin}

\begin{document}

\title[Limit Theorems for Entropic Maps 
and the Sinkhorn Divergence]{Limit Theorems for Entropic Optimal Transport Maps
and the Sinkhorn Divergence}

\thanks{
Z. Goldfeld is partially supported by NSF grants CCF-1947801,  CCF-2046018, and DMS-2210368, and the 2020 IBM Academic Award.
K. Kato is partially supported by NSF grants DMS-1952306, DMS-2014636, and DMS-2210368.
G. Rioux is partially supported by the NSERC postgraduate fellowship PGSD-567921-2022.}

\date{First version: July, 12 2022. This version: \today}

\author[Z. Goldfeld]{Ziv Goldfeld}
\address[Z. Goldfeld]{
School of Electrical and Computer Engineering, Cornell University.
}
\email{goldfeld@cornell.edu}

\author[K. Kato]{Kengo Kato}
\address[K. Kato]{
Department of Statistics and Data Science, Cornell University.
}
\email{kk976@cornell.edu}

\author[G. Rioux]{Gabriel Rioux}

\address[G. Rioux]{
Center for Applied Mathematics, Cornell University.}
\email{ger84@cornell.edu}

\author[R. Sadhu]{Ritwik Sadhu}
\address[R. Sadhu]{
Department of Statistics and Data Science, Cornell University.
}
\email{rs2526@cornell.edu}

\begin{abstract}
We study limit theorems for entropic optimal transport (EOT) maps, dual potentials, and the Sinkhorn divergence. The key technical tool we use is a first and second-order Hadamard differentiability analysis of EOT potentials with respect to the marginal distributions, which may be of independent interest. Given the differentiability results, the functional delta method is used to obtain central limit theorems for empirical EOT potentials and maps. The second-order functional delta method is leveraged to establish the limit distribution of the empirical Sinkhorn divergence under the null.
Building on the latter result, we further derive the null limit distribution of the Sinkhorn independence test statistic and characterize the correct order. Since our limit theorems follow from Hadamard differentiability of the relevant maps, as a byproduct, we also obtain bootstrap consistency and asymptotic efficiency of the empirical EOT map, potentials, and Sinkhorn divergence.
\end{abstract}


\maketitle

\section{Introduction}

\subsection{Overview}

Optimal Transport (OT) \cite{kantorovich1942translocation} quantifies the discrepancy between  Borel probability measures $\mu^1,\mu^2$ on $\RR^d$ as 
\begin{equation}
    \mathsf{OT}_c(\mu^1,\mu^2):=\inf_{\pi\in\Pi(\mu^1,\mu^2)}\int c\,\, d\pi,\label{EQ:Kantorovich_OT}
\end{equation}
where $c:\RR^d\times\RR^d\to \R_{+} = [0,\infty)$ is the cost function and $\Pi(\mu^1,\mu^2)$ is the set~of couplings (or plans) between $\mu^1$ and $\mu^2$. Under certain conditions on the distributions and the cost function, the OT plan $\pi$ that achieves the infimum in \eqref{EQ:Kantorovich_OT} concentrates on the graph of a deterministic map $T$, called the OT map or the  Brenier map when $c$ is quadratic \cite{brenier1991polar,gangbo1996}. 
OT tools have been successfully employed for various applications, encompassing machine learning \cite{arjovsky_wgan_2017,gulrajani2017improved,tolstikhin2018wasserstein,bernton2019parameter,chen2021inferential,blanchet2019quantifying,wong2019wasserstein,solomon2015convolutional,courty2016optimal,rubner2000earth,sandler2011nonnegative,li2013novel}, statistics \cite{carlier2016vector,chernozhukov2017monge,panaretos2019statistical,bernton2019approximate,panaretos2020invitation,bigot2020statistical,hallin2021quantile,torous2021optimal,chen2021wasserstein,zhang2021wasserstein,ghosal2019multivariate}, and applied mathematics \cite{jordan1998variational,santambrogio2017}. We refer the reader to \cite{villani2003,villani2008optimal,ambrosio2005,santambrogio15} as standard references on OT theory.

Statistical OT seeks to estimate and carry out inference for OT and related objects thereof based on data. Two central objects of interest are the OT cost, which has natural applications to minimum distance estimation and testing, and the OT map, which is useful for transfer learning and domain adaptation tasks.
Alas, despite its widespread applicability, the OT problem suffers from computational and statistical scalability issues. In general, the OT cost is difficult to compute and its plug-in empirical estimator converges towards the ground truth at the rate $n^{-1/d}$ \cite{dudley1969speed,fournier2015rate,weed2019}, which is known to be minimax optimal without further assumptions \cite{niles2019estimation}. 
Estimation of the Brenier map encounters similar difficulties, as the results of \cite{hutter2021minimax} suggest that the minimax rate would be $n^{-1/d}$ in general (though formally a conjecture). Imposing smoothness on the marginal distributions or the Brenier map can speed up minimax rates, but verification of such assumptions is nontrivial and computations of the minimax optimal estimators tend to be burdensome \cite{hutter2021minimax,deb2021rates,manole2021plugin}.

Entropic OT (EOT) has emerged as an appealing alternative to the classic Kantorovich formulation that circumvents these statistical and computational difficulties.
EOT regularizes the transportation cost by the Kullback-Leibler (KL) divergence as \cite{schrodinger1931uber,leonard2014survey}
\begin{equation}
    \label{eq:EOTProblem}
    \mathsf S_{c,\varepsilon}(\mu^1,\mu^2):=\inf_{\pi\in \Pi(\mu^1,\mu^2)}\int c\,\, d\pi+\varepsilon \mathsf D_{\mathsf{KL}}(\pi||\mu^1\otimes \mu^2),
\end{equation}
where $\varepsilon>0$ is a regularization parameter. As $\varepsilon\to 0$, the EOT problem converges towards OT, not only in terms of the transportation cost but also in optimal plans and dual  potentials \cite{mikami2004monge,mikami2008optimal,leonard2012schrodinger,carlier2017convergence,conforti2021formula,chizat2020faster,pal2019difference,carlier2022convergence,bernton2021entropic,nutz2021entropic,altschuler2022asymptotics}. For fixed $\varepsilon>0$, EOT alleviates the computational and statistical challenges associated with classic OT. Indeed,~EOT between discrete distributions (e.g., empirical distributions) can be efficiently solved via the Sinkhorn algorithm \cite{cuturi2013sinkhorn,altschuler2017near}, whose time complexity scales quadratically in the number of support points. Regarding empirical estimation, EOT and its centered version $\bar{\mathsf{S}}_{c,\varepsilon}(\mu^1,\mu^2)=\mathsf S_{c,\varepsilon}(\mu^1,\mu^2)-(\mathsf S_{c,\varepsilon}(\mu^1,\mu^1)+\mathsf S_{c,\varepsilon}(\mu^2,\mu^2))/2$ (known as the Sinkhorn divergence) enjoy the parametric $n^{-1/2}$ convergence rate under several settings \cite{genevay2019sample, mena2019statistical,delbarrio22EOT,rigollet2022sample}. 
Recent work has further established central limit theorems (CLTs) for the EOT cost in certain cases \cite{mena2019statistical, bigot2019central, klatt2020empirical,delbarrio22EOT,goldfeld2022statistical} as well as estimation rate results for EOT maps and plans \cite{pooladian2021entropic,pooladian2022debiaser,rigollet2022sample}; see a literature review below. Still, much is left to be desired on deeper understanding of limit behaviors of EOT and related objects, such as EOT potentials, maps, and Sinkhorn divergences, whose analysis poses a significant challenge from a probabilistic perspective.

\subsection{Contributions} The present paper contributes to the growing literature on statistical OT  by establishing limit theorems of the aforementioned objects: EOT potentials, maps, and the Sinkhorn divergence. 
The key ingredient of our derivations is a first and second-order Hadamard differentiability analysis of the EOT potentials with respect to (w.r.t.) the marginal distributions. 
Importantly, we establish Hadamard differentiability of the EOT potentials as maps into H\"{o}lder function spaces. The derivation first establishes Hadamard differentiability as maps into the space of continuous functions ($\calC$-space) using the Schr\"{o}dinger system, which characterizes EOT potentials, and a version of the implicit function theorem.  We then lift the Hadamard differentiability to H\"{o}lder spaces  by showing that derivatives of the EOT potentials are again Hadamard differentiable as maps into the $\calC$-space. Having this result, the functional delta method \cite{romisch2004} yields a CLT for the empirical EOT potentials, which, in turn, implies a CLT for the EOT map under the quadratic cost via the continuous mapping theorem. Both limit variables are characterized as Gaussians with values in the appropriate function spaces.

Hadamard differentiability results for the EOT potentials further enable developing a limit distribution theory for the Sinkhorn divergence.
While asymptotic normality under the alternative (when $\mu^1 \ne \mu^2$) is a straightforward consequence of existing EOT limit theorems \cite{delbarrio22EOT,goldfeld2022statistical}, the null case (when $\mu^1=\mu^2$) for general distributions is significantly more challenging and is a subject of this paper. 
The difficulty originates from the first-order Hadamard derivative of the Sinkhorn divergence nullifying when $\mu^1=\mu^2$ (as the functional achieves its global minimum there), which implies that the variance of the empirical Sinkhorn divergence vanishes under the null and the limit degenerates. To overcome this, we employ the second-order functional delta method \cite{romisch2004}, which requires second-order Hadamard derivatives of the EOT potentials and the Sinkhorn divergence. The Hadamard differentiability result of EOT potentials in H\"{o}lder spaces is key to finding such higher-order derivatives. Application of the second-order functional delta method then yields a distributional limit for the empirical Sinkhorn divergence at the rate of $n^{-1}$. To the best of our knowledge, the null limit distribution for the Sinkhorn divergence beyond the discrete case has been an open problem, and our result closes this gap. 

The Sinkhorn divergence was applied to independence testing in \cite{liu2022entropy}, although the null limit distribution of the test statistic was not obtained in that work. \cite{liu2022entropy} proposed critical values of  order $n^{-1/2}$ derived from concentration inequalities, which is at odds with the $n^{-1}$ order fluctuations under the null implied by our limit~theorem. Building on the second-order Hadamard differentiability result for the Sinkhorn divergence and techniques for analysis of $U$-processes \cite{PeGi1999,chen2020jackknife}, we establish the null limit distribution of the Sinkhorn independence test statistic and the correct $n^{-1}$ order. Finally, the Hadamard differentiability results automatically yield bootstrap consistency and asymptotic efficiency of the empirical EOT potentials, map, and Sinkhorn divergence (under the alternative), which is another virtue of our approach.

\subsection{Literature review}\label{subsec: lit review}

Statistical and probabilistic analyses of EOT and related objects have seen active research in the past couple of years. Regarding limit distribution theory, \cite{bigot2019central,klatt2020empirical} derived CLTs for the EOT cost with $c(x_1,y_2)=\|x_1-x_2\|^p$ ($p \in [1,\infty)$) for finitely discrete distributions. \cite{bigot2019central} also derived the null limit of the Sinkhorn divergence in the discrete case by parameterizing it by finite-dimensional simplex vectors and directly finding the Hessian w.r.t. the simplex vectors. This approach does not directly extend to the general case. For general sub-Gaussian distributions,
\cite{mena2019statistical} showed asymptotic normality of $\sqrt{n}\big(\mathsf{S}_{\|\cdot\|^2,\varepsilon}(\hat{\mu}_n^1,\mu^2)-\E\big[\mathsf{S}_{\|\cdot\|^2,\varepsilon}(\hat{\mu}_n^1,\mu^2)\big]\big)$ and its two-sample analog. The main limitation of this result is that the centering term is the expected empirical EOT cost, which precludes performing inference for $\sS_{\|\cdot\|^2,\varepsilon}(\mu^1,\mu^2)$ itself. This limitation was addressed in \cite{delbarrio22EOT}, who derived a limit theorem for the population centering by combining the CLT from \cite{mena2019statistical} with a bias bound of the form $\E\big[\mathsf{S}_{\|\cdot\|^2,\varepsilon}(\hat{\mu}_n^1,\mu^2)\big] - \mathsf{S}_{\|\cdot\|^2,\varepsilon}(\mu^1,\mu^2) =o(n^{-1/2})$. The recent work by the present authors  \cite{goldfeld2022statistical} generalized this result to dependent data and further complemented it with asymptotic efficiency of the empirical EOT cost and consistency of the bootstrap estimate. It is worth noting that \cite{delbarrio22EOT}  derived the $n^{-1}$ rate for the empirical Sinkhorn divergence under the null but did not derive its limit distribution.

Estimation of the EOT plan and map were also studied as a means to obtain computationally efficient proxies of the OT plan and the Brenier map, respectively. \cite{pooladian2021entropic} considered estimation of the Brenier map under the quadratic cost via an entropic approximation; see also \cite{pooladian2022debiaser}. They analyzed the empirical EOT map and established a rate toward the Brenier map by taking $\varepsilon = \varepsilon_n \to 0$, which is however sub-optimal.  \cite{rigollet2022sample} established the parametric rate toward the EOT map with  $\varepsilon > 0$ fixed. CLTs for the empirical EOT plan were studied in \cite{klatt2020empirical} for discrete distributions and \cite{gunsilius2021matching} for more general cases. The latter work also obtained a limit theorem for the EOT potentials in the $\calC$-space, which is derived via a significantly different proof technique than ours (not relying on Hadamard differentiability) and is weaker than the convergence in H\"{o}lder spaces established herein. Hadamard differentiability in H\"{o}lder spaces is crucial for obtaining the null limit distribution of the empirical Sinkhorn divergence, which is one of our main contributions. Furthermore, our argument based on Hadamard differentiability of EOT potentials yields not only limit distributions but also asymptotic efficiency and consistency of the bootstrap. Finally, \cite{harchaoui2020asymptotics} derived limit theorems for a \textit{different} entropic regularization that makes the optimal solution explicit and thus rely on analysis techniques that significantly differ from~ours.

Our Hadamard differentiability results also contribute to the study of stability of EOT, which has attracted growing interest in the mathematics literature \cite{mikami2019regularity,carlier2020differential,mikami2021regularity,ghosal2022stability,eckstein2022quantitative,nutz2023stability}, and hence would be of independent interest beyond statistical applications. Those references study stability of EOT-related objects w.r.t. varying marginals (e.g., weak convergence of marginals) under general settings, but do not contain differentiability results like ours. The recent preprint \cite{rioux2023entropic} studies stability of the cost and dual potentials of the quadratic Gromov-Wasserstein distance with entropic panelty, leveraging the variational form from \cite{Zhang22Gromov} that represents it as an infimum of a sequence of EOT problems.

\subsection{Concurrent work}
The concurrent and independent work \cite{gonzalez2022weak} establishes similar results concerning limit distributions of EOT potentials and the Sinkhorn divergence, but via a markedly different proof technique that does not involve Hadamard derivatives or the functional delta method. They also do not discuss bootstrapping or asymptotic efficiency. Our overlapping results---a mean-zero Gaussian limit for the EOT potentials in the H\"{o}lder space and a non-Gaussian limit for the Sinkhorn divergence under the null (with a scaling factor of $n$)---are consistent with each other, although they derive explicit forms of the  limit distributions.
Compared with \cite{gonzalez2022weak}, our contribution is to formulate and derive Hadamard differentiability for the EOT potentials and the Sinkhorn divergence, including higher-order ones, from which the limit distributions, consistency of resampling methods, and asymptotic efficiency of the empirical estimators automatically follow. The resampling methods enable performing inference without knowing the explicit limits, and our semiparametric efficiency result implies that, though the limit is not explicit, it is the best one can hope for. Furthermore, the Hadamard differentiability results enable deriving limit distributions beyond the empirical estimators, such as for the Sinkhorn independence statistic. As such, we view the contributions of our work and that of \cite{gonzalez2022weak} as complementary to each other.

Another related work is \cite{gonzalez2023weak}, which appeared after the initial version of this paper was posted on arXiv. That work derived several limit theorems for EOT-related objects for nonsmooth costs, combining the approaches of \cite{gonzalez2022weak} and \cite{rigollet2022sample}, but their scope and proof techniques are substantially different. 

\subsection{Organization}

The rest of the paper is organized as follows. In \cref{sec: prelminiaries}, we collect background material on the EOT problem, potentials, map, and Sinkhorn divergence. 
In \cref{sec: main results}, we derive limit distributions of these objectives. We also derive the null limit distribution of the Sinkhorn independence test statistic in \cite{liu2022entropy}. In \cref{sec: EOT potential}, we collect Hadamard differentiability results for the relevant maps, including higher-order ones.  In \cref{sec: efficiency}, we discuss bootstrap consistency and asymptotic efficiency of the empirical estimators. \cref{sec: proofs} contains proofs for Sections \ref{sec: main results} and \ref{sec: EOT potential}. \cref{sec: conclusion} leaves some concluding remarks and discussions on the extensions to the unbounded support case and multimarginal EOT. The Appendix contains  additional results concerning the $m$-out-of-$n$ bootstrap for the Sinkhorn null limit (including small-scale numerical experiments), proofs that are omitted from the main text, technical tools used in the proofs, and other auxiliary results.  

\subsection{Notation}\label{subsec: notation}
 For a subset $A$ of a topological space~$S$, let $\overline{A}^S$ denote the closure of $A$.
We use $\calP(S)$ to denote the set of Borel probability measures on $S$. For $\mu \in \calP(S)$, $\supp(\mu)$ denotes its support.
For a nonempty set $S$, let $\ell^\infty (S)$ denote the Banach space of bounded real functions on $S$ equipped with the sup-norm $\| \cdot \|_{\infty,S} = \sup_{x \in S}| \cdot |$. 
For every compact set $\calX \subset \R^d$, let $\calC(\calX)$ denote the Banach space of continuous functions on $\calX$ equipped with the sup-norm $\| \cdot \|_{\infty,\calX}$. 
For every multi-index $k=(k_1,\dots,k_d) \in \mathbb{N}_0^d$ with $|k| = \sum_{j=1}^ d k_j$ ($\NN_0 = \NN \cup \{ 0 \}$), define the differential operator $D^k$ by $D^k = \frac{\partial^{|k|}}{\partial x_1^{k_1} \cdots \partial x_{d}^{k_d}}$
with $D^0 f = f$. 
For every $s \in \NN_0$ and nonempty compact set $\calX \subset \R^d$ that agrees with the closure of its interior, $\calC^s(\calX)$ denotes the set of functions $f$ on $\calX$ such that $f$ has continuous derivatives of all orders $\le s$ on $\inte (\calX)$ and the derivatives have continuous extensions to $\calX$ ($\calC^0 (\calX) = \calC (\calX)$). Define the norm $\| f \|_{\calC^s(\calX)} = \sum_{j=0}^s  \max_{|k| = j} \| D^k f \|_{\infty, \inte (\calX)}$; ($\calC^s(\calX),\| \cdot \|_{\calC^s(\calX)})$
 is a separable Banach space [see Problem 5.1 in \cite{gilbarg2015elliptic}; separability follows by noting that $f \mapsto (D^k f)_{|k| \le s}$ is isomorphic from $\calC^s(\calX)$ onto a closed subset of $\prod_{k: |k| \le s} \calC(\calX)$]. We often identify a finite signed Borel measure $\gamma$ on $\R^d$ with the linear functional $f \mapsto \gamma (f) := \int f \, d\gamma$ defined on the bounded Borel functions on $\R^d$.
For two real numbers $a$ and $b$, let $a \vee b = \max \{a,b \}$.

\section{Background and Preliminaries}
\label{sec: prelminiaries}

\subsection{EOT problem}
In this paper, we study EOT problems with smooth cost functions for compactly supported distributions on $\R^d$. 
We briefly review basic definitions and results concerning EOT problems.
Let $c: \R^d \times \R^d \to \R_{+}$ be a smooth (i.e., infinitely differentiable) cost function.
To simplify exposition concerning the Sinkhorn divergence, we will assume that $c$ is symmetric, i.e., $c(x_1,x_2) = c(x_2,x_1)$ for all $x_1,x_2 \in \R^d$. 
A canonical example is the quadratic cost $c(x_1,x_2) = \| x_1-x_2\|^2/2$. The corresponding EOT problem for compactly supported distributions $\mu^1,\mu^2$ on $\R^d$ is defined as
\begin{equation}
\mathsf{S}_{c,\varepsilon} (\mu^1,\mu^2) = \inf_{\pi \in \Pi (\mu^1,\mu^2)} \int c \, \, d\pi + \varepsilon \mathsf{D}_{\mathsf{KL}}(\pi || \mu^1 \otimes \mu^2), 
\label{eq: EOT}
\end{equation}
where $\varepsilon>0$ is a regularization parameter, $\Pi (\mu^1,\mu^2)$ is the set of couplings of $(\mu^1,\mu^2)$,  and 
$\mathsf{D}_{\mathsf{KL}}$ is the Kullback-Leibler divergence defined by
\[
\mathsf{D}_{\mathsf{KL}}(\alpha\|\beta):=
\begin{cases}
\int\log (d \alpha/d \beta) \, d\alpha & \text{if} \ \alpha \ll \beta \\
+\infty & \text{otherwise}
\end{cases}.
\]
The EOT problem admits a unique solution $\pi^\star$, which we call the \textit{EOT plan}. Throughout this paper, we assume that the regularization parameter $\varepsilon > 0$ is fixed, so we often omit the dependence on $\varepsilon$.

The EOT problem admits strong duality, which reads as
\begin{equation}
\mathsf{S}_{c,\varepsilon} (\mu^1,\mu^2) = \sup_{\bm{\varphi} = (\varphi_1,\varphi_2)} \int \varphi_1 \, d\mu^1 + \int \varphi_2 \, d\mu^2 - \varepsilon \int e^{\frac{\varphi_1 \oplus \varphi_2 -c}{\varepsilon}} \, \, d\mu^1 \otimes \mu^2 + \varepsilon,
\label{eq: dual}
\end{equation}
where $(\varphi_1  \oplus \varphi_2 )(x_1,x_2) = \varphi_1(x_1) + \varphi_2(x_2)$ and  the supremum is taken over all $\bm{\varphi} = (\varphi_1,\varphi_2) \in L^1 (\mu^1) \times L^1(\mu^2)$.
There exist functions $\bm{\varphi} = (\varphi_1,\varphi_2) \in L^1 (\mu^1) \times L^1(\mu^2)$ that achieve the supremum in the dual problem \eqref{eq: dual}, which we call \textit{EOT potentials}. EOT potentials are a.e. unique up to additive constants in the sense that if $(\tilde{\varphi}_1,\tilde{\varphi}_2)$ is another pair of EOT potentials, then there exists a constant $a \in \R$ such that $\tilde{\varphi}_1 = \varphi_1 +a $ $\mu^1$-a.e. and $\tilde{\varphi}_2 = \varphi_2  -a $ $\mu^2$-a.e.  
A pair of functions $\bm{\varphi} \in L^1 (\mu^1) \times L^1(\mu^2)$ are EOT potentials if and only if they satisfy the so-called \textit{Schr\"{o}dinger system}
 \[
\int e^{\frac{\varphi_1 \oplus \varphi_2 - c}{\varepsilon}} \, d\mu^{j} - 1 = 0 \quad \text{$\mu^i$-a.e., $i \ne j$},
 \]
where  $\mu^j$ acts on the $j$-th coordinate, i.e.,
 \[
 \int e^{\frac{\varphi_1 \oplus \varphi_2 - c}{\varepsilon}} \, d\mu^2 = \int e^{\frac{\varphi_1(\cdot) + \varphi_2 (x_2) - c(\cdot,x_2)}{\varepsilon}} \, d\mu^2(x_2).
 \]
Given EOT potentials $(\varphi_1,\varphi_2)$, the (unique) EOT plan $\pi^\star$ can be expressed as
\begin{equation}
\, d\pi^\star = e^{\frac{\varphi_1 \oplus \varphi_2 - c}{\varepsilon}} d(\mu^1 \otimes \mu^2).
\label{eq: EOT plan}
\end{equation}
 See Section 1 in \cite{nutz2021entropic} and the references therein for the above results.

\subsection{EOT potentials}\label{subsec:EOT potentials} 
In what follows, we deal with distributions supported in a compact set $\calX \subset \R^d$. We will maintain the following assumption throughout the paper: 
\[
\text{the set $\calX \subset \R^d$ is a bounded closed ball},
\]
where we implicitly assume that the radius of $\calX$ is sufficiently large to contain the supports of the population distributions.

Our limit theorems  rely on regularity properties of EOT potentials.
These properties are summarized in the following lemma (proved in \cref{sec: auxiliary proofs}), where the notation $\equiv$ is used to represent equality that holds everywhere on the domain.

\begin{lemma}[Regularity of EOT potentials]
\label{lem: EOT potential}
Pick and fix an arbitrary reference point 
$(x_1^\circ,x_2^\circ) \in \calX \times \calX$.
The following hold. 
\begin{enumerate}
    \item[(i)] For every 
    $\bm{\mu}=(\mu^1,\mu^2) \in \calP(\calX) \times \calP(\calX),$
    there exist EOT potentials 
    $(\varphi_1,\varphi_2) \in \calC(\calX) \times \calC(\calX)$
    such that 
\begin{equation}
\int e^{\frac{\varphi_1(\cdot) + \varphi_2(x_2) - c(\cdot,x_2)}{\varepsilon}} \, d\mu^2(x_2) - 1 \equiv 0, 
  \int e^{\frac{\varphi_1 (x_1) + \varphi_2(\cdot) - c(x_1,\cdot)}{\varepsilon}} \, d\mu^1(x_1) - 1 \equiv 0. 
 \label{eq: FOC}
 \end{equation}
Furthermore, if $(\tilde{\varphi}_1,\tilde{\varphi}_2)$ is another pair of EOT potentials satisfying \eqref{eq: FOC}, then there exists $a \in \R$ such that $(\tilde{\varphi}_1,\tilde{\varphi}_2) \equiv (\varphi_1+a,\varphi_2-a)$. Hence, there exists a unique pair of functions
$\bm{\varphi}^{\bm{\mu}} = (\varphi_1^{\bm{\mu}},\varphi_2^{\bm{\mu}}) \in \calC(\calX) \times \calC(\calX)$
that satisfies \eqref{eq: FOC} and $\varphi_1^{\bm{\mu}} (x_1^\circ) = \varphi_2^{\bm{\mu}}(x_2^\circ)$.
\item[(ii)] For every $s \in \mathbb{N}$, there exists $R_s > 0$ that may depend on $c, \varepsilon, d, \calX$, 
such that 
$\| \varphi_1^{\bm{\mu}} \|_{\calC^s(\calX)} \vee \| \varphi_2^{\bm{\mu}} \|_{\calC^s(\calX)} \le R_s$
for all 
$\bm{\mu} \in \calP(\calX) \times \calP(\calX)$.
\item[(iii)] Fix arbitrary $s \in \NN$ and equip $\calP(\calX)$ with the topology of weak convergence. Then, the map $\bm{\mu} \mapsto \bm{\varphi}^{\bm{\mu}}, \calP(\calX) \times \calP(\calX) \to \calC^s(\calX) \times \calC^s(\calX)$
is continuous, i.e., if  each $\mu_n^i$ converges weakly to $\mu^i$ ($i=1,2$), then $\bm{\varphi}^{\bm{\mu}_n} \to \bm{\varphi}^{\bm{\mu}}$ in 
$\calC^s(\calX) \times \calC^s(\calX)$.
\end{enumerate}
\end{lemma}

\subsection{EOT map}
The EOT map is an efficiently computable surrogate of the Brenier map. Recall that the (vanilla) OT problem with the quadratic cost between $(\mu^1,\mu^2)$ with absolutely continuous $\mu^1$ admits a ($\mu^1$-a.e.) unique OT map (called the Brenier map) $T:\R^d \to \R^d$ and the (unique) OT plan concentrates on the graph of $T$,  $\{ (x_1,T(x_1)) : x_1 \in \supp(\mu^1) \}$.
Hence, the Brenier map agrees with the conditional expectation of the second coordinate given the first under the OT plan (also called the barycenter projection). Motivated by this observation, \cite{pooladian2021entropic} considered the following EOT analog of the Brenier map. 
 
\begin{definition}[EOT map]
\label{def: EOT map}
Consider the quadratic cost $c (x_1,x_2) =  \| x_1 - x_2 \|^2/2$. For 
$\bm{\mu} = (\mu^1,\mu^2) \in \calP(\calX) \times \calP(\calX)$,
the \textit{EOT map} is defined by
\[
T^{\bm{\mu}}= \E_{\pi^\star}[X_2 \mid X_1=\cdot], \ (X_1,X_2) \sim \pi^\star,
\]
where $\pi^\star$ is the unique EOT plan for $(\mu^1,\mu^2)$.
\end{definition}

The EOT map is \textit{a priori} defined only $\mu^1$-a.e. However, as noted in \cite{pooladian2021entropic}, using the expression \eqref{eq: EOT plan} of the EOT plan, we may define a version of the conditional expectation for all 
$x_1\in\calX$
(and indeed $x_1 \in \R^d$) as 

\begin{equation}
T^{\bm{\mu}} (x_1) = \frac{\int_{\calX} x_2 e^{\frac{\varphi_2^{\bm{\mu}}(x_2) - \| x_1-x_2\|^2/2}{\varepsilon}}\, d\mu^2(x_2)}{\int_{\calX} e^{\frac{\varphi_2^{\bm{\mu}}(x_2) - \| x_1-x_2\|^2/2}{\varepsilon}}\, d\mu^2(x_2)}, \quad x_1 \in \calX.
\label{eq: EOT map}
\end{equation}
We always choose this version throughout the  paper. Just as the Brenier map, the EOT map can be characterized in terms of the gradient of the EOT potential. 

\begin{lemma}[Proposition 2 in \cite{pooladian2021entropic}]
\label{lem: brenier}
Under the setting of \cref{def: EOT map}, we have
$T^{\bm{\mu}}(x_1) = x_1 - \nabla \varphi_1^{\bm{\mu}} (x_1)$ for 
$x_1 \in \calX$.
\end{lemma}

\cite{pooladian2021entropic} used the EOT map as a means to estimate the (original) Brenier map by taking $\varepsilon = \varepsilon_n \to 0$. Here, we view the EOT map $T^{\bm{\mu}}$ as an object of interest on its own right, rather than approximating the Brenier map.

\begin{example}[Vector quantile function]
\label{ex: vector quantile}
The Brenier map can be interpreted as a vector version of the quantile function when $\mu^1$ is taken as a known reference measure, such as $\mu^1 = \mathrm{Unif} [0,1]^d$ \cite{chernozhukov2017monge,carlier2016vector,ghosal2019multivariate}. Indeed, for $d=1$, the Brenier map sending $\mu^1=\mathrm{Unif} [0,1]$ to $\mu^2$ agrees with the quantile function of $\mu^2$ \cite{santambrogio15}. The EOT map $T^{\bm{\mu}}$ can thus be viewed as an efficiently computable  surrogate of the vector quantile function \cite{carlier2020vector}. 
\end{example}

\subsection{Sinkhorn divergence} One drawback of  EOT in  \eqref{eq: EOT} is that it is not a metric even for distance-like costs, such as $c(x_1,x_2) = \| x_1-x_2 \|^{p}$, $p \ge 1$. In fact, EOT is not even a divergence since $\mathsf{S}_{c,\varepsilon}(\mu^1,\mu^2) \ne 0$ when $\mu^1 = \mu^2$, which renders it incompatible for applications to homogeneity and independence testing.\footnote{A divergence on the space of probability distributions is a mapping to the extended reals that is nonnegative and nullifies if and only if the distributions are the same.}  To remedy this issue, a popular approach is center EOT to obtain the Sinkhorn divergence:
\begin{equation}
\bar{\mathsf{S}}_{c,\varepsilon}(\mu^1,\mu^2) = \mathsf{S}_{c,\varepsilon} (\mu^1,\mu^2)-\frac{1}{2}\big( \mathsf{S}_{c,\varepsilon} (\mu^1,\mu^1) + \mathsf{S}_{c,\varepsilon} (\mu^2,\mu^2) \big).
\label{eq: sinkhorn}
\end{equation}
Under certain regularity conditions on the cost function (satisfied by the quadratic cost), the Sinkhorn divergence  satisfies $\bar{\mathsf{S}}_{c,\varepsilon}(\mu^1,\mu^2) \ge 0$ and $\bar{\mathsf{S}}_{c,\varepsilon}(\mu^1,\mu^2)= 0$ if and only if $\mu^1=\mu^2$ \cite{feydy2019interpolating}. 

Assuming $x_1^\circ = x_2^\circ$, by duality \eqref{eq: dual} and  $\varphi_1^{(\mu^i,\mu^i)} = \varphi_2^{(\mu^i,\mu^i)}$ (which follows by symmetry of the cost function), the Sinkhorn divergence admits the following expression \cite{pooladian2022debiaser}:
\begin{equation}
\label{eq: Sinkhorn dual}
\bar{\mathsf{S}}_{c,\varepsilon}(\mu^1,\mu^2)=\int (\varphi_1^{(\mu^1,\mu^2)} - \varphi_1^{(\mu^1,\mu^1)}) \, d\mu^1 + \int (\varphi_2^{(\mu^1,\mu^2)} - \varphi_2^{(\mu^2,\mu^2)}) \, d\mu^2.
\end{equation}
We will use this expression when studying the null limit distribution of the empirical Sinkhorn divergence.

\section{Main results}
\label{sec: main results}

In this section, we derive limit distributions for the empirical EOT potentials, map, and Sinkhorn divergence. For $\mu^i \in \calP(\calX)$,
$i=1,2$, let $\hat{\mu}_n^i$ denote the empirical distribution of a sample of $n$ independent observations from $\mu^i$, i.e., $\hat{\mu}_n^i = n^{-1}\sum_{j=1}^n \delta_{X_j^i}$, where $X_1^i,\dots,X_n^i$ are i.i.d. according to $\mu^i$. The samples from $\mu^1$ and $\mu^2$ are assumed to be independent, and we set $\hat{\bm{\mu}}_n = (\hat{\mu}_n^1,\hat{\mu}_n^2)$. 

\subsection{Limit theorems for EOT potentials and  maps}
\cref{lem: EOT potential} (ii) implies that EOT potentials lie in the H\"{o}lder space $\calC^s(\calX)$ for arbitrary $s \in \NN$.
The first main result concerns a limit distribution for the empirical EOT potentials in 
$\calC^s(\calX) \times \calC^s(\calX)$,
which, in view of \cref{lem: brenier}, automatically yields a limit distribution  for the empirical EOT map in 
$\calC^{s-1} (\calX;\R^d)$.
Recall that a random variable $G$ with values in a (real) Banach space $\BB$ is called Gaussian if for every $b^* \in \BB^*$ (the topological dual of $\BB$), $b^*G$ is a real-valued Gaussian random variable. We say that $G$ has mean zero if $b^*G$ does so for every $b^* \in \BB^*$. Let $\stackrel{d}{\to}$ denote convergence in distribution. When necessary, convergence in distribution is understood in the sense of Hoffmann-J{\o}rgensen (cf. Chapter 1 in \cite{van1996weak}).
For the product of two metric (or normed) spaces, we always consider a product metric (or norm).

\begin{theorem}[Limit theorem for EOT potentials]
\label{thm: limit theorem EOT potential}
Let $s \in \NN$
and $\bm{\mu} = (\mu^1,\mu^2) \in \calP(\calX) \times \calP(\calX)$ be arbitrary.
Then, for $\hat{\bm{\varphi}}_n = \bm{\varphi}^{\hat{\bm{\mu}}_n}$, we have as $n \to \infty$,
\[
\sqrt{n}\big(\hat{\bm{\varphi}}_n - \bm{\varphi}^{\bm{\mu}}\big) \stackrel{d}{\to} \bm{G}^{\bm{\mu}} \quad \text{in} \ \calC^s (\calX) \times \calC^s (\calX),
\]
where $\bm{G}^{\bm{\mu}} = (G_1^{\bm{\mu}}, G_2^{\bm{\mu}})$ is a zero-mean Gaussian random variable with values in 
$\calC^s (\calX) \times \calC^s (\calX)$.
\end{theorem}

\cref{thm: limit theorem EOT potential} immediately implies the following corollary concerning the limit distribution of the empirical EOT map. 
For $s \in \NN_0$, let
$\calC^{s}(\calX;\R^d)$
denote the space of vector-valued functions 
$f=(f_1,\dots,f_d): \calX \to \R^d$
whose coordinate functions belong to 
$\calC^s(\calX)$,
equipped with the norm 
$\| f \|_{\calC^{s}(\calX;\R^d)} = \sum_{j=1}^d \| f_j \|_{\calC^s(\calX)}$.
\begin{corollary}[Limit theorem for EOT map]
\label{cor: limit theorem EOT map}
Let $s \in \NN$
and $\bm{\mu} = (\mu^1,\mu^2) \in \calP(\calX) \times \calP(\calX)$ be arbitrary.
Consider the quadratic cost $c(x_1,x_2) = \| x_1-x_2 \|^2/2$ and the EOT map $T^{\bm{\mu}}$ given in \eqref{eq: EOT map}. Then, for $\hat{T}_n = T^{\hat{\bm{\mu}}_n}$, we have as $n \to \infty$,
\[
\sqrt{n} \big (\hat{T}_n - T^{\bm{\mu}} \big) \stackrel{d}{\to} -\nabla G_{1}^{\bm{\mu}}  \quad \text{in} \ \calC^{s-1} (\calX; \R^d). 
\]
The limit $-\nabla G_{1}^{\bm{\mu}}$ is a zero-mean Gaussian random variable in 
$\calC^{s-1} (\calX; \R^d)$.
\end{corollary}

The recent work of \cite{delbarrio22EOT} shows that 
$\E[\| \hat{\varphi}^n_1 - \varphi_1^{\bm{\mu}}\|_{\calC^s(\calX)}^2] \vee \E[\| \hat{\varphi}^n_2 - \varphi_2^{\bm{\mu}}\|_{\calC^s(\calX)}^2] = O(n^{-1})$.
\cref{thm: limit theorem EOT potential} complements their result by further showing distributional convergence of $\sqrt{n}(\hat{\bm{\varphi}}_n - \bm{\varphi}^{\bm{\mu}})$ in 
$\calC^s (\calX) \times \calC^s (\calX)$.

The proof of \cref{thm: limit theorem EOT potential} employs  Hadamard differentiability of the map $\bm{\mu} \mapsto \bm{\varphi}^{\bm{\mu}}$ in 
$\calC^s (\calX) \times \calC^s (\calX)$
(stated in \cref{thm: H derivative potential} ahead) and the functional delta method; see \cref{sec: functional delta} for a review of Hadamard differentiability and the functional delta method. To this effect, we embed 
$\calP(\calX)$
into 
$\ell^\infty (B^s)$, where $B^s$
is the unit ball in 
$\calC^s(\calX)$.
Since 
$B^s$
is $\mu^i$-Donsker 
for $i=1,2$
when $s > d/2$ (cf. Theorem 2.7.1 in \cite{van1996weak} 
), the conclusion of \cref{thm: limit theorem EOT potential} follows from the functional delta method. The case~of $s\le d/2$ follows by noting that the inclusion map 
$f \mapsto f, \calC^s(\calX) \to \calC^{s'}(\calX)$,
with $s ' < s$ is continuous. Having \cref{thm: limit theorem EOT potential}, \cref{cor: limit theorem EOT map} follows by \cref{lem: brenier} and the fact that the gradient $f \mapsto \nabla f$ is continuous from 
$\calC^s(\calX)$
into 
$\calC^{s-1}(\calX;\R^d)$.
The corollary can also be deduced directly from Hadamard differentiability of the map $\bm{\mu} \mapsto \nabla \varphi_1^{\bm{\mu}}$ in 
$\calC^{s-1}(\calX;\R^d)$.

\begin{remark}[One-sample case]
We have only presented the two-sample limit distribution results for the EOT potentials and map, but as evident from the proof strategy, analogous conclusions continue to hold for the one-sample case where either $\mu^1$ or $\mu^2$ is known (cf. \cref{ex: vector quantile}). 
\end{remark}

\begin{remark}[Measurability]
Since $(x_1,\dots,x_n) \mapsto n^{-1}\sum_{j=1}^n \delta_{x_j}, \calX^n\to\calP(\calX)$ 
is weakly continuous, in view of Lemma \ref{lem: EOT potential} (iii), $\hat{\bm{\varphi}}_n$ is a proper, 
$\calC^s (\calX) \times \calC^s (\calX)$-valued
random variable. Likewise, the empirical EOT map $\hat{T}_n$ is a proper, 
$\calC^{s-1} (\calX;\R^d)$-valued
random variable.
\end{remark}

\begin{remark}[Higher-order fluctuations]
More can be said about higher-order fluctuations of the empirical EOT potentials. In \cref{thm: second order H derivative}, we will establish second-order Hadamard differentiability of the EOT potentials, which implies that 
\[
n \big(\hat{\bm{\varphi}}_n - \bm{\varphi}^{\bm{\mu}} - [\bm{\varphi}^{\bm{\mu}}]' (\hat{\bm{\mu}}_n - \bm{\mu}) \big)
\]
converges in distribution in 
$\calC(\calX) \times \calC(\calX)$,
where $[\bm{\varphi}^{\bm{\mu}}]'$ is the first Hadamard derivative at $\bm{\mu}$ (cf. \cref{thm: H derivative potential}). 
\end{remark}

\begin{remark}[Comparison with \cite{gunsilius2021matching}]
Theorem 1 in the latest update of  \cite{gunsilius2021matching} (updated on July 9, 2022 on arXiv)\footnote{The initial version stated a weak convergence result in $L^\infty(\mu^1) \times L^\infty(\mu^2)$.} states a limit distribution result for the EOT potentials in $\calC(S_1) \times \calC(S_2)$, where $S_i:=\supp(\mu^i)$ is a compact convex set (in fact \cite{gunsilius2021matching} consider the multimarginal setting, but we focus our discussion on the two-marginal case). Their proof differs from ours in that they do not derive Hadamard differentiability of EOT potentials (nor does the proof contain  Hadamard differentiability results). Importantly, the Hadamard differentiability result is stronger than just deriving a limit distribution, as it automatically yields bootstrap consistency and asymptotic efficiency of the empirical estimators; see \cref{sec: efficiency} for further discussion. The question of asymptotic efficiency is not accounted for in \cite{gunsilius2021matching}.
Furthermore, Hadamard differentiability of the EOT potentials in 
$\calC^s (\calX) \times \calC^s(\calX)$
plays a crucial role in deriving the null limit distribution of the empirical Sinkhorn divergence that involves the \textit{second-order} Hadamard derivative of the EOT potentials in 
$\calC(\calX) \times \calC(\calX)$.
\end{remark}

\subsection{Limit theorems for Sinkhorn divergence}
\label{sec: main results Sinkhorn}
The second main result concerns limit distributions for the empirical Sinkhorn divergence. 
We first state an asymptotic normality result for the empirical Sinkhorn divergence. 

\begin{proposition}[Limit theorem for Sinkhorn divergence]
\label{prop: limit theorem Sinkhorn}
For every $\bm{\mu}= (\mu^1,\mu^2) \in \calP(\calX) \times \calP(\calX)$, we have
\[
\sqrt{n}\big ( \bar{\mathsf{S}}_{c,\varepsilon} (\hat{\mu}_n^1,\hat{\mu}_n^2) - \bar{\mathsf{S}}_{c,\varepsilon} (\mu^1,\mu^2) \big )\stackrel{d}{\to} N ( 0, \sigma_{\bm{\mu}}^2), \quad \text{as} \ n \to \infty, 
\]
where $\sigma_{\bm{\mu}}^2 = \Var_{\mu^1} \big(\varphi_1^{\bm{\mu}}-\varphi_1^{(\mu^1,\mu^1)}\big) + \Var_{\mu^2} \big(\varphi_2^{\bm{\mu}}-\varphi_2^{(\mu^2,\mu^2)}\big)$. Furthermore, the asymptotic variance $\sigma_{\bm{\mu}}^2$ is strictly positive whenever $ \bar{\mathsf{S}}_{c,\varepsilon} (\mu^1,\mu^2) \ne 0$ and $\supp(\mu^1) \cap \supp (\mu^2) \ne \varnothing$.
\end{proposition}

\begin{remark}
Since $(\mu^1,\mu^2) \mapsto \bar{\mathsf{S}}_{c,\varepsilon} (\mu^1,\mu^2)$ is weakly continuous (this follows by \cref{lem: EOT potential} and the duality formula), $\bar{\mathsf{S}}_{c,\varepsilon} (\hat{\mu}_n^1,\hat{\mu}_n^2)$ is a proper random variable. 
\end{remark}

The first claim of this proposition follows from relatively minor modifications to the proof of Theorem~7 in \cite{goldfeld2022statistical} that establishes asymptotic normality of the empirical EOT for the quadratic cost (or \cite{delbarrio22EOT}).
The second claim provides conditions under which the limiting Gaussian distribution is nondegenerate. If $\bar{\mathsf{S}}_{c,\varepsilon}$ is a valid divergence (e.g., when the cost is quadratic), then $\bar{\mathsf{S}}_{c,\varepsilon}(\mu^1,\mu^2)\ne 0$ if and only if $\mu^1 \ne \mu^2$. The assumption that $\supp(\mu^1) \cap \supp (\mu^2) \ne \varnothing$ can not be dropped in general. Indeed, if $\mu^1$ and $\mu^2$ are point masses at distinct points, we have $\sigma_{\bm{\mu}}^2 = 0$ but $\mu^1 \ne \mu^2$. 

In \cref{prop: limit theorem Sinkhorn}, when $\mu^1=\mu^2$, we have $\sigma_{\bm{\mu}}^2 = 0$, which entails $\sqrt{n}\bar{\mathsf{S}}_{c,\varepsilon} (\hat{\mu}_n^1,\hat{\mu}_n^2) \to 0$ in probability. Indeed, \cite{delbarrio22EOT} show that $\E[\bar{\mathsf{S}}_{c,\varepsilon} (\hat{\mu}_n^1,\hat{\mu}_n^2)] = O(n^{-1})$ under $\mu^1 = \mu^2$ for the quadratic cost, which implies that $n\bar{\mathsf{S}}_{c,\varepsilon} (\hat{\mu}_n^1,\hat{\mu}_n^2)$ is uniformly tight. The next theorem shows that, when $\mu^1 = \mu^2$, $n\bar{\mathsf{S}}_{c,\varepsilon} (\hat{\mu}_n^1,\hat{\mu}_n^2)$  in fact converges in distribution to some random variable, thereby determining a more precise random fluctuation of the empirical Sinkhorn divergence under the null. 

\begin{theorem}[Limit theorem for Sinkhorn divergence under null]
\label{thm: limit theorem Sinkhorn null}
Suppose that $\mu^1=\mu^2 = \mu\in\calP(\calX).$ 
Then, 
$n \bar{\mathsf{S}}_{c,\varepsilon} (\hat{\mu}_n^1,\hat{\mu}_n^2) \stackrel{d}{\to} \chi_{\mu}$ as $n \to \infty$
for some random variable $\chi_\mu$.
Furthermore, assuming nonnegativity of $\bar{\mathsf{S}}_{c,\varepsilon}$, the support of $\chi_{\mu}$ agrees with $[0,\infty)$, unless $\chi_{\mu} = 0$ a.s.
\end{theorem}

 The proof of \cref{thm: limit theorem Sinkhorn null} is significantly more involved than that of \cref{prop: limit theorem Sinkhorn} and relies on the second-order functional delta method. To this end, we establish second-order Hadamard differentiability of the map $(\nu^1,\nu^2) \mapsto \bar{\mathsf{S}}_{c,\varepsilon} (\nu^1,\nu^2)$ at $(\nu^1,\nu^2) = (\mu,\mu)$, which in turn involves the second-order Hadamard derivative of the EOT potentials.
 The limit variable is given by a nonlinear functional of a certain Gaussian process, but it seems highly nontrivial to derive an explicit expression of the limit distribution from our derivation.\footnote{The concurrent work \cite{gonzalez2022weak} derives an explicit expression of the Sinkhorn null limit, albeit with a different technique.} Still, the limit distribution can be consistently estimated by the (two-sample version of) subsampling or the $m$-out-of-$n$ bootstrap \cite{politis1994large,bickel1997resampling}. See Appendix \ref{sec: subsampling} for details. The proof of the consistency of the $m$-out-of-$n$ bootstrap again relies on second-order Hadamard differentiability of the Sinkhorn divergence w.r.t. the marginals.

The second claim of the theorem shows that 
the limit variable $\chi_\mu$ is nondegenerate in most cases. Indeed, 
assuming that either $\mu^1$ or $\mu^2$ is not a point mass, in view of \cref{lem: support Sinkhorn} below, $\chi_{\mu} = 0$ a.s. if and only if the the second Hadamard derivative of $\bar{\mathsf{S}}_{c,\varepsilon}$ at $(\mu,\mu)$ is identically zero. Otherwise, the support of $\chi_\mu$ agrees with $[0,\infty)$. 

\begin{remark}[One-sample case]
Analogous results hold for the one-sample case. For example, when $\mu^2$ is known, then $\sqrt{n} ( \bar{\mathsf{S}}_{c,\varepsilon} (\hat{\mu}_n^1,\mu^2) - \bar{\mathsf{S}}_{c,\varepsilon} (\mu^1,\mu^2)  ) \stackrel{d}{\to} N(0,\Var_{\mu^1} (\varphi_1^{\bm{\mu}}-\varphi_1^{(\mu^1,\mu^1)}))$ under the setting of \cref{prop: limit theorem Sinkhorn}, while $n \bar{\mathsf{S}}_{c,\varepsilon} (\hat{\mu}_n^1,\mu)$ converges in distribution to some random variable under the setting of \cref{thm: limit theorem Sinkhorn null}.
\end{remark}

\begin{remark}[Higher-order fluctuations]
The proof of \cref{thm: limit theorem Sinkhorn null} reveals that in general (i.e., $\mu^1 \ne \mu^2$ is allowed), the following stochastic expansion holds:
\[
\sqrt{n}\big ( \bar{\mathsf{S}}_{c,\varepsilon} (\hat{\mu}_n^1,\hat{\mu}_n^2) - \bar{\mathsf{S}}_{c,\varepsilon} (\mu^1,\mu^2) \big ) = \sum_{i=1}^2 \sqrt{n}(\hat{\mu}_n^i -\mu^i) \big(\varphi_i^{\bm{\mu}}-\varphi_i^{(\mu^i,\mu^i)}\big) + n^{-1/2}r_n,
 \]
where $r_n$ converges in distribution as $n \to \infty$. This expansion characterizes more precise random fluctuations of $\bar{\mathsf{S}}_{c,\varepsilon} (\hat{\mu}_n^1,\hat{\mu}_n^2)$; similar expansions hold for the (uncentered) EOT cost $\mathsf{S}_{c,\varepsilon}$. 
\end{remark}

\begin{remark}[Comparison with \cite{bigot2019central}]
A version of \cref{thm: limit theorem Sinkhorn null} was derived in \cite{bigot2019central} when the (common) distribution $\mu$ is finitely discrete, where the limit is given by a weighted sum of independent $\chi^2_1$-random variables. When $\mu$ is finitely discrete, it may be parameterized by a finite-dimensional simplex vector. Using this parameterization, \cite{bigot2019central} directly computed the Hessian matrix of the Sinkhorn divergence w.r.t. the simplex vectors and apply the second-order delta method. Clearly, this proof technique is restricted to the finitely discrete case and does not directly extend to the general case of \cref{thm: limit theorem Sinkhorn null}. Indeed, the major challenge in the proof of \cref{thm: limit theorem Sinkhorn null} stems from the fact that in the general case, the problem is inherently infinite-dimensional and requires delicate functional analytic arguments. This is accounted for by the second-order Hadamard differentiability result of the Sinkhorn divergence, stated in \cref{thm: second H derivative Sinkhorn}.
\end{remark}

\subsubsection{Independence testing with Sinkhorn divergence}
Let $(V_i,W_i) \in \R^{d_1} \times \R^{d_2}, i =1,\dots, n$ be i.i.d. with common distribution $\pi$. Set $d=d_1+d_2$ and $X_i = (V_i,W_i) \in \R^d$. Let $\pi^V$ and $\pi^W$ denote the marginal distributions of $V_i$ and $W_i$, respectively. Assume 
that $\pi^V$ and $\pi^W$ are compactly supported 
and let $\calX\subset \R^{d}$ be a closed ball containing $\supp(\pi^V)\times \supp(\pi^W)$.
Consider independence testing 
\[
H_0: \pi = \pi^V \otimes \pi^W \quad \text{vs}. \quad H_1: \pi \ne \pi^V \otimes \pi^W. 
\]
Motivated by computational considerations, \cite{liu2022entropy} proposed a test based on the Sinkhorn divergence that compares the empirical distribution of $X_i = (V_i,W_i)$ with the product of the marginal empirical distributions of $V_i$ and $W_i$. \cite{liu2022entropy} focused on the quadratic cost, but we allow here a general smooth cost. Specifically, \cite{liu2022entropy} proposed a test that rejects the null for large values of the statistic
\[
D_n = \bar{\mathsf{S}}_{c,\varepsilon}(\hat{\pi}_n,\hat{\pi}_n^V \otimes \hat{\pi}_n^W),
\]
where $\hat{\pi}_n = n^{-1}\sum_{i=1}^n\delta_{X_i}, \hat{\pi}_n^V = n^{-1}\sum_{i=1}^n \delta_{V_i}$, and $\hat{\pi}_n^W = n^{-1}\sum_{i=1}^n \delta_{W_i}$. The rationale behind this test is as follows. By Varadarajan's theorem, it holds that $\hat{\pi}_n \to \pi$ and $\hat{\pi}_n^V \otimes \hat{\pi}_n^W \to \pi^V \otimes \pi^W$ weakly a.s., so by \cref{lem: EOT potential} (iii) and duality, we have $D_n \to \bar{\mathsf{S}}_{c,\varepsilon}(\pi,\pi^V \otimes \pi^W)$ a.s. At least for the quadratic cost, $\bar{\mathsf{S}}_{c,\varepsilon}(\pi,\pi^V \otimes \pi^W) = 0$ if and only if $\pi = \pi^V \otimes \pi^W$, so it is reasonable to reject $H_0$ when $D_n$ is large. 

\cite{liu2022entropy} suggested a critical value of order $n^{-1/2}$ derived from a finite sample deviation inequality for $D_n$. However, under
$H_0$, both $\hat{\pi}_n$ and $\hat{\pi}_n^V \otimes \hat{\pi}_n^W $ converge to the same limit, so \cref{thm: limit theorem Sinkhorn null} suggests that the correct order of $D_n$ under the null should be $n^{-1}$. The next proposition confirms this, thereby determining the precise rate for $D_n$ under the null. 

\begin{proposition}[Null limit of Sinkhorn independence test]
\label{prop: independence test}
Consider the setting as stated above.  Then, under the null $H_0$, we have $n D_n \stackrel{d}{\to} \aleph_{\pi}$ as $n \to \infty$ for some random variable $\aleph_{\pi}$. Furthermore, assuming nonnegativity of $\bar{\mathsf{S}}_{c,\varepsilon}$, the support of $\aleph_{\pi}$ agrees with $[0,\infty)$, unless $\aleph_{\pi} = 0$ a.s.
\end{proposition}

Note that \cref{prop: independence test} does not immediately follow from \cref{thm: limit theorem Sinkhorn null} since $\hat{\pi}_n^V \otimes \hat{\pi}_n^W$ is not an empirical process but a two-sample $V$-process \cite{PeGi1999}, and $\hat{\pi}_n$ and $\hat{\pi}_n^V \otimes \hat{\pi}_n^W$ are dependent. The proof first finds a joint limit distribution of $\sqrt{n}(\hat{\pi}_n-\pi)$ and $\sqrt{n}(\hat{\pi}_n^V \otimes \hat{\pi}_n^W - \pi)$ in $\ell^\infty (B^s) \times \ell^\infty (B^s)$ with $s > 2d$ using techniques from $U$-processes \cite{chen2020jackknife}, and then applies the second-order functional delta method. Additionally, as in \cref{thm: limit theorem Sinkhorn null}, unless $\pi$ degenerates to a point mass or the second derivative of $\bar{\mathsf{S}}_{c,\varepsilon}$ at $(\pi,\pi)$ with $\pi = \pi^V \otimes \pi^W$ is identically zero, the support of the limit variable $\aleph_\pi$ agrees with $[0,\infty)$, yielding nondegeneracy of the limit law.

\section{Differentiability of EOT potentials  and Sinkhorn divergence}
\label{sec: EOT potential}
As already stated, the main ingredients of the proofs of the results in the preceding section are the first and second-order Hadamard differentiability results for the EOT potentials and Sinkhorn divergence. The present section collects those results.

\subsection{Hadamard differentiability of EOT potentials}
\label{sec: H derivative of potential}

Our goal is to establish Hadamard differentiability of the map $\bm{\mu} \mapsto \bm{\varphi}^{\bm{\mu}}$. 
Since $
\calP(\calX) \times \calP(\calX)
$ is \textit{a priori} not a vector space, we embed the preceding map into a normed space as follows. Note that $\bm{\varphi}^{\bm{\mu}}$ is fully characterized as the solution of \eqref{eq: FOC} (subject to the normalization $\varphi_1(x_1^\circ)=\varphi_2(x_2^\circ)$),  and whenever $\varphi_i \in \calC^s(\calX)$
for $i=1,2$, we have
\[
\left(e^{\frac{\varphi_1(\cdot) + \varphi_2(x_2) - c(\cdot,x_2)}{\varepsilon}}\,,\,e^{\frac{\varphi_1(x_1) + \varphi_2(\cdot) - c(x_1,\cdot)}{\varepsilon}}\right) \in 
\calC^s(\calX)\times \calC^s(\calX),
\quad \forall (x_1,x_2) \in
\calX \times \calX,
\]
for $s \in \mathbb{N}$ arbitrary. With this in mind,  it is natural to think of $\mu^i$ as a functional on 
$\calC^s(\calX)$,
and we identify 
$\calP(\calX)\times \calP(\calX)$
as a subset of 
$\ell^\infty (B^s) \times \ell^\infty (B^s)$,
where 
\[
B^s= \big\{ f \in \calC^s (\calX) : \| f \|_{\calC^s(\calX)} \le 1 \big\}.
\]
Formally, defining $\tau : 
\calP(\calX)\times \calP(\calX)
\to 
\ell^\infty(B^s) \times \ell^\infty(B^s)
$
by $(\tau \bm{\mu})_i = (f \mapsto \int f \, d\mu^i)_{f\in B^s}$,
$i=1,2$, we identify the map $\bm{\nu} \mapsto \bm{\varphi}^{\bm{\nu}}$ with $\tau \bm{\nu} \mapsto \bm{\varphi}^{\bm{\nu}}$. Since $\tau$ is one-to-one (cf. \cref{lem: convergence determining}), the latter map is well-defined. 
Equip 
$\ell^\infty (B^s) \times \ell^\infty (B^s)$
with the norm 
$\| \gamma^1 \|_{\infty,B^s} \vee \| \gamma^2 \|_{\infty,B^s}$ for $\bm{\gamma} = (\gamma^1,\gamma^2) \in \ell^\infty (B^s) \times \ell^\infty (B^s)$.
Finally, for $\eta\in\calP(\calX)$, set
\[
\frakP_{\eta}=\big\{\nu\in\calP(\calX) : \supp(\nu)\subset\supp(\eta)\big\} \quad
    \text{and} \quad \calM_{\eta} = \big\{ t(\nu-\eta) : \nu \in \frakP_{\eta}, t > 0 \big\}.  
\]
Elements of $\calM_{\eta}$ are signed measures with total mass zero supported in $\supp(\eta)$. 
Observe that $\calM_{\eta} \subset \ell^\infty (B^s)$.
These definitions are an artifact of our proof technique.

\begin{theorem}[Hadamard differentiability of EOT potentials]
\label{thm: H derivative potential}
For every $s \in \mathbb{N}$ and $\bm{\mu}=(\mu^1,\mu^2)\in
\calP(\calX)\times \calP(\calX)
$, the map $\bm{\nu} \mapsto \bm{\varphi}^{\bm{\nu}}, 
\calP_{\mu^1} \times \calP_{\mu^2}
\subset 
\ell^\infty (B^s) \times \ell^\infty (B^s)
\to  
\calC^s(\calX) \times \calC^s(\calX)$
is Hadamard differentiable at $\bm{\mu}$ tangentially to 
$\overline{\calM_{\mu^1}}^{\ell^\infty(B^s)} \times \overline{\calM_{\mu^2}}^{\ell^\infty(B^s)}$.
\end{theorem}

The first step of the proof of \cref{thm: H derivative potential} is to establish Hadamard differentiability in 
$\calC(\calX) \times \calC(\calX)$
instead of 
$\calC^s(\calX) \times \calC^s(\calX)$.
To this effect, we regard $\bm{\varphi}^{\bm{\mu}}$ as a solution to the system of functional equations \eqref{eq: FOC} and use a version of the implicit function theorem to prove Hadamard differentiability of the map $\bm{\mu} \mapsto \bm{\varphi}^{\bm{\mu}}$ in 
$\calC(\calX) \times \calC(\calX)$.
Precisely, Hadamard differentiability of the EOT potentials is first established in $\calC(\supp(\mu^1))\times \calC(\supp(\mu^2))$, then the EOT potentials are extended to $\calX\times \calX$ via \eqref{eq: FOC} and differentiability in $\calC(\calX)\times \calC(\calX)$ follows readily.

To lift the Hadamard differentiability to 
$\calC^s(\calX) \times \calC^s(\calX)$,
we again use the expression from~\eqref{eq: FOC},
\[
e^{-\varphi_i^{\bm{\mu}}/\varepsilon} (x_i)=  \int e^{(\varphi_j^{\bm{\mu}}(x_j) - c(x_1,x_2))/\varepsilon} \, d\mu^j(x_j),  \quad i \ne j,
\]
and show that derivatives of $\varphi_i^{\bm{\mu}}$ are Hadamard differentiable in 
$\calC(\calX)$.
This argument is partly inspired by the proof of Theorem 4.5 in \cite{delbarrio22EOT}. Completeness of 
$\calC^s(\calX)$
then yields that the map $\bm{\mu} \mapsto \varphi_i^{\bm{\mu}}$ is Hadamard differentiable in 
$\calC^s(\calX)$.
See the proof in \cref{sec: proof of H derivative potential} for the full details.

\begin{remark}[Tangent cone]
\label{rem: tangent cone}
Since $\calP_{\mu^1} \times \calP_{\mu^2}$  
is convex, the set 
$\overline{\calM_{\mu^1}}^{\ell^\infty(B^s)} \times \overline{\calM_{\mu^2}}^{\ell^\infty(B^s)}$
agrees with the tangent cone to 
$\calP_{\mu^1} \times \calP_{\mu^2}$
at $\bm{\mu}$ in 
$\ell^\infty (B^s) \times \ell^\infty (B^s)$
(cf. Chapter 4 in \cite{aubin2009set}; see also \cref{sec: functional delta}). 
Each element of
$\overline{\calM_{\mu^i}}^{\ell^\infty(B^s)}$
extends uniquely to a bounded linear functional on 
$\calC^s(\calX)$
(cf. Lemma 1 in \cite{nickl2009convergence}). Hence, by the Riesz-Markov-Kakutani theorem, each 
$\gamma^i \in \overline{\calM_{\mu^i}}^{\ell^\infty(B^s)}$
corresponds to finite signed Borel measures $(\gamma^i_{k})_{k \in \NN_0^d,|k| \le s}$ on 
$\calX$ supported in $\supp(\mu^i)$
such that $\gamma^i(f) = \sum_{|k| \le s} \int D^k f \, d\gamma_k^i$ for 
$f \in \calC^s(\calX)$.
We remark that this action also makes sense for a function $f\in\calC(\supp(\mu^i))$ that merely admits some extension $\bar f\in\calC^s(\calX)$. Indeed, let $\gamma^i_n\in\calM_{\mu^i}$ converge to $\gamma^i$ in $\ell^{\infty}(B^s)$, then $\gamma^i_n(f)=\gamma^i_n(\bar f)\to\gamma^i(\bar f)$ as $\supp(\gamma^i_n)\subset \supp(\mu^i)$; the first equality also shows that the limit is independent of the choice of extension, hence we may define $\gamma^i(f):=\gamma^i(\bar f)$.
Abusing notation, we often denote the action of $\gamma^i$ on $f\in\calC(\supp(\mu^i))$ admitting a $\calC^s$ extension to $\calX$ or
$f \in \calC^s(\calX)$
as $\gamma^i (f) = \int f \, d\gamma^i$. 
\end{remark}

\begin{remark}[Functional delta method]
Recalling that $s$ in \cref{lem: EOT potential} (ii) is arbitrary, if we choose $s > d/2$, then 
$B^s$
is $\mu^i$-Donsker for each $i=1,2$ (cf. Theorem 2.7.1 in \cite{van1996weak}
). We use this fact in the proof of the limit theorem for EOT potentials from \cref{thm: limit theorem EOT potential}. Since the support of a tight $\mu^i$-Brownian bridge in 
$\ell^\infty(B^s)$
is contained in 
$\overline{\calM_{\mu^i}}^{\ell^\infty(B^s)}$
(see \cref{lem: support}), the functional delta method (see \cref{lem: functional delta method}) immediately applies to the map $\bm{\mu} \mapsto \bm{\varphi}^{\bm{\mu}}$. Similar comments apply to other Hadamard differentiability results. 
\end{remark}

The derivation of the null limit distribution of the empirical Sinkhorn divergence involves the second-order Hadamard derivative of EOT potentials, which is given~next. 

\begin{theorem}[Second-order Hadamard differentiability of EOT potentials]
\label{thm: second order H derivative}
For every $s \in \mathbb{N}$ and $\bm{\mu}=(\mu^1,\mu^2)\in
\calP(\calX)\times \calP(\calX)
$, there exists a continuous map 
$[\bm{\varphi}^{\bm{\mu}}]'': 
\overline{\calM_{\mu^1}}^{\ell^\infty(B^s)} \times \overline{\calM_{\mu^2}}^{\ell^\infty(B^s)} \to \calC(\calX) \times \calC(\calX)$
such that for every sequence $(\bm{\mu}_t)_{t > 0} \subset  
\calP_{\mu^1} \times \calP_{\mu^2}$
with $\bm{\gamma}_t := t^{-1}(\bm{\mu}_t-\bm{\mu}) \to \bm{\gamma}$ in $\ell^\infty (B^s) \times \ell^\infty (B^s)$
as $t \downarrow 0$, we have 
\[
\frac{\bm{\varphi}^{\bm{\mu}_t}-\bm{\varphi}^{\bm{\mu}} - t[\bm{\varphi}^{\bm{\mu}}]'(\bm{\gamma}_t)}{t^2/2} \to [\bm{\varphi}^{\bm{\mu}}]''(\bm{\gamma}) \quad \text{in} \ 
\calC(\calX)\times \calC(\calX).
\]
The map $[\bm{\varphi}^{\bm{\mu}}]''$ is positively homogeneous of degree $2$, i.e., $[\bm{\varphi}^{\bm{\mu}}]''(t\bm{\gamma}) = t^2[\bm{\varphi}^{\bm{\mu}}]''(\bm{\gamma})$ for every $t > 0$ and $\bm{\gamma} \in \overline{\calM_{\mu^1}}^{\ell^\infty(B^s)} \times \overline{\calM_{\mu^2}}^{\ell^\infty(B^s)}$. 
\end{theorem}

To prove this theorem, we extend Lemma 3.9.34 in \cite{van1996weak} to the second-order Hadamard derivative, which is presented in \cref{lem: second order H derivative}. It is worth noting that, while the second-order Hadamard differentiability result is stated in terms of 
$\calC(\calX) \times \calC(\calX)$,
its proof requires the (first-order) Hadamard derivative in 
$\calC^s(\calX) \times \calC^s(\calX)$
to verify Condition (iv) in \cref{lem: second order H derivative}. See the proof in \cref{sec: proof of second order H derivative} for details.

\subsection{Hadamard differentiability of Sinkhorn divergence}

In this section, we study Hadamard derivatives of the Sinkhorn divergence. 
The following lemma follows relatively easily from the proof of Theorem 7 in \cite{goldfeld2022statistical}.
\begin{lemma}[Hadamard derivative of Sinkhorn divergence]
\label{lem: H derivative Sinkhorn}
For every $s \in \mathbb{N}$ and $\bm{\mu} = (\mu^1,\mu^2) \in \calP(\calX) \times \calP(\calX)$,  the map $\bm{\nu} = (\nu^1,\nu^2) \mapsto \bar{\mathsf{S}}_{c,\varepsilon}(\nu^1,\nu^2), 
\calP_{\mu^1} \times \calP_{\mu^2}
\subset \ell^\infty(B^s) \times \ell^\infty(B^s) \to \R$ is Hadamard differentiable at $\bm{\mu}$  tangentially to $\overline{\calM_{\mu^1}}^{\ell^\infty(B^s)} \times \overline{\calM_{\mu^2}}^{\ell^\infty(B^s)}$ with derivative 
\[
\big[\bar{\mathsf{S}}^{\bm{\mu}}_{c,\varepsilon}\big]'(\bm{\gamma}) = \int \big( \varphi_1^{\bm{\mu}} - \varphi_1^{(\mu^1,\mu^1)} \big) \, d\gamma^1 + \int \big ( \varphi_2^{\bm{\mu}} - \varphi_2^{(\mu^2,\mu^2)} \big) \, d\gamma^2
\]
for $\bm{\gamma} =(\gamma^1,\gamma^2) \in \overline{\calM_{\mu^1}}^{\ell^\infty(B^s)} \times \overline{\calM_{\mu^2}}^{\ell^\infty(B^s)}$. 
\end{lemma}

When $\mu^1=\mu^2$, we have $\big[\bar{\mathsf{S}}^{\bm{\mu}}_{c,\varepsilon}\big]'(\bm{\gamma}) = 0$. Therefore, to explore the limit distribution of the empirical Sinkhorn divergence under the null $\mu^1=\mu^2$, we need to look at the second-order Hadamard derivative of the Sinkhorn divergence, which is given next. 

\begin{theorem}[Second-order Hadamard derivative of Sinkhorn divergence]
\label{thm: second H derivative Sinkhorn}
For every $s \in \mathbb{N}$ and $\bm{\mu} = (\mu^1,\mu^2) \in \calP(\calX) \times \calP(\calX)$, 
there exists a continuous functional $\Delta_\mu: \overline{\calM_{\mu}}^{\ell^\infty(B^s)} \times \overline{\calM_{\mu}}^{\ell^\infty(B^s)} \to \R$ such that for every sequence $(\mu_t^1,\mu_t^2) \in \frakP_{\mu}\times \frakP_{\mu}$ 
with $t^{-1}(\mu_t^i-\mu) \to \gamma^i$ in $\ell^\infty(B^s)$ as $t \downarrow 0$ for $i=1,2$, we have
\[
\frac{\bar{\mathsf{S}}_{c,\varepsilon}(\mu_t^1,\mu_t^2)}{t^2/2} \to \Delta_\mu(\bm{\gamma})
\]
with $\bm{\gamma} = (\gamma^1,\gamma^2)$. 
The functional $\Delta_\mu$ is positively homogeneous of degree $2$.
\end{theorem}

Given the first and second-order Hadamard differentiability results for the EOT potentials, the proof of \cref{thm: second H derivative Sinkhorn} is reasonably straightforward. Indeed, the proof consists of expanding $\varphi_i^{(\mu_t^1,\mu_t^2)}$ and $\varphi_i^{(\mu_t^i,\mu_t^i)}$ up to the second-order and plugging these expansions into the dual expression \eqref{eq: Sinkhorn dual} of $\bar{\mathsf{S}}_{c,\varepsilon}(\mu_t^1,\mu_t^2)$. 

\section{Bootstrap consistency and asymptotic efficiency}
\label{sec: efficiency}
As discussed before, our limit theorems in \cref{thm: limit theorem EOT potential}, \cref{cor: limit theorem EOT map}, and \cref{prop: limit theorem Sinkhorn} follow by establishing Hadamard differentiability of the relevant maps. Importantly, Hadamard differentiability results automatically also yield bootstrap consistency and asymptotic efficiency (cf. Chapter 3.11 in \cite{van1996weak} and Chapter 25 in \cite{vanderVaart1998asymptotic}) of the empirical estimators of the EOT potentials, map, and Sinkhorn divergence (with $\mu^1 \ne \mu^2$). To simplify our discussion, we focus here on estimating the EOT map for the quadratic cost with known $\mu^1$ (i.e., the one-sample case); other cases are similar. The setting of known $\mu^1$ is motivated by the connection of the EOT map to the vector quantile function (\cref{ex: vector quantile}).  Consider the setting of \cref{cor: limit theorem EOT map}. 

\subsection{Bootstrap consistency}

Let $\tilde{T}_n = T^{(\mu_n^{1},\hat{\mu}_n^{2})}$ be the empirical EOT map, and $\tilde{T}_n^{B}$ be its bootstrap analog, i.e., $\tilde{T}_n^B = T^{(\mu_n^{1},\hat{\mu}_n^{2,B})}$, where $\hat{\mu}_n^{2,B}$ is the empirical distribution of a bootstrap sample from $\hat{\mu}_n^2$ of size $n$ (cf. Chapter 3.6 in \cite{van1996weak} and Chapter 23 in \cite{vanderVaart1998asymptotic}). 

Pick any $s,s' \in \NN_0$ with $s' < s$ and $s > d/2$.
\cref{lem: brenier} and \cref{thm: H derivative potential} yield that the map 
$\delta: \nu^2 \mapsto T^{(\mu^1,\nu^2)},
\frakP_{\mu^2} \subset \ell^\infty(B^s) \to \calC^{s'}(\calX; \R^d)$
is Hadamard differentiable at $\nu^2=\mu^2$ tangentially to 
$\overline{\calM_{\mu^2}}^{\ell^\infty(B^s)}$
with derivative $\delta_{\mu^2}'(\gamma^2) = -[\nabla \varphi_1^{\bm{\mu}}]'(0,\gamma^2)$. The tangent cone 
$\overline{\calM_{\mu^2}}^{\ell^\infty(B^s)}$
contains a vector subspace on which  a tight $\mu^2$-Brownian bridge $\GG_2^{\mu^2}$ in 
$\ell^\infty(B^s)$
concentrates (cf. \cref{lem: support}; note that 
$B^s$
with $s > d/2$ is $\mu^2$-Donsker). Hence, by the functional delta method, we have 
\[
\sqrt{n}(\tilde{T}_n - T^{\bm{\mu}}) \stackrel{d}{\to} \delta_{\mu^2}'(\GG_2^{\mu^2}) \quad  \text{in} \quad \calC^{s'}(\calX;\R^d).
\] 
Furthermore, by Theorems 3.6.1 and 3.9.11 in \cite{van1996weak},  the following bootstrap consistency holds:
\begin{equation}
\sup_{h \in \mathrm{BL}_1(\calC^{s'}(
\calX
;\R^d))}\big|\E^B\big[h\big(\sqrt{n}(\tilde{T}_n^B-\tilde{T}_n) \big)\big]  - \E\big[h\big ( \delta_{\mu^2}'(\GG_2^{\mu^2})\big) \big]  \big | \to 0
\label{eq: bootstrap consistency}
\end{equation}
in outer probability, where $\E^B$ denotes the conditional expectation given the sample and $\mathrm{BL}_1(\calC^{s'}(
\calX
;\R^d))$ is the class of $1$-Lipschitz functions $h: \calC^{s'}(
\calX
;\R^d) \to [-1,1]$. 

\begin{example}[Confidence bands for EOT map]
Consider constructing confidence bands for $T_j^{\bm{\mu}}$ with $T^{\bm{\mu}} = (T_1^{\bm{\mu}},\dots,T_d^{\bm{\mu}})$. 
The continuous mapping theorem yields that  $\| \sqrt{n}(\tilde{T}_{n,j} - T_j^{\bm{\mu}}) \|_{\infty,
\calX
} \stackrel{d}{\to}  \| [\delta_{\mu^2}'(\GG_2^{\mu^2})]_j\|_{\infty,
\calX
}$. 
For a given $\alpha \in (0,1)$, let $\hat{q}_{1-\alpha}$ denote the conditional $(1-\alpha)$-quantile of $\| \sqrt{n}(\tilde{T}_{n,j}^B - \tilde{T}_{n,j}) \|_{\infty,
\calX
}$ given the data.
The bootstrap consistency result \eqref{eq: bootstrap consistency} yields that $\{ [\tilde{T}_j^n (x_1) \pm \hat{q}_{1-\alpha}/\sqrt{n}] : x_1 \in 
\calX
\}$ is a valid confidence band for $T_j^{\bm{\mu}}$, i.e., $\Prob \big(T_j^{\bm{\mu}}(x_1) \in [\tilde{T}_j^n(x_1) \pm \hat{q}_{1-\alpha}/\sqrt{n}] \ \text{for all} \ x_1 \in 
\calX
\big) \to 1-\alpha$.
\end{example}

\begin{remark}[Statistical inference for unregularized Brenier maps]
When the source measure $\mu^1$ is absolutely continuous, Brenier's theorem \cite{brenier1991polar} guarantees the existence of the $\mu^1$-a.e. unique map  (called the Brenier map) given by the gradient of a convex function transporting $\mu^1$ onto $\mu^2$. The EOT map approximates the (unregularized) Brenier map as $\varepsilon \to 0$ (cf. \cite{mikami2004monge,carlier2022convergence}). However, all the limit theorems in the present paper crucially rely on the fact that the regularization parameter $\varepsilon>0$ is fixed and do not directly extend to the case where $\varepsilon = \varepsilon_n \to 0$. Indeed, estimation of the unregularized Brenier map suffers from the curse of dimensionality \cite{hutter2021minimax}, so $\sqrt{n}$-consistency toward the unregularized Brenier map does not hold in general for $\varepsilon = \varepsilon_n \to 0$. One exception is the semidiscrete case where $\mu^1$ is absolutely continuous and $\mu^2$ is finitely discrete, for which \cite{pooladian2023minimax} establish the parametric convergence rate of empirical EOT maps with vanishing regularization parameters under the $L^2$-loss. For the semidiscrete case, the recent preprint \cite{sadhu2023limit} by a subset of the authors studies statistical inference for the empirical (unregularized) Brenier map. Beyond the semidiscrete case, however, the problem of developing valid inferential methods for Brenier maps remains largely open. 
\end{remark}

\subsection{Asymptotic efficiency}

Regarding asymptotic efficiency, we follow Chapter~3.11 in \cite{van1996weak}.
Consider the setting of \cref{cor: limit theorem EOT map} and fix $(\mu^1,\mu^2) \in \calP(\calX) \times \calP(\calX)$. As before, let $s,s' \in \NN_0$ be such that $s'<s$ and $s > d/2$. For notational convenience, set $\BB = \calC^{s'}(
\calX
; \R^d)$. By linearity of the derivative, the limit variable $\delta_{\mu^2}'(\GG_2^{\mu^2})$ is zero-mean Gaussian in $\BB$. 

To apply the results of Chapter 3.11 in \cite{van1996weak}, we need to specify statistical experiments $(\mathsf{X}_n,\calA_n,P_{n,h}: h \in H)$ indexed by a vector subspace $H$ of a Hilbert space and local parameters $\kappa_n (h)$. Choose $H$ to be the set of bounded measurable functions on 
$\calX$
with $\mu^2$-mean zero equipped with the $L^2(\mu^2)$-inner product, and set $\mathsf{X}_n = 
\calX^n,
\calA_n=$ (Borel $\sigma$-field on 
$\calX^n$
), $P_{n,h} = (\mu_{n,h}^2)^{\otimes n}$ with $d\mu_{n,h}^2 = (1+h/\sqrt{n})d\mu^2$, and $\kappa_n (h) = \delta(\mu_{n,h}^2) = T^{(\mu^1,\mu_{n,h}^2)}: H \to \BB$. Note that for each $h \in H$, $\mu_{n,h}^2$ is a valid probability measure for sufficiently large $n$. By Lemma 3.10.11 in \cite{van1996weak}, the sequence of experiments $(\mathsf{X}_n,\calA_n,P_{n,h}: h \in H)$ is asymptotically normal in the sense of \cite[p~412]{van1996weak}. 
Under this setup, the following proposition regarding asymptotic efficiency of the empirical EOT map holds. We say that the parameter sequence $\kappa_n(h)$ is \textit{regular} if there exists a continuous linear operator $\dot{\kappa}:H \to \mathbb{B}$ such that $\sqrt{n}(\kappa_n(h) - \kappa_n(0)) \to \dot{\kappa}(h)$ for every $h \in H$; a sequence of ($\mathbb{B}$-measurable) estimators $\calT_n$ is called \textit{regular}  if the limit law of $\sqrt{n}(\calT_n-\kappa_n(h))$ under $P_{n,h}$ exists for every $h \in H$ and is independent of $h$. Additionally, a function $\ell: \mathbb{B} \to \R_{+}$ is called \textit{subconvex} if for every $c \in \R_+$, the level set $\ell^{-1}([0,c])$ is closed, convex, and symmetric.

\begin{proposition}[Asymptotic efficiency of empirical EOT map]
\label{prop: asymptotic efficiency}
Consider the above setting. Then the following hold.
\begin{enumerate}
\item[(i)] (Convolution) The sequences of parameters $\kappa_n(h)$ and estimators $\tilde{T}_n$ are regular. For every regular sequence of $\mathbb{B}$-measurable estimators $\calT_n$  based on $X_1^2,\dots,X_n^2$, the limit law of $\sqrt{n}(\calT_n-T^{\bm{\mu}})$ under $P_{n,0}$ equals the distribution of the sum  $\delta_{\mu^2}'(\GG_2^{\mu^2})+W$ for some $\mathbb{B}$-valued random variable $W$ independent of $\delta_{\mu^2}'(\GG_2^{\mu^2})$.
\item[(ii)] (Local asymptotic minimaxity) For every sequence of $\mathbb{B}$-measurable estimators $\calT_n$  based on $X_1^2,\dots,X_n^2$ and every subconvex function $\ell: \mathbb{B} \to \R_{+}$,
\[
\sup_{I \subset H: \text{finite}} \liminf_{n \to \infty}\sup_{h \in I}\E_{h}\Big [ \ell \Big(\sqrt{n}(\calT_n-\kappa_n(h) \big)\Big) \Big] \ge \E\Big[ \ell \big ( \delta_{\mu^2}'(\GG_2^{\mu^2}) \big ) \Big],
\]
where $\E_h$ denotes the expectation under $P_{n,h}$. 
\end{enumerate}
\end{proposition}

The regularity of the parameter sequence $\kappa_n(h)$ follows from Hadamard differentiability of the map $\nu^2 \mapsto T^{(\mu^1,\nu^2)}$. The regularity of the empirical EOT map follows from the Hadamard differentiability result and Le Cam's third lemma. The second claim of \cref{prop: asymptotic efficiency} (i) follows from applying Theorem 3.11.2 in \cite{van1996weak}. To this effect, we need to verify that the law of the Gaussian variable appearing in the cited theorem agrees with that of $\delta_{\mu^2}'(\GG_2^{\mu^2})$ in our setting, which follows by adapting the argument in the proof of Proposition 2 in \cite{goldfeld2022statistical}. Given (i), \cref{prop: asymptotic efficiency} (ii) directly follows from Theorem 3.11.5 in \cite{van1996weak}.

\cref{prop: asymptotic efficiency} (i) shows that the limit law $\delta_{\mu^2}'(\GG_2^{\mu^2})$ of the empirical EOT map is the most concentrated around zero among all regular estimators for $T$. Furthermore, by regularity of the empirical EOT map, for every bounded continuous function $\ell: \BB \to \R_{+}$ and every finite set $I \subset H$, 
\[
\lim_{n \to \infty}\sup_{h \in I}\E_{h}\Big [ \ell \Big(\sqrt{n}(\tilde{T}_n-\kappa_n(h) \big)\Big) \Big] = \E\Big[ \ell \big ( \delta_{\mu^2}'(\GG_2^{\mu^2}) \big ) \Big],
\]
showing that $\tilde{T}_n$ is asymptotically minimax in a local sense.

\section{Proofs for Sections \ref{sec: main results} and \ref{sec: EOT potential}}
\label{sec: proofs}

\subsection{Proof of \cref{thm: H derivative potential}}
\label{sec: proof of H derivative potential}
As noted in \cref{sec: H derivative of potential}, we will first establish Hadamard differentiability of the map $\bm{\nu} \mapsto \bm{\varphi}^{\bm{\nu}}$ in 
$\calC(\calX) \times \calC(\calX)$
.
\begin{lemma}
\label{lem: H derivative potential}
Consider the setting of \cref{thm: H derivative potential}.  Then, the map $\bm{\nu} \mapsto \bm{\varphi}^{\bm{\nu}}, 
\calP_{\mu^1} \times \calP_{\mu^2}
\subset 
\ell^\infty (B^s) \times \ell^\infty (B^s)
\to  \calC(\calX)\times\calC(\calX)$ is Hadamard differentiable at $\bm{\mu}$ tangentially to 
$\overline{\calM_{\mu^1}}^{\ell^\infty(B^s)} \times \overline{\calM_{\mu^2}}^{\ell^\infty(B^s)}$.
\end{lemma}

The proof of \cref{lem: H derivative potential} proceeds as follows. Fix $\bm{\mu}=(\mu^1,\mu^2)\in
\calP(\calX)\times \calP(\calX)
$ with $S_i:=\supp(\mu^i)$ for $i=1,2$. For notational convenience, set $S := S_1 \times S_2$ and $\DD:= \calC(S_1) \times \calC(S_2)$. We equip $\DD$ with a product norm, $\| (\varphi_1,\varphi_2)\|_{\DD} = \| \varphi_1 \|_{\infty,S_1} \vee \| \varphi_2 \|_{\infty,S_2}$.
Choose an arbitrary fixed reference point $(x_1^\circ,x_2^\circ) \in S$. 
With \cref{lem: EOT potential} in mind, consider
\[
\begin{split}
\Theta^s=\Big \{(\varphi_1\vert_{S_1},\varphi_2\vert_{S_2})  &:(\varphi_1,\varphi_2) \in \calC^s(\calX) \times \calC^s(\calX), \\
& \| \varphi_1 \|_{\calC^s(\calX)} \vee \| \varphi_2 \|_{\calC^s(
\calX
)} \le R_s, 
\varphi_1(x_1^\circ) = \varphi_2(x_2^\circ) \Big\} \subset \DD.
\end{split}
\]
Define the map $\Psi: 
\big(\calP_{\mu^1} \times \calP_{\mu^2}\big)
\times 
\Theta^s
\to 
\DD
$ by 
\begin{equation}
\Psi (
\bm\nu
,\bm{\varphi}) = \left ( \int e^{\frac{\varphi_1 \oplus \varphi_2 - c}{\varepsilon}} d
\nu^2
-1, \int e^{\frac{\varphi_1 \oplus \varphi_2 - c}{\varepsilon}} d
\nu^1
-1 \right )
\label{eq: psi function}
\end{equation}
for $
\bm\nu=(\nu^1,\nu^2) \in \calP_{\mu^1} \times \calP_{\mu^2}
$ 
and $\bm{\varphi} = (\varphi_1,\varphi_2) \in 
\Theta^s
$.  Given 
$\bm\nu\in\calP_{\mu^1} \times \calP_{\mu^2}$, the corresponding EOT potentials $\bm{\varphi}^{
\bm\nu}$ are fully characterized by its restriction to $S$, 
$\bm{\varphi}^{\bm{\nu}}\vert_{S}:=
({\varphi}^{\bm{\nu}}_1\vert_{S_1},{\varphi}^{\bm{\nu}}_2\vert_{S_2}) \in \Theta^s$,
being the unique solution to $\Psi (\bm{\nu},\bm{\varphi}) \equiv 0$ on $S$. We will decompose the map 
$\bm\nu\mapsto \bm\varphi^{\bm\nu}$
into the composition of the maps 
$
\bm\nu
\mapsto \Psi(
\bm\nu
,\cdot)$, $\Psi (
\bm\nu
,\cdot) \mapsto \bm{\varphi}^{
\bm\nu
}\vert_{S}$, and $\bm{\varphi}^{\bm\nu
}\vert_{S}\mapsto \bm{\varphi}^{\bm{\nu}}$, 
and separately show their Hadamard differentiability.

To formulate Hadamard differentiability of the map $\Psi (
\bm\nu
,\cdot) \mapsto 
\bm{\varphi}^{\bm{\nu}}\vert_{S}
$, we consider the following setting (cf. Section 3.9.4.7 in \cite{van1996weak}; see also \cref{sec: Z-functional}). Let 
$\ell^\infty (\Theta^s,\DD)$
be the Banach space of all uniformly norm-bounded maps $z: \Theta^s \to \DD$ equipped with the norm $\| z \|_{\ell^\infty (\Theta^s,\DD)} = \sup_{\bm{\varphi} \in \Theta^s} \| z (\bm{\varphi}) \|_{\DD}$. Also, let $Z(\Theta^s,\DD)$ be the subset of $\ell^\infty (\Theta^s,\DD)$ consisting of all maps with at least one zero. Let $\phi: Z(\Theta^s,\DD) \to \Theta^s$ be a map that assigns each $z \in Z(\Theta^s,\DD)$ to one of its zeros $\phi(z)$, i.e., $z (\phi(z)) = 0$.

For a given $\bm\nu
\in 
\calP_{\mu^1} \times \calP_{\mu^2}
$, $\Psi (
\bm\nu
,\cdot): \bm{\varphi} \mapsto \Psi (
\bm\nu
,\bm{\varphi})$ is a uniformly norm-bounded  map from $\Theta^s$ into $\DD$ with a zero at 
$\bm\varphi^{\bm\nu}\vert_{S}$, guaranteeing that $\Psi (
\bm\nu
,\cdot) \in Z(\Theta^s,\DD)$. Uniqueness of the EOT potentials then yields that
\[
\bm{\varphi}^{
\bm\nu
}\vert_{S} = \phi \circ \Psi (
\bm\nu
,\cdot).
\]
Furthermore, in light of \eqref{eq: FOC}, we have the following representation of the EOT potentials $\bm\varphi^{\bm\nu}$ for $\bm\nu\in\calP_{\mu^1} \times \calP_{\mu^2}$,
\begin{equation}
\label{eq:PotentialRepresentation}
    \bm\varphi^{\bm\nu}=-\varepsilon\Bigg(\log\left(\int e^{\frac{\varphi_2^{\bm\nu}\vert_{S_2}(x_2)-c(\cdot,x_2)}{\varepsilon}}d\nu^2(x_2)\right),
    \log\left(\int e^{\frac{\varphi_1^{\bm\nu}\vert_{S_1}(x_1)-c(x_1,\cdot)}{\varepsilon}}d\nu^1(x_1)\right)\Bigg),
\end{equation}
which depends on the potentials only through their restrictions to $S_1$ and $S_2$.

We will establish Hadamard differentiability of $\phi$ at $\Psi(\bm{\mu},\cdot)$  by invoking Lemma 3.9.34 in \cite{van1996weak} (see also \cref{lem: Z-functional}).  Hadamard differentiability of the map $
\bm\nu
\mapsto \Psi (\bm\nu,\cdot)$ is straightforward. The chain rule for Hadamard differentiable maps (cf. Lemma 3.9.3 in \cite{van1996weak}) then yields Hadamard differentiability of the map $\bm{\nu} \mapsto \bm{\varphi}^{\bm{\nu}}\vert_{S}$ in $\DD$. 
Hadamard differentiability of $\bm\nu \mapsto\bm\varphi^{\bm\nu}$ in $\calC(\calX) \times \calC(\calX)$ follows readily.

We shall first verify Hadamard differentiability of the map 
$\bm{\nu} \mapsto \Psi (\bm{\nu},\cdot), \calP_{\mu^1} \times \calP_{\mu^2} \subset \ell^\infty (B^s) \times \ell^\infty (B^s) \to Z(\Theta^s,\DD) \subset \ell^\infty (\Theta^s,\DD)$.
For notational convenience, define $\psi: 
\calP_{\mu^1} \times \calP_{\mu^2} 
\to Z(\Theta^s,\DD)$ by $\psi(
\bm\nu
) =  \Psi (
\bm\nu
,\cdot)$. 

 \begin{lemma}
\label{lem: Hadamard differentiability of objective function}
The map $\psi: 
\calP_{\mu^1} \times \calP_{\mu^2}
\subset 
\ell^\infty (B^s) \times \ell^\infty (B^s)
\to Z(\Theta^s,\DD) \subset \ell^\infty(\Theta^s,\DD)$ is Hadamard differentiable at $\bm{\mu}$ tangentially to $\overline{\calM_{\mu^1}}^{\ell^\infty(
B^s
)} \times \overline{\calM_{\mu^2}}^{\ell^\infty(
B^s
)}$ with derivative $\psi': \overline{\calM_{\mu^1}}^{\ell^\infty(
B^s
)} \times \overline{\calM_{\mu^2}}^{\ell^\infty(
B^s
)} \to \ell^\infty(\Theta^s,\DD)$ given by
\begin{equation}
\begin{split}
&\psi'(\bm{\gamma})(\bm{\varphi}) =\left ( \int e^{\frac{\varphi_1 \oplus \varphi_2 - c}{\varepsilon}} \, d\gamma^{2}, \int e^{\frac{\varphi_1 \oplus \varphi_2 - c}{\varepsilon}} \, d\gamma^{1} \right ), \\ 
&\bm{\varphi} = (\varphi_1,\varphi_2) \in \Theta^s, \bm{\gamma} = (\gamma^1,\gamma^2) \in\overline{\calM_{\mu^1}}^{\ell^{\infty}(B^s)}\times \overline{\calM_{\mu^2}}^{\ell^{\infty}(B^s)}. 
\end{split}
\label{eq: psi derivative}
\end{equation}
\end{lemma}

As before, $\gamma^j$ acts on the $j$-th coordinate, i.e.,
\[
\int e^{\frac{\varphi_1 \oplus \varphi_2 - c}{\varepsilon}} \, d\gamma^{2} = \int e^{\frac{\varphi_1 (\cdot)+ \varphi_2(x_2) - c(\cdot,x_2)}{\varepsilon}} \, d\gamma^{2}(x_2).
\]
\begin{proof}
Let $\mathfrak{M}_i$ 
denote the space of finite signed Borel measures on 
$\calX$ supported in $S_i=\supp(\mu^i)$.
The definition \eqref{eq: psi derivative} makes sense for $\bm{\gamma} \in \mathfrak{M}_1 \times \mathfrak{M}_2$, and the extended map $\psi': \mathfrak{M}_1 \times \mathfrak{M}_2 \to \ell^\infty (\Theta^s,\DD)$ is linear. 
For every $\bm{\varphi}=(\varphi_1,\varphi_2)\in\Theta^s$, let $\bar{\bm\varphi}=(\bar\varphi_1,\bar\varphi_2)\in \calC^s(\calX) \times \calC^s(\calX)$ denote an arbitrary $\calC^s$-extension with $\| \bar\varphi_1 \|_{\calC^s(\calX)} \vee \| \bar\varphi_2 \|_{\calC^s(\calX)} \le R_s$. 
Observe that 
\begin{equation}
\begin{split}
R' = \sup_{(x_1,x_2) \in \calX \times \calX
} \sup_{(\varphi_1,\varphi_2) \in \Theta^s} &\left \| e^{\frac{\bar \varphi_1(x_1) + \bar \varphi_2(\cdot) - c(x_1,\cdot)}{\varepsilon}} \right \|_{\calC^s(\calX
)} \\
&\qquad \qquad \bigvee \left \| e^{\frac{\bar \varphi_1(\cdot) + \bar \varphi_2(x_2) - c(\cdot,x_2)}{\varepsilon}} \right \|_{\calC^s(
\calX
)} < \infty.
\end{split}
\label{eq: R prime}
\end{equation}
Since $\psi'(\bm{\gamma})(\bm\varphi)=\psi'(\bm{\gamma})(\bar{\bm\varphi})|_{S}$ for
every $\bm{\gamma}  \in \mathfrak{M}_1 \times \mathfrak{M}_2$, we have
\begin{equation}
\| \psi'(\bm{\gamma}) \|_{\ell^\infty (\Theta^s,\DD)}\le R'\big ( \| \gamma^1 \|_{\infty,
B^s
} \vee \| \gamma^2 \|_{\infty,
B^s
}  \big).
\label{eq: lipschitz}
\end{equation}
Hence, $\psi'$  extends uniquely to a continuous linear operator from $\overline{\mathfrak{M}_1}^{\ell^\infty(
B^s
)} \times \overline{\mathfrak{M}_2}^{\ell^\infty(
B^s
)}$ into $\ell^\infty (\Theta^s,\DD)$. 

Second, for every $\bm{\gamma} \in \overline{\calM_{\mu^1}}^{\ell^\infty(
B^s
)} \times \overline{\calM_{\mu^2}}^{\ell^\infty(
B^s
)}$, pick a sequence of pairs of signed Borel measures $(\bm{\gamma}_t)_{t > 0}$ with total mass zero such that $\bm{\mu}+t\bm{\gamma}_t \in 
\calP_{\mu^1} \times \calP_{\mu^2}
$ for sufficiently small $t$ and $\bm{\gamma}_t \to \bm{\gamma}$ in $\ell^\infty (
B^s
) \times \ell^\infty(
B^s
)$ as $t \downarrow 0$. Then, as $t \downarrow 0$,
\[
\frac{\psi(\bm{\mu}+t\bm{\gamma}_t) - \psi(\bm{\mu})}{t} = \psi'(\bm{\gamma}_t) \to \psi'(\bm{\gamma}) \quad \text{in} \ \ell^\infty (\Theta^s,\DD). 
\]
This completes the proof. 
\end{proof}

Next, we shall establish Hadamard differentiability of the map $\phi: Z(\Theta^s,\DD) \subset \ell^\infty(\Theta^s,\DD) \to \DD$  at $\Psi_{\bm{\mu}} := \Psi (\bm{\mu},\cdot)$.
To this end, we apply Lemma 3.9.34 in \cite{van1996weak}. The following lemma verifies the required conditions to apply the lemma. The proof, which we defer to \cref{sec: auxiliary proofs}, relies on the results from \cite{carlier2020differential}.
\begin{lemma}
\label{lem: Psi functional}
The following hold.
\begin{enumerate}
    \item[(i)] The map $\Psi_{\bm{\mu}}: \Theta^s \ni  \bm{\varphi} \mapsto \Psi (\bm{\mu},\bm{\varphi}) \in \DD$ is injective and its inverse (defined on $\Psi_{\bm{\mu}}(\Theta^s)$) is continuous at $0$.
    \item[(ii)] The map $\Psi_{\bm{\mu}}$ is Fr\'{e}chet differentiable at $\bm{\varphi} = \bm{\varphi}^{\bm{\mu}}\vert_{S}$ with derivative $\dot{\Psi}_{\bm{\mu}}: \lin (\Theta^s) \to \DD$ given by
\[
\dot{\Psi}_{\bm{\mu}} (\bm{h})= \varepsilon^{-1}\left ( \int e^{\frac{\varphi_1^{\bm{\mu}}\oplus \varphi_2^{\bm{\mu}} - c}{\varepsilon}} (h_1 \oplus h_2) \, d\mu^2,\int e^{\frac{\varphi_1^{\bm{\mu}}\oplus \varphi_2^{\bm{\mu}} - c}{\varepsilon}} (h_1 \oplus h_2) \, d\mu^1  \right )
\]
for $\bm{h} = (h_1,h_2) \in \lin (\Theta^s)$, where $\lin (\Theta^s)$ is the linear hull of $\Theta^s$.
Furthermore, $\dot{\Psi}_{\bm{\mu}}: (\lin (\Theta^s), \| \cdot \|_{\DD}) \to \DD$ is injective and its inverse is continuous. 
\end{enumerate} 
\end{lemma}

Extend $\dot \Psi_{\bm{\mu}}^{-1}$  continuously to $\overline{\dot \Psi_{\bm{\mu}} (\lin (\Theta^s))}^{\DD}$. 
Now, Lemma 3.9.34 in \cite{van1996weak} implies the following. See also \cref{sec: Z-functional}.

\begin{lemma}
\label{lem: Hadamard differentiability of zero}
The map $\phi: Z(\Theta^s,\DD) \subset \ell^\infty(\Theta^s,\DD) \to \DD$ is Hadamard differentiable at $\Psi_{\bm{\mu}} := \Psi (\bm{\mu},\cdot)$ tangentially to the set
\[
\begin{split}
\mathcal{Z}_{\bm{\mu}} =&\Big \{ z \in \ell^\infty (\Theta^s,\DD) : z=\lim_{t \downarrow 0} \frac{z_t -\Psi_{\bm{\mu}}}{t} \ \text{for some} \ z_t \to \Psi_{\bm{\mu}} \ \text{in} \ Z(\Theta^s,\DD), \  t \downarrow  0 \Big\} \\
&\bigcap \Big \{ z \in \ell^\infty (\Theta^s,\DD) : \text{$z$ is continuous at $\bm{\varphi}^{\bm{\mu}}\vert_{S}$} \Big \}. 
\end{split}
\]
 The derivative is given by $\phi_{\Psi_{\bm{\mu}}}'(z) = - \dot{\Psi}_{\bm{\mu}}^{-1}(z(\bm{\varphi}^{\bm{\mu}}\vert_{S}))$.
\end{lemma}

Hadamard differentiability of the map $\bm\nu\in\calP_{\mu^1} \times \calP_{\mu^2}\mapsto \bm\varphi^{\bm\nu}\vert_{S}\in\DD$ at $\bm\mu$ now follows directly from Lemmas  \ref{lem: Hadamard differentiability of objective function} and \ref{lem: Hadamard differentiability of zero}.

\begin{lemma}
\label{lem:HadamardDerivativeRestriction}
     The map $\bm\nu\mapsto \bm\varphi^{\bm\nu}\vert_{S},$ $\calP_{\mu^1} \times \calP_{\mu^2}\subset \ell^{\infty}(B^s)\times\ell^{\infty}(B^s)\to \DD$  is Hadamard differentiable at $\bm\mu$ tangentially to $\overline{\calM_{\mu^1}}^{\ell^{\infty}(B^s)}\times \overline{\calM_{\mu^2}}^{\ell^{\infty}(B^s)}$.
\end{lemma}

\begin{proof}
Given Lemmas  \ref{lem: Hadamard differentiability of objective function} and \ref{lem: Hadamard differentiability of zero}, the lemma follows from the chain rule for Hadamard differentiable maps; see Lemma 3.9.3 in \cite{van1996weak}. The only thing we need to verify is that, for every $\bm{\gamma} = (\gamma^1,\gamma^2) \in \overline{\calM_{\mu^1}}^{\ell^\infty(
B^s
)} \times \overline{\calM_{\mu^2}}^{\ell^\infty(
B^s
)}$, it holds that $\psi' (\bm{\gamma}) \in \mathcal{Z}_{\bm{\mu}}$.  Let $(\bm{\mu}_t)_{t > 0} \subset 
\calP_{\mu^1} \times \calP_{\mu^2}
$ be a sequence such that $\bm{\gamma}_t := t^{-1}(\bm{\mu}_t - \bm{\mu}) \to \bm{\gamma}$ in $\ell^\infty (
B^s
) \times \ell^\infty(
B^s
)$. Then $\Psi_{\bm{\mu}_t} \in Z(\Theta^s,\DD)$ and $t^{-1}(\Psi_{\bm{\mu}_t}-\Psi_{\bm{\mu}}) = \psi'(\bm{\gamma}_t)  \to \psi'(\bm{\gamma})$ in $\ell^\infty(\Theta^s,\DD)$. 
It remains to show that $\psi'(\bm{\gamma})$ is continuous at $\bm{\varphi}^{\bm{\mu}}\vert_{S}$. By construction, for every $\eta > 0$, there exists $\tilde{\bm{\gamma}} = (\tilde{\gamma}^1,\tilde{\gamma}^2) \in \calM_{\mu^1} \times \calM_{\mu^2}$ such that $\| \gamma^1 - \tilde{\gamma}^1 \|_{\ell^\infty(
B^s
)} \vee \| \gamma^2 - \tilde{\gamma}^2 \|_{\ell^\infty(
B^s
)} \le \eta$. By \eqref{eq: lipschitz}, we have 
$\| \psi' (\bm{\gamma}) - \psi'(\tilde{\bm{\gamma}}) \|_{\ell^\infty(\Theta^s,\DD)} \le R'\eta$ with $R'$ given by \eqref{eq: R prime}. Since $\tilde{\bm{\gamma}}$ is a pair of signed measures, for every sequence $\bm{\varphi}_n \in \Theta^s$ with $\| \bm{\varphi}_n - \bm{\varphi}^{\bm{\mu}}\vert_{S}\|_{\DD} \to 0$, we see that $\| \psi'(\tilde{\bm{\gamma}})(\bm{\varphi}_n) - \psi'(\tilde{\bm{\gamma}})(\bm{\varphi}^{\bm{\mu}}\vert_{S})\|_{\DD} \to 0$. We thus conclude that $\limsup_{n \to \infty} \| \psi'(\bm{\gamma})(\bm{\varphi}_n) - \psi'(\bm{\gamma})(\bm{\varphi}^{\bm{\mu}}\vert_{S}) \|_{\DD} \le 2R'\eta$, and as $\eta > 0$ is arbitrary, we have shown that $\psi'(\bm{\gamma})$ is continuous at $\bm{\varphi}^{\bm{\mu}}\vert_{S}$. Hence, the chain rule applies, and the Hadamard derivative $[\bm{\varphi}^{\bm{\mu}}\vert_{S}]': 
\overline{\calM_{\mu^1}}^{\ell^\infty(
B^s
)} \times \overline{\calM_{\mu^2}}^{\ell^\infty(
B^s
)} \to \DD$ is given by $[\bm{\varphi}^{\bm{\mu}}\vert_{S}]'(\bm{\gamma}) = -\dot \Psi_{\bm{\mu}}^{-1} (\psi'(\bm{\gamma})(\bm{\varphi}^{\bm{\mu}}\vert_{S}))$.
\end{proof}

Now, Lemma \ref{lem:HadamardDerivativeRestriction} and the formula for the potentials \eqref{eq:PotentialRepresentation} together yield \cref{lem: H derivative potential}.

\begin{proof}[Proof of \cref{lem: H derivative potential}]
We show Hadamard differentiability of the map $\bm\nu\mapsto\varphi_1^{\bm\nu}$ at $\bm\mu$; differentiability of the second potential follows analogously. 
To simplify notation, define 
\[
    \xi^{\bm\nu}:(x_1,x_2)\in\calX\times S_2\mapsto e^{\frac{\bm\varphi^{\bm\nu}\vert_{S_2}(x_2)-c(x_1,x_2)}{\varepsilon}}.
\]
By the formula \eqref{eq:PotentialRepresentation} 
 and the chain rule, it suffices to show that the map $\Upsilon:\bm\nu\in\calP_{\mu^1} \times \calP_{\mu^2}\mapsto \int\xi^{\bm\nu}(\cdot,x_2)d\nu^2(x_2)\in\calC(\calX)$ is Hadamard differentiable at $\bm\mu$ tangentially to $\overline{\calM_{\mu^1}}^{\ell^{\infty}(B^s)}\times \overline{\calM_{\mu^2}}^{\ell^{\infty}(B^s)}$. 
To this effect, let $(\bm\mu_t)_{t>0}\subset \calP_{\mu^1} \times \calP_{\mu^2}$ be such that $\bm\gamma_t:=t^{-1}(\bm\mu_t-\bm\mu)\to\bm\gamma$ in $\ell^{\infty}(B^s)\times \ell^{\infty}(B^s)$, and consider
\begin{align*} 
&t^{-1}\left(
\Upsilon(\bm\mu_t)-\Upsilon(\bm\mu)
\right)
=
    \int t^{-1}\left(\xi^{\bm\mu_t}(\cdot,x_2)-\xi^{\bm\mu}(\cdot,x_2)\right)\, d\mu^2(x_2)+\int\xi^{\bm\mu_t}(\cdot,x_2)\, d\gamma_t^2(x_2). 
\end{align*}
As $t\downarrow 0$, the first term on the right-hand side converges  to the Hadamard derivative of  the map $\bm\nu\in\calP_{\mu^1} \times \calP_{\mu^2}\mapsto \int \xi^{\bm\nu}(\cdot,x_2)d\nu^2(x_2)\in\calC(\calX)$ at $\bm\mu$, which agrees with
$
\varepsilon^{-1}  
    \int \xi^{\bm\mu}(\cdot,x_2)
    [\varphi_2^{\bm\mu}\vert_{S_2}]'(\bm \gamma)(x_2)\, d\mu^2(x_2)
$ 
by \cref{lem:HadamardDerivativeRestriction} and the chain rule.

Pertaining to the second term, since $\supp(\gamma_t^2)\subset S_2$, we have
\begin{align*}
&\sup_{x_1\in\calX}\Big|\int\xi^{\bm\mu_t}(x_1,x_2)\, d\gamma_t^2(x_2)-\int\xi^{\bm\mu}(x_1,x_2)\, d\gamma^2(x_2)\Big| \\
    &\leq \sup_{x_1\in\calX}\Big\|e^{\frac{\varphi^{\bm\mu_t}_2(\cdot)-c(x_1,\cdot)}{\varepsilon}}-e^{\frac{\varphi^{\bm\mu}_2(\cdot)-c(x_1,\cdot)}{\varepsilon}}\Big\|_{\calC^s(\calX)}\|\gamma_t^2\|_{\infty,B^s}\\
&\quad + \sup_{x_1\in\calX}\Big\|e^{\frac{\varphi^{\bm\mu}_2(\cdot)-c(x_1,\cdot)}{\varepsilon}}\Big\|_{\calC^s(\calX)}\|\gamma_t^2-\gamma^2\|_{\infty,B^s}, 
\end{align*}
The right-hand side  converges to $0$ as $t\downarrow 0$ since $\|\gamma_t^2\|_{\infty,B^s}=O(1)$, $\|\gamma_t^2-\gamma^2\|_{\infty,B^s}=o(1)$, and $\varphi^{\bm\mu_t}_2\to\varphi^{\bm\mu}_2$ in $\calC^s(\calX)$ by \cref{lem: EOT potential} (iii) (note that $\bm\mu_t\to\bm\mu$ in $\ell^{\infty}(B^s)\times\ell^{\infty}(B^s)$ and hence $\mu_t^i\to \mu^i$ weakly for $i=1,2$ by \cref{lem: convergence determining}).  
Conclude that, for $\bm{\gamma}\in\overline{\mathcal M_{\mu^1}}^{\ell^{\infty}(B^s)}\times \overline{\mathcal M_{\mu^2}}^{\ell^{\infty}(B^s)}$,
\begin{equation}
\label{eq:IntegratedXiDerivative}
  \Upsilon'_{\bm\mu}(\bm\gamma)=\varepsilon^{-1}\int \xi^{\bm\mu}(\cdot,x_2)[\varphi_2^{\bm\mu}\vert_{S_2}]'(\bm\gamma)(x_2)\, d\mu^2(x_2)
   + \int \xi^{\bm\mu}(\cdot,x_2)\, d\gamma^2(x_2),
\end{equation}
and hence $[\varphi_1^{\bm\mu}]'(\bm\gamma)=-\varepsilon e^{\frac{\varphi^{\bm\mu}_1}{\varepsilon}}\Upsilon'_{\bm\mu}(\bm\gamma)$ by the chain rule. 
\end{proof}

    \begin{remark}[Compatibility of derivatives of EOT potentials]
    \label{rmk:CompatibilityPotentialDerivatives}
        As Lemmas \ref{lem:HadamardDerivativeRestriction} and \ref{lem: H derivative potential} establish, respectively, the Hadamard derivatives of the maps $\bm\nu\in\calP_{\mu^1} \times \calP_{\mu^2}\mapsto \bm\varphi^{\bm\nu}\vert_{S}\in \DD$ and $\bm\nu\in\calP_{\mu^1} \times \calP_{\mu^2}\mapsto \bm\varphi^{\bm\nu}\in\calC(\calX)\times \calC(\calX)$, the latter derivative should extend the former. Indeed, for every $\bm\gamma\in\overline{\calM_{\mu^1}}^{\ell^{\infty}(B^s)}\times\overline{\calM_{\mu^2}}^{\ell^{\infty}(B^s)}$ and $(\bm\mu_t)_{t>0}\subset\calP_{\mu^1} \times \calP_{\mu^2}$ for which $t^{-1}(\bm\mu_t-\bm\mu)\to \bm\gamma$ in $\ell^{\infty}(B^s)\times\ell^{\infty}(B^s)$, $t^{-1}(\bm\varphi^{\bm\mu_t}-\bm\varphi^{\bm\mu})\to[\bm\varphi^{\bm\mu}]'(\bm\gamma)$ in $\calC(\calX)\times \calC(\calX)$ and hence also in $\DD$. So, $[\bm\varphi^{\bm\mu}]'(\bm\gamma)\vert_{S}=[\bm\varphi^{\bm\mu}\vert_{S}]'(\bm\gamma)$, as desired.  
    \end{remark}

    \begin{remark}[Choice of reference point]
We have chosen a reference point $(x_1^\circ,x_2^\circ)$ from $S$ in the proof of Lemma \ref{lem: H derivative potential}, but this is immaterial. Indeed, for a different choice of reference point $(\tilde x_1^\circ,\tilde x_2^\circ) \in \calX \times \calX$, the functions $\tilde{\varphi}_1^\mu = \varphi_1^\mu - \frac{1}{2} (\varphi_1(\tilde{x}_1^\circ)-\varphi_2(\tilde{x}_1^\circ))$ and $\tilde{\varphi}_2^\mu = \varphi_2^\mu+\frac{1}{2} (\varphi_1(\tilde{x}_1^\circ)-\varphi_2(\tilde{x}_1^\circ))$ are EOT potentials satisfying the constraint $\tilde{\varphi}_1^\mu(\tilde x_1^\circ)= \tilde{\varphi}_2^\mu(\tilde x_2^\circ)$. Clearly, by construction, the map $\bm{\nu} \mapsto \tilde{\bm{\varphi}}^{\bm{\nu}}, \calP_{\mu^1} \times \calP_{\mu^2} \subset \ell^\infty (B^s) \times \ell^\infty (B^s) \to \calC(\calX) \times \calC(\calX)$ is Hadamard differentiable at $\bm{\mu}$. Hence the conclusion of Lemma \ref{lem: H derivative potential} holds for an arbitrary reference point $(x_1^\circ,x_2^\circ) \in \calX \times \calX$.
    \end{remark}

We are now ready to prove \cref{thm: H derivative potential}.

\begin{proof}[Proof of \cref{thm: H derivative potential}]
We divide the proof into two steps.

\underline{Step 1}. 
Pick any multi-index $k=(k_1,\dots,k_d) \in \mathbb{N}_0^d$ with $0 < |k| \le s$. In what follows, $i \in \{ 1,2 \}$ is arbitrary. Also, $\mu^{-1} = \mu^{2}$ and $\mu^{-2} = \mu^1$. Similar conventions apply to $x_{-i}$ etc. We will show that the map
\[
\bm{\nu} \mapsto D^{k} \varphi_i^{\bm{\nu}}, \ 
\calP_{\mu^1} \times \calP_{\mu^2}
\subset \ell^\infty(
B^s
)\times\ell^\infty(
B^s
) \to \calC(
\calX
)
\]
is Hadamard differentiable at $\bm{\mu}$.
Observe that 
\[
e^{-\varphi_i^{\bm{\nu}}(x_i)/\varepsilon} =\int  e^{ (\varphi_{-i}^{\bm{\nu}}(x_{-i})-c(x_1,x_2))/\varepsilon}d\nu^{-i}(x_{-i}),
\]
so that by interchanging differentiation and integration, we see that $D^{k}_{x_i} \big (e^{-\varphi_i^{\bm{\nu}}(x_i)/\varepsilon}\big)$ can be expressed as a linear combination of functions of the form
\[
\int \prod_{j=1}^J \big[D_{x_i}^{\ell_j}c(x_1,x_2)  \big]^{m_j} \times e^{ (\varphi_{-i}^{\bm{\nu}}(x_{-i})-c(x_1,x_2))/\varepsilon} d\nu^{-i}(x_{-i}),
\]
where $1 \le J \le |k|$, and $m_1,\dots,m_J \in \NN$ and $\ell_1,\dots,\ell_J \in \NN_0^d \setminus \{ 0 \}$ are such that $m_1\ell_1+\dots+m_J\ell_J = k$.
Combining the fact that $(\log y)^{(n)} = (-1)^{n+1}(n-1)!y^{-n}$ and the multivariate Fa\`{a} di Bruno formula (cf. Theorem 2.1 in \cite{constantine1996multivariate}), we see that 
$D^{k}\varphi_i^{\bm{\nu}} (x_i) = -\varepsilon D^{k}  \log (e^{-\varphi_i^{\bm{\nu}} (x_i)/\varepsilon})$ can be expressed as a linear combination of products of functions of the form 
\[
\begin{split}
&\frac{\int \zeta(x_1,x_2) e^{(\varphi_{-i}^{\bm{\nu}}(x_{-i})-c(x_1,x_2))/\varepsilon}d\nu^{-i}(x_{-i})}{\int  e^{(\varphi_{-i}^{\bm{\nu}}(x_{-i})-c(x_1,x_2))/\varepsilon}d\nu^{-i}(x_{-i})} \\
&= \int \zeta(x_1,x_2) e^{\frac{\varphi_1^{\bm{\nu}}(x_1)+\varphi_2^{\bm{\nu}}(x_2)-c(x_1,x_2)}{\varepsilon}}d\nu^{-i}(x_{-i}),
\end{split}
\] 
where $\zeta$ is a smooth function on $\R^d \times \R^d$ that depends only on the cost $c$ and multi-index $k$.
For example, for $y=(y_1,\dots,y_d)$, 
\[
\begin{split}
\frac{\partial^3}{\partial y_{j_1} \partial y_{j_2} \partial y_{j_3}} \log (f(y)) &= 
\frac{\frac{\partial^3}{\partial y_{j_1} \partial y_{j_2} \partial y_{j_3}}f(y)}{f(y)} - \frac{\frac{\partial^2}{\partial y_{j_1}\partial y_{j_3}}f(y)}{f(y)}\frac{\frac{\partial}{\partial y_{j_2}}f(y)}{f(y)} \\
&- \frac{\frac{\partial}{\partial y_{j_1}}f(y)}{f(y)} \frac{\frac{\partial^2}{\partial y_{j_2}\partial y_{j_3}}f(y)}{f(y)} +2
\frac{\frac{\partial}{\partial y_{j_1}}f(y)}{f(y)}\frac{\frac{\partial}{\partial y_{j_2}}f(y)}{f(y)}\frac{\frac{\partial}{\partial y_{j_3}}f(y)}{f(y)}.
\end{split}
\]
Hence, it suffices to show that, for every smooth function $\zeta$ on $\R^d \times \R^d$, the map
\[
\bm{\nu} \mapsto \int \zeta(x_1,x_2) e^{\frac{\varphi_1^{\bm{\nu}}(x_1)+\varphi_2^{\bm{\nu}}(x_2)-c(x_1,x_2)}{\varepsilon}} d\nu^{-i}(x_{-i}), \ 
\calP_{\mu^1} \times \calP_{\mu^2}
\subset \ell^\infty(
B^s
)\times\ell^\infty(
B^s
) \to \calC(
\calX
)
\]
is Hadamard differentiable at $\bm{\mu}$.

Let $(\bm{\mu}_t)_{t > 0} \subset  
\calP_{\mu^1} \times \calP_{\mu^2}
$ be a sequence such that $\bm{\gamma}_t := t^{-1}(\bm{\mu}_t-\bm{\mu}) \to \bm{\gamma}$ in $\ell^\infty (
B^s
) \times \ell^\infty (
B^s
)$ as $t \downarrow 0$ with $\bm{\gamma}=(\gamma_1,\gamma_2)$. Define 
\[
g_t(x_1,x_2)= \zeta(x_1,x_2) e^{\frac{\varphi_1^{\bm{\mu}_t}(x_1)+\varphi_2^{\bm{\mu}_t}(x_2)-c(x_1,x_2)}{\varepsilon}}, \ t \ge 0
\]
with $\bm{\mu}_0 = \bm{\mu}$. We have shown in \cref{lem: H derivative potential} that $t^{-1} (\varphi_{i}^{\bm{\mu}_t}-\varphi_i^{\bm{\mu}}) \to [\varphi_i^{\bm{\mu}}]'(\bm{\gamma})$ in $\calC(
\calX
)$
as $t \downarrow 0$, so that
\[
t^{-1}(g_t - g_0) \to \varepsilon^{-1}  \big \{ [\varphi_1^{\bm{\mu}}]'(\bm{\gamma}) \oplus [\varphi_2^{\bm{\mu}}]'(\bm{\gamma}) \big \} g_0 =: h(\bm{\gamma}) \quad \text{in} \ \calC(
\calX\times \calX
). 
\]

Observe that
\[
\int g_t (x_1,x_2) \, d\mu^{-i}_t(x_{-i}) 
= \int g_t (x_1,x_2)\, d\mu^{-i}(x_{-i})  + t \int g_t (x_1,x_2)\, d\gamma_{t}^{-i}(x_{-i}). 
\]
As $t \downarrow 0$, we have
\[
\begin{split}
&t^{-1}\left\{ \int g_t (x_1,x_2) \, d\mu^{-i}(x_{-i}) - \int g_0 (x_1,x_2) \, d\mu^{-i}(x_{-i}) \right \} \\
&\quad \to \int h(\bm{\gamma})(x_1,x_2) \, d\mu^{-i}(x_{-i}) \quad \text{in} \ \calC(
\calX
).
\end{split}
\]
To control $\int g_t (x_1,x_2)\, d\gamma_{t}^{-i}(x_{-i})$, observe that
\[
\begin{split}
&\left| \int g_t (x_1,x_2) \, d\gamma_t^2(x_2) - \int g_{0} (x_1,x_2) \, d\gamma^2(x_2) \right| \\
&\le \left| \int (g_t-g_0) (x_1,x_2) \, d\gamma_t^2(x_2) \right | + \left | \int g_{0} (x_1,x_2) d(\gamma_t^2-\gamma^2)(x_2) \right | 
 \\ 
&\le \|  g_t(x_1,\cdot)-g_0(x_1,\cdot) \|_{\calC^s(
\calX
)}  \| \gamma_t^2 \|_{\infty,
B^s
} +  \|g_0(x_1,\cdot) \|_{\calC^s(
\calX
)} \| \gamma_t^2 - \gamma^2 \|_{\infty,
B^s
}.
\end{split}
\]
We have $\| \gamma_t^2 \|_{\infty,
B^s
} = O(1)$ and $\| \gamma_t^2 - \gamma^2 \|_{\infty,
B^s
} = o(1)$ as $t \downarrow 0$ by construction, and $\sup_{x_1 \in 
\calX
} \|g_0(x_1,\cdot) \|_{\calC^s(
\calX
)} < \infty$ by \cref{lem: EOT potential} (ii). Since $\mu_t^i$ converges weakly to $\mu^i$ as $t \downarrow 0$ (as convergence in $\ell^\infty(
B^s
)$ implies weak convergence; cf. \cref{lem: convergence determining}), \cref{lem: EOT potential} (iii) implies that $\bm{\varphi}^{\bm{\mu}_t} \to \bm{\varphi}^{\bm{\mu}}$ in $\calC^s(
\calX
) \times \calC^s(
\calX
)$, which in turn implies
that $\sup_{x_1 \in 
\calX
} \| g_t(x_1,\cdot)-g_0(x_1,\cdot) \|_{\calC^s(
\calX
)} = o(1)$ as $t \downarrow 0$. Hence, we have
\[
\int g_t (\cdot,x_2) \, d\gamma_t^2(x_2) \to \int g_0 (\cdot,x_2) \, d\gamma^2(x_2) \quad \text{in} \ \calC(
\calX
). 
\]
Likewise, we have $\int g_t (x_1,\cdot) \, d\gamma_t^1(x_1) \to \int g_0 (x_1,\cdot) \, d\gamma^1(x_1)$ in $\calC(
\calX
)$. Conclude that
\[
\begin{split}
&t^{-1} \left \{\int g_t (x_1,x_2) \, d\mu^{-i}_t(x_{-i}) - \int g_0(x_1,x_2) \, d\mu^{-i}(x_{-i}) \right\} \\
&\quad \to \int h(\bm{\gamma})(x_1,x_2) \, d\mu^{-i}(x_{-i}) + \int g_0 (x_1,x_2) \, d\gamma^{-i}(x_{-i}) \quad \text{in} \ \calC(
\calX
).
\end{split}
\]
The limit is linear and continuous from $\overline{\calM_{\mu^1}}^{\ell^\infty (
B^s
)} \times \overline{\calM_{\mu^2}}^{\ell^\infty (
B^s
)}$ into $\calC(
\calX
)$.

\underline{Step 2}. As in Step 1, let $(\bm{\mu}_t)_{t > 0} \subset  
\calP_{\mu^1} \times \calP_{\mu^2}
$ be a sequence such that $\bm{\gamma}_t := t^{-1}(\bm{\mu}_t-\bm{\mu}) \to \bm{\gamma}$ in $\ell^\infty (
B^s
) \times \ell^\infty (
B^s
)$ as $t \downarrow 0$. By \cref{lem: H derivative potential} and Step 1, for every multi-index $k\in \mathbb{N}_0^d$ with $|k| \le s$, the map $\bm{\nu} \mapsto D^k \varphi_i^{\bm{\nu}},  
\calP_{\mu^1} \times \calP_{\mu^2}
\subset \ell^\infty(
B^s
)\times\ell^\infty(
B^s
) \to \calC(
\calX
)$ is Hadamard differentiable at $\bm{\mu}$. Denote its derivative by $[D^k \varphi_i^{\bm{\mu}}]'$, so that
\[
t^{-1}(D^k\varphi_i^{\bm{\mu}_t} - D^k\varphi_i^{\bm{\mu}}) \to [D^k \varphi_i^{\bm{\mu}}]'(\bm{\gamma}) \quad \text{in} \ \calC(
\calX
). 
\]
Pick any sequence $t_n \downarrow 0$. Then, $t_n^{-1}(\varphi_i^{\bm{\mu}_{t_n}} - \varphi_i^{\bm{\mu}})$ is Cauchy in $\calC^s(
\calX
)$, so by completeness of $\calC^s(
\calX
)$, the limit in $\calC^s(
\calX
)$ exists, i.e., $t_n^{-1}(\varphi_i^{\bm{\mu}_{t_n}} - \varphi_i^{\bm{\mu}}) \to \bar{\varphi}_i$ in $\calC^s(
\calX
)$. 
The limit $\bar{\varphi}_i$ satisfies that $D^k\bar{\varphi}_i = [D^k \varphi_i^{\bm{\mu}}]'(\bm{\gamma})$ for every multi-index $k\in \mathbb{N}_0^d$ with $|k| \le s$, which shows that $[\varphi_i^{\bm{\mu}}]'(\bm{\gamma}) \in \calC^s(
\calX
)$ with $D^k [\varphi_i^{\bm{\mu}}]'(\bm{\gamma}) = [D^k \varphi_i^{\bm{\mu}}]'(\bm{\gamma})$ for every multi-index $k\in \mathbb{N}_0^d$ with $|k| \le s$. Since the map $\bm{\gamma} \mapsto [\varphi_i^{\bm{\mu}}]'(\bm{\gamma})$ is linear and continuous from $\overline{\calM_{\mu^1}}^{\ell^\infty(
B^s
)} \times \overline{\calM_{\mu^2}}^{\ell^\infty(
B^s
)}$ into $\calC^s(
\calX
)$, we obtain the desired result.
\end{proof}

\begin{remark}
Another possible approach would be to employ the implicit function theorem for Banach spaces (see, e.g., Theorem I.5.9 in \cite{lang2012fundamentals}), which asks Fr\'{e}chet differentiability of the map $\bm{\mu} \mapsto \bm{\varphi}^{\bm{\mu}}$. However, in our problem, it seems highly nontrivial to verify the required conditions to directly apply the implicit function theorem. For statistical purposes, Hadamard differentiability is sufficient in most cases; cf. Chapter 3.9 in \cite{van1996weak}. 
\end{remark}

\subsection{Proof of \cref{thm: second order H derivative}}
\label{sec: proof of second order H derivative}
We divide the proof into two steps.

\underline{Step 1}. We first show twice Hadamard differentiability of the mapping $\bm{\nu} \mapsto \bm{\varphi}^{\bm{\nu}}$ in $\DD$. 
Recall that $\bm{\varphi}^{\bm{\mu}_t}\vert_{S} = \phi \circ \Psi (\bm{\mu}_t,\cdot)$. Observe that $\Psi (\bm{\mu}_t,\cdot) =\Psi (\bm{\mu},\cdot) + t \psi'(\bm{\gamma}_t)$.
We will apply \cref{lem: second order H derivative} below with $\Theta = \Theta^s, \D = \DD, \LL = \DD, \Psi = \Psi (\bm{\mu},\cdot), \theta_0 = \bm{\varphi}^{\bm{\mu}}\vert_{S}, \theta_t = \bm{\varphi}^{\bm{\mu}_t}\vert_{S}$, and $z_t = \psi'(\bm{\gamma}_t)$. To this end, we shall verify the conditions in \cref{lem: second order H derivative}.

Twice Fr\'{e}chet differentiability of $\Psi(\bm{\mu},\cdot)$ is straightforward to verify, with second derivative given by
\[
\ddot{\Psi}_{\bm{\mu}} (h_1,h_2)= \varepsilon^{-2}\left ( \int e^{\frac{\varphi_1^{\bm{\mu}}\oplus \varphi_2^{\bm{\mu}} - c}{\varepsilon}} (h_1 \oplus h_2)^2 \, d\mu^2, \int e^{\frac{\varphi_1^{\bm{\mu}}\oplus \varphi_2^{\bm{\mu}} - c}{\varepsilon}} (h_1 \oplus h_2)^2 \, d\mu^1  \right ). 
\]
We have already verified Conditions (i)--(ii) of \cref{lem: second order H derivative} in the proof of \cref{lem:HadamardDerivativeRestriction}. Regarding Condition (iii), since $\psi'(\bm{\gamma}_t) \in \mathcal{Z}_{\bm{\mu}}$ (cf. the proof of \cref{lem:HadamardDerivativeRestriction}), in view of \cref{rem: Z-functional} after \cref{lem: Z-functional}, it holds that $\psi'(\bm{\gamma}_t)(\bm{\varphi}^{\bm{\mu}}\vert_{S}) \in \overline{\dot \Psi_{\bm{\mu}}(\lin (\Theta^s))}^{\DD}$.  
It remains to verify Condition (iv) in \cref{lem: second order H derivative}. Observe that, as $\supp (\gamma_t^i) \subset S_i$,
\[
\begin{split}
&t^{-1} \left \{\psi'(\bm{\gamma}_t)(\bm{\varphi}^{\bm{\mu}_t}\vert_{S}) - \psi'(\bm{\gamma}_t)(\bm{\varphi}^{\bm{\mu}}\vert_{S}) \right \}\\
&= \left ( \int t^{-1}\Big ( e^{\frac{\varphi_1^{\bm{\mu}_t}\oplus \varphi_2^{\bm{\mu}_t} - c}{\varepsilon}} - e^{\frac{\varphi_1^{\bm{\mu}}\oplus \varphi_2^{\bm{\mu}} - c}{\varepsilon}} \Big ) \, d\gamma_t^{2},
\int t^{-1}\Big ( e^{\frac{\varphi_1^{\bm{\mu}_t} \oplus \varphi_2^{\bm{\mu}_t} - c}{\varepsilon}} - e^{\frac{\varphi_1^{\bm{\mu}}\oplus \varphi_2^{\bm{\mu}} - c}{\varepsilon}} \Big ) \, d\gamma_t^{1} \right ). 
\end{split}
\]
By symmetry, it suffices to show convergence of the first coordinate. Recall from \cref{thm: H derivative potential} that $t^{-1}(\varphi_i^{\bm{\mu}_t} - \varphi_i^{\bm{\mu}}) \to [\varphi_i^{\bm{\mu}}]'(\bm{\gamma})$ in $\calC^s(
\calX)$ for $i=1,2$. We shall show that
\begin{equation}
\begin{split}
\sup_{x_1 \in 
\calX
} \Bigg \|
&t^{-1}\Big\{ e^{\frac{\varphi_1^{\bm{\mu}_t}(x_1) + \varphi_2^{\bm{\mu}_t}(\cdot) - c(x_1,\cdot)}{\varepsilon}} - e^{\frac{\varphi_1^{\bm{\mu}}(x_1) + \varphi_2^{\bm{\mu}}(\cdot) - c(x_1,\cdot)}{\varepsilon}} \Big \} \\
&- \varepsilon^{-1}\big \{ [\varphi_1^{\bm{\mu}}]'(\bm{\gamma})(x_1) + [\varphi_2^{\bm{\mu}}]'(\bm{\gamma})(\cdot) \big \}  e^{\frac{\varphi_1^{\bm{\mu}} (x_1)+ \varphi_2^{\bm{\mu}}(\cdot) - c(x_1,\cdot)}{\varepsilon}} \Bigg \|_{\calC^s(
\calX
)} \to 0. 
\end{split}
\label{eq: equi differentiatbility}
\end{equation}
Indeed, since $\gamma_t^2 \to \gamma^2$ in $\ell^\infty (
B^s
)$, \eqref{eq: equi differentiatbility} implies that 
\[
\int t^{-1}\Big ( e^{\frac{\varphi_1^{\bm{\mu}_t} \oplus \varphi_2^{\bm{\mu}_t} - c}{\varepsilon}} - e^{\frac{\varphi_1^{\bm{\mu}} \oplus \varphi_2^{\bm{\mu}} - c}{\varepsilon}} \Big ) \, d\gamma_t^{2} \to \varepsilon^{-1}\int \big \{ [\varphi_1^{\bm{\mu}}]'(\bm{\gamma}) \oplus [\varphi_2^{\bm{\mu}}]'(\bm{\gamma})\big \}  e^{\frac{\varphi_1^{\bm{\mu}} \oplus \varphi_2^{\bm{\mu}} - c}{\varepsilon}} \, d\gamma^2
\]
in $\calC(
\calX
)$, yielding the desired result in view of \cref{rem: tangent cone}.

Observe that,  in general, for $f,g \in \calC^s(
\calX
)$, it holds that \[
\| fg \|_{\calC^s(
\calX
)} \le K \| f \|_{\calC^s(
\calX
)} \| g \|_{\calC^s(
\calX
)}
\]
for some constant $K$ that depends only on $s,d$. 
Now, since
\[
\sup_{x_1 \in 
\calX
} \Big \| e^{\frac{\varphi_1^{\bm{\mu}} (x_1)+ \varphi_2^{\bm{\mu}}(\cdot) - c(x_1,\cdot)}{\varepsilon}} \Big \|_{\calC^s(
\calX
)} < \infty
\]
by \cref{lem: EOT potential} (ii), the left-hand side of \eqref{eq: equi differentiatbility} is bounded by 
\begin{equation}
\sup_{x_1 \in 
\calX
} \Bigg \|
t^{-1} \Big \{ e^{\frac{\varphi_1^{\bm{\mu}_t}(x_1)-\varphi_1^{\bm{\mu}}(x_1) + \varphi_2^{\bm{\mu}_t}(\cdot)-\varphi_2^{\bm{\mu}}(\cdot)}{\varepsilon}}  -1 \Big \}
- \varepsilon^{-1}\big \{ [\varphi_1^{\bm{\mu}}]'(\bm{\gamma})(x_1) + [\varphi_2^{\bm{\mu}}]'(\bm{\gamma})(\cdot) \big \}  \Bigg \|_{\calC^s(
\calX
)}
\label{eq: equi differentiability 2}
\end{equation}
up to a constant independent of $t$.  It is not difficult to see that \eqref{eq: equi differentiability 2} converges to zero when $\| \cdot \|_{\calC^s(
\calX
)}$  is replaced by $\| \cdot \|_{\infty,
\calX
}$. We shall show that for every multi-index $k \in \NN_0^d$ with $0 < |k| \le s$,
\begin{equation}
\sup_{x_1 \in 
\calX
}\Bigg \|
t^{-1} D_{x_2}^k\big( e^{\frac{\varphi_1^{\bm{\mu}_t}(x_1)-\varphi_1^{\bm{\mu}}(x_1) + \varphi_2^{\bm{\mu}_t}(\cdot)-\varphi_2^{\bm{\mu}}(\cdot)}{\varepsilon}}  \big)
- \varepsilon^{-1}D_{x_2}^k [\varphi_2^{\bm{\mu}}]'(\bm{\gamma})(\cdot)   \Bigg \|_{\infty,
\calX
} \to 0.
\label{eq: convergence of derivative}
\end{equation}
By the multivariate Fa\`{a} di Bruno formula (cf. Theorem 2.1 in \cite{constantine1996multivariate}), 
\[
t^{-1} D_{x_2}^k\Big( e^{\frac{\varphi_1^{\bm{\mu}_t}(x_1)-\varphi_1^{\bm{\mu}}(x_1) + \varphi_2^{\bm{\mu}_t}(x_2)-\varphi_2^{\bm{\mu}}(x_2)}{\varepsilon}} \Big)
\]
can be expressed as a linear combination of functions of the form
\begin{equation}
e^{\frac{\varphi_1^{\bm{\mu}_t}(x_1)-\varphi_1^{\bm{\mu}}(x_1) + \varphi_2^{\bm{\mu}_t}(x_2)-\varphi_2^{\bm{\mu}}(x_2)}{\varepsilon}}  \times t^{-1}\prod_{j=1}^J \Big[D_{x_2}^{\ell_j}\big( \varphi_2^{\bm{\mu}_t}(x_2)-\varphi_2^{\bm{\mu}}(x_2) \big) \Big]^{m_j},
\label{eq: equi differentiability 3}
\end{equation}
where $1 \le J \le |k|$, and $m_1,\dots,m_J \in \NN$ and $\ell_1,\dots,\ell_J \in \NN_0^d \setminus \{ 0 \}$ are such that $m_1\ell_1+\dots+m_J\ell_J = k$.
The coefficient for the leading term with $J=1$ and $m_1=1$ is $\varepsilon^{-1}$. 
Except for the leading term, the order of \eqref{eq: equi differentiability 3} is $O(t)$ as $t \downarrow 0$ in $\calC(
\calX\times \calX
)$, so we arrive at the following expansion in $\calC(
\calX\times \calX
)$:
\[
\begin{split}
&t^{-1} D_{x_2}^k\Big( e^{\frac{\varphi_1^{\bm{\mu}_t}(x_1)-\varphi_1^{\bm{\mu}}(x_1) + \varphi_2^{\bm{\mu}_t}(x_2)-\varphi_2^{\bm{\mu}}(x_2)}{\varepsilon}} \Big) \\
&= \varepsilon^{-1}e^{\frac{\varphi_1^{\bm{\mu}_t}(x_1)-\varphi_1^{\bm{\mu}}(x_1) + \varphi_2^{\bm{\mu}_t}(x_2)-\varphi_2^{\bm{\mu}}(x_2)}{\varepsilon}}  \times t^{-1} D_{x_2}^{k}\big( \varphi_2^{\bm{\mu}_t}(x_2)-\varphi_2^{\bm{\mu}}(x_2) \big) + o(1).
\end{split}
\]
The right-hand side converges to $\varepsilon^{-1}D_{x_2}^k [\varphi_2^{\bm{\mu}}]'(\bm{\gamma})(x_2)$ in $\calC(
\calX\times\calX
)$ as $t^{-1}(\varphi_i^{\bm{\mu}_t} - \varphi_i^{\bm{\mu}}) \to [\varphi_i^{\bm{\mu}}]'(\bm{\gamma})$ in $\calC^s(
\calX
)$ for $i=1,2$, which leads to \eqref{eq: convergence of derivative} as desired. 

Therefore, \cref{lem: second order H derivative} guarantees that the limit 
\[
\lim_{t \downarrow 0} \frac{\bm{\varphi}^{\bm{\mu}_t}\vert_{S} - \bm{\varphi}^{\bm{\mu}}\vert_{S} - t[\bm{\varphi}^{\bm{\mu}}\vert_{S}]'(\bm{\gamma}_t)}{t^2/2}
\]
exists in $\DD$. By our construction, the limit depends on $\bm{\gamma}$ but not on the choice of sequence $\bm{\mu}_t$, so denote the limit by $[\bm{\varphi}^{\bm{\mu}}\vert_{S}]''(\bm{\gamma})$. 

\underline{Step 2}. Next, 
 we leverage twice differentiability of $\bm\nu\mapsto \bm\varphi^{\bm\nu}\vert_{S}$ to establish the second derivative of $\bm\nu\mapsto \bm\varphi^{\bm\nu}$ using  \eqref{eq:PotentialRepresentation}. As in the proof of \cref{lem: H derivative potential}, we only deal with the first potential.
 Recall the notation $\xi^\nu$ and $\Upsilon$ that appeared in the proof of \cref{lem: H derivative potential} and observe that $\varphi_1^{\bm\nu}=-\varepsilon\log(\Upsilon(\bm\nu))$. Precisely, $\varphi_1^{\bm\nu}$ agrees with the composition of the maps $f \in \calC_{+}(\calX) \mapsto -\varepsilon \log (f) \in \calC(\calX)$ with $\calC_{+}(\calX) = \{ f \in \calC(\calX) : f > 0 \}$ and $\Upsilon: \calP_{\mu^1} \times \calP_{\mu^2} \to \calC_+(\calX)$. The former map is twice Hadamard (indeed Fr\'{e}chet) differentiable and its domain $\calC_+(\calX)$ is open in $\calC(\calX)$, so the tangent cone $\mathfrak{T}_{\calC_+(\calX)}(f)$ agrees with $\calC(\calX)$ for every $f \in \calC_+(\calX)$.
From the second-order chain rule for Hadamard  differentiable maps (see \cref{lem:SecondOrderChainRule}), it suffices to establish second-order Hadamard differentiability of $\Upsilon$ at $\bm\mu$ tangentially to $\overline{\mathcal M_{\mu^1}}^{\ell^{\infty}(B^s)}\times \overline{\mathcal M_{\mu^2}}^{\ell^{\infty}(B^s)}$. As such, with $\bm\mu_t$ and $\bm\gamma_t$ as above, recalling the expression of the derivative of $\Upsilon$ from \eqref{eq:IntegratedXiDerivative} and that $[\varphi_2^{\bm\mu}\vert_{S_2}]'(\bm\gamma_t)=[\varphi_2^{\bm\mu}]'(\bm\gamma_t)$ on $S_2$ 
(see \cref{rmk:CompatibilityPotentialDerivatives}), we have
\begin{align*}
&\frac{\Upsilon(\bm\mu_t)-\Upsilon(\bm\mu)-t\Upsilon'_{\bm\mu}(\bm\gamma_t)}{t^2/2}
\\
&=(2/t^2)\int \left \{\xi^{\bm\mu_t}(\cdot,x_2)-\xi^{\bm\mu}(\cdot,x_2)-t\varepsilon^{-1}\xi^{\bm\mu}(\cdot,x_2)[\varphi_2^{\bm\mu}]'(\bm\gamma_t)(x_2)\right\}\, d\mu^2(x_2)
\\
&\quad +(2/t)\int \left \{\xi^{\bm\mu_t}(\cdot,x_2)-\xi^{\bm\mu}(\cdot,x_2)\right \}\, d\gamma^2_t(x_2),
\end{align*}
As $t\downarrow 0$, the first term on the right-hand side  converges to the second-order Hadamard derivative of the map $\bm\nu\in\calP_{\mu^1} \times \calP_{\mu^2}\mapsto \int \xi^{\bm\nu}(\cdot,x_2)\, d\mu^2(x_2)\in\calC(\calX)$ at $\bm\mu$, which agrees with
\begin{align*}
    \int \left \{  \varepsilon^{-2}\xi^{\bm\mu}(\cdot,x_2)\left([\varphi_2^{\bm\mu}]'(\bm\gamma)(x_2)\right)^2+\varepsilon^{-1}\xi^{\bm\mu}(\cdot,x_2)[\varphi_2^{\bm\mu}\vert_{S_2}]''(\bm\gamma)(x_2) \right \}\, d\mu^2(x_2)
\end{align*} 
by the chain rule and twice differentiability of the restricted potentials.

As $(x_1,x_2)\in\calX\times\calX\mapsto e^{\frac{\varphi_2^{\bm\nu}(x_2)-c(x_1,x_2)}{\varepsilon}}$ extends $\xi^{\bm\nu}$, the second term satisfies  
\begin{equation}
\label{eq:AuxBoundXi}
\begin{aligned}
    &\sup_{x_1\in\calX}\left|t^{-1}\int \big \{ \xi^{\bm\mu_t}(x_1,\cdot)-\xi^{\bm\mu}(x_1,\cdot)\big \}\, d\gamma^2_t-\varepsilon^{-1}\int\xi^{\bm\mu}(x_1,\cdot)[\varphi_2^{\bm\mu}]'(\bm\gamma)(\cdot)\, d\gamma^2\right|
    \\
    &\leq
    \sup_{x_1\in\calX} \Big\|e^{\frac{\varphi_2^{\bm\mu}-c(x_1,\cdot)}{\varepsilon}}\Big(t^{-1}\Big(e^{\frac{\varphi_2^{\bm\mu_t}-\varphi_2^{\bm\mu}}{\varepsilon}}-1\Big)-\varepsilon^{-1}[\varphi_2^{\bm\mu}]'(\bm\gamma)(\cdot)\Big)\Big\|_{\calC^s(\calX)}\|\gamma_t^2\|_{\infty,B^s}
    \\
    &\quad +\sup_{x_1\in\calX} \Big\|\varepsilon^{-1}e^{\frac{\varphi_2^{\bm\mu}-c(x_1,\cdot)}{\varepsilon}}[\varphi_2^{\bm\mu}]'(\bm\gamma)(\cdot)\Big\|_{\calC^s(\calX)}\|\gamma_t^2-\gamma^2\|_{\infty,B^s}.  
\end{aligned}
\end{equation}
As in Step 1, one can show that the right-hand side of \eqref{eq:AuxBoundXi} converges to $0$ as $t\downarrow 0$.

Consequently, $\Upsilon$ is twice Hadamard differentiable at $\bm\mu$ with derivative
\begin{align*}
 \Upsilon''_{\bm\mu}(\bm\gamma)&=\int \left \{ \varepsilon^{-2}\xi^{\bm\mu}(\cdot,x_2)\left([\varphi_2\vert_{S_2}^{\bm\mu}]'(\bm\gamma)(x_2)\right)^2+\varepsilon^{-1}\xi^{\bm\mu}(\cdot,x_2)[\varphi_2\vert_{S_2}^{\bm\mu}]''(\bm\gamma)(x_2) \right \}\, d\mu^2(x_2)
 \\
 &\quad +2\varepsilon^{-1}\int\xi^{\bm\mu}(x_1,\cdot)[\varphi_2\vert_{S_2}^{\bm\mu}]'(\bm\gamma)(\cdot)\, d\gamma^2, \quad \bm{\gamma}\in\overline{\mathcal M_{\mu^1}}^{\ell^{\infty}(B^s)}\times \overline{\mathcal M_{\mu^2}}^{\ell^{\infty}(B^s)}, 
\end{align*}
and hence
    $[\varphi_1^{\bm\mu}]''(\bm\gamma)=\varepsilon e^{\frac{2\varphi^{\bm\mu}_1}{\varepsilon}}\big(\Upsilon'_{\bm\mu}(\bm\gamma)\big)^2-\varepsilon e^{\frac{\varphi_1^{\bm\mu}}{\varepsilon}}\Upsilon''_{\bm\mu}(\bm\gamma)$ by the chain rule.
 Finally, continuity and positive homogeneity (of degree $2$) of $[\bm{\varphi}^{\bm{\mu}}]''$ follow from the construction.
\qed

\subsection{Proof of Lemma \ref{lem: H derivative Sinkhorn}}
We divide the proof into two steps. Since the Sinkhorn divergence is invariant w.r.t. the choice of reference points and $\gamma^i (a) = 0$ for every constant $a \in \R$, we may assume without loss of generality that $x_1^\circ=x_2^\circ$.

\underline{Step 1}. The map $\bm{\nu} = (\nu^1,\nu^2) \mapsto \mathsf{S}_{c,\varepsilon}(\nu^1,\nu^2), 
\calP_{\mu^1} \times \calP_{\mu^2}
\subset \ell^\infty(B^s) \times \ell^\infty(B^s) \to \R$ is Hadamard differentiable at $\bm{\mu}$  tangentially to $\overline{\calM_{\mu^1}}^{\ell^\infty(B^s)} \times \overline{\calM_{\mu^2}}^{\ell^\infty(B^s)}$ with derivative 
\[
\big[\mathsf{S}^{\bm{\mu}}_{c,\varepsilon}\big]'(\bm{\gamma}) = \int \varphi_1^{\bm{\mu}}\, d\gamma^1 + \int \varphi_2^{\bm{\mu}} \, d\gamma^2.
\]
Given regularity of EOT potentials (\cref{lem: EOT potential}), the proof follows similarly to the proof (or its argument) of Theorem 7 in \cite{goldfeld2022statistical}, with small modifications. To avoid repetitions, we omit the details. 

\underline{Step 2}. The conclusion of the lemma follows by Step 1 and noting that $\varphi_1^{(\mu^i,\mu^i)} = \varphi_2^{(\mu^i,\mu^i)}$ by symmetry of the cost function. 
\qed

\subsection{Proof of \cref{thm: second H derivative Sinkhorn}}
As before, we may assume without loss of generality that $x_1^\circ=x_2^\circ$. Using \eqref{eq: Sinkhorn dual}, we can expand $\bar{\mathsf{S}}_{c,\varepsilon}(\mu_t^1,\mu_t^2)$ as
\[
\begin{split}
\bar{\mathsf{S}}_{c,\varepsilon}(\mu_t^1,\mu_t^2) &= \int (\varphi_1^{(\mu_t^1,\mu_t^2)}-\varphi_{1}^{(\mu_t^1,\mu_t^1)}) \, d\mu_t^1 + \int (\varphi_2^{(\mu_t^1,\mu_t^2)}-\varphi_{2}^{(\mu_t^2,\mu_t^2)}) \, d\mu_t^2 \\
&=\int (\varphi_1^{(\mu_t^1,\mu_t^2)}-\varphi_{1}^{(\mu_t^1,\mu_t^1)}) \, d\mu  + \int (\varphi_2^{(\mu_t^1,\mu_t^2)}-\varphi_{2}^{(\mu_t^2,\mu_t^2)}) \, d\mu \\
&\quad+ t \int (\varphi_1^{(\mu_t^1,\mu_t^2)}-\varphi_{1}^{(\mu_t^1,\mu_t^1)}) \, d\gamma_t^1 +t \int (\varphi_2^{(\mu_t^1,\mu_t^2)}-\varphi_{2}^{(\mu_t^2,\mu_t^2)}) \, d\gamma_t^2. 
\end{split}
\]
By \cref{thm: H derivative potential}, we have 
\[
\frac{\varphi_i^{(\mu_t^1,\mu_t^2)} - \varphi_i^{(\mu,\mu)}}{t} \to [\varphi_i^{(\mu,\mu)}]'(\gamma^1,\gamma^2) \quad \text{in} \ \calC^s (\calX)
\]
as $t \downarrow 0$, while
\[
\frac{\varphi_i^{(\mu_t^i,\mu_t^i)} - \varphi_i^{(\mu,\mu)}}{t} \to [\varphi_i^{(\mu,\mu)}]'(\gamma^i,\gamma^i) \quad \text{in} \ \calC^s (\calX).
\]
Since $\gamma_t^i \to \gamma^i$ in $\ell^\infty(B^s)$ as $t \downarrow 0$, we have 
\[
\int (\varphi_i^{(\mu_t^1,\mu_t^2)}-\varphi_{i}^{(\mu_t^i,\mu_t^i)}) \, d\gamma_t^i = t\int [\varphi_i^{(\mu,\mu)}]'(\gamma^1-\gamma^i,\gamma^2-\gamma^i) \, d\gamma^i + o(t). 
\]
On the other hand, by \cref{thm: second order H derivative}, we have
\[
\begin{split}
&\frac{\varphi_1^{(\mu_t^1,\mu_t^2)}-\varphi_{1}^{(\mu_t^1,\mu_t^1)} - t[\varphi_1^{(\mu,\mu)}]'(0,\gamma_t^2-\gamma_t^1)}{t^2/2} \\
&\to [\varphi_1^{(\mu,\mu)}]''(\gamma^1,\gamma^2) - [\varphi_1^{(\mu,\mu)}]''(\gamma^1,\gamma^1) \quad \text{in} \ \calC(\calX) \times \calC(\calX). 
\end{split}
\]
A similar expansion holds for $\varphi_2^{(\mu_t^1,\mu_t^2)}-\varphi_{2}^{(\mu_t^2,\mu_t^2)}$. Hence, we have
\begin{equation}
\begin{split}
&\int (\varphi_1^{(\mu_t^1,\mu_t^2)}-\varphi_{1}^{(\mu_t^1,\mu_t^1)}) \, d\mu  + \int (\varphi_2^{(\mu_t^1,\mu_t^2)}-\varphi_{2}^{(\mu_t^2,\mu_t^2)}) \, d\mu  \\
&=t \int \big ([\varphi_1^{(\mu,\mu)}]'(0,\gamma_t^2-\gamma_t^1) + [\varphi_2^{(\mu,\mu)}]'(\gamma_t^1-\gamma_t^2,0) \big) \, d\mu \\
&\quad + \frac{t^2}{2}\int \Big (\sum_{i=1}^2\big( [\varphi_i^{(\mu,\mu)}]''(\gamma^1,\gamma^2) - [\varphi_i^{(\mu,\mu)}]''(\gamma^i,\gamma^i)\big) \Big) \, d\mu + o(t^2). 
\end{split}
\label{eq: expansion}
\end{equation}
As the cost function is symmetric, we have 
\[
[\varphi_2^{(\mu,\mu)}]'(\gamma_t^1-\gamma_t^2,0) = [\varphi_1^{(\mu,\mu)}]'(0,\gamma_t^1-\gamma_t^2) = - [\varphi_1^{(\mu,\mu)}]'(0,\gamma_t^2-\gamma_t^1),
\]
so that the first term on the right-hand side of \eqref{eq: expansion} vanishes. Conclude that 
\[
\begin{split}
\lim_{t \downarrow 0} \frac{\bar{\mathsf{S}}_{c,\varepsilon}(\mu_t^1,\mu_t^2)}{t^2/2} &= \int \Big (\sum_{i=1}^2\big( [\varphi_i^{(\mu,\mu)}]''(\gamma^1,\gamma^2) - [\varphi_i^{(\mu,\mu)}]''(\gamma^i,\gamma^i)\big) \Big)\, d\mu \\
&\quad + 2 \sum_{i=1}^2\int [\varphi_i^{(\mu,\mu)}]'(\gamma^1-\gamma^i,\gamma^2-\gamma^i) \, d\gamma^i =: \Delta_\mu (\bm{\gamma}).
\end{split}
\]
Continuity and positive homogeneity (of degree $2$) of $\Delta_\mu$ follows immediately (or from its construction).
\qed

\subsection{Proofs for \cref{sec: main results} (except \cref{prop: independence test})}
In what follows, let 
\[
\mathfrak{M}_{\mu^i} = \Big \{ gd\mu^i : \text{$g: 
\calX
\to \R$ is bounded and measurable with $\mu^i$-mean zero} \Big \},
\]
where $gd\mu^i$ should be understood as a signed measure $A \mapsto \int_{A} g d\mu^i$.

For a (generic) probability measure $\mu$  and a function class $\calF \subset L^2 (\mu)$, a stochastic process $G = (G(f))_{f \in \calF}$ is called a $\mu$-Brownian bridge if it is a Gaussian process with mean zero and covariance function $\E[G(f)G(g)] = \Cov_\mu(f,g)$; furthermore, if $G$ is a tight measurable map into $\ell^\infty (\calF)$, then we call $G$ a tight $\mu$-Brownian bridge in $\ell^\infty(\calF)$. Recall that a (zero-mean) Gaussian process that is a tight measurable map into $\ell^\infty(\calF)$ is an $\ell^\infty(\calF)$-valued Gaussian random variable (with mean zero) in the Banach space sense; see Lemma 3.9.8 in \cite{van1996weak}.
\begin{lemma}
\label{lem: weak convergence}
Let $s$ be a positive integer with $s > d/2$. Then, for every $\bm{\mu} = (\mu^1,\mu^2) \in \calP(
\calX
) \times \calP(
\calX
)$, we have 
\begin{equation}
\sqrt{n}(\hat{\bm{\mu}}_n - \bm{\mu}) \stackrel{d}{\to} \GG^{\bm{\mu}} = \big(\GG_1^{\mu^1},\GG_2^{\mu^2}\big) \quad \text{in} \ \ell^\infty(
B^s
) \times \ell^\infty(
B^s
),
\label{eq: weak convergence}
\end{equation}
where $\GG_1^{\mu^1}$ and $\GG_2^{\mu^2}$ are independent, tight $\mu^1$- and $\mu^2$-Brownian bridges in $\ell^\infty(
B^s
)$ and $\ell^\infty(
B^s
)$, respectively. Furthermore, $\supp(\GG^{\bm{\mu}}) \subset  \overline{\mathfrak{M}_{\mu^1}}^{\ell^\infty(
B^s
)} \times \overline{\mathfrak{M}_{\mu^2}}^{\ell^\infty(
B^s
)}$. 
\end{lemma}

\begin{proof}[Proof of \cref{lem: weak convergence}]
The set $B^s$
is $\mu^i$-Donsker by Theorem 2.7.1 in \cite{van1996weak}. 
Since the samples are independent, by Example 1.4.6 in \cite{van1996weak} (combined with Lemma 3.2.4 in \cite{dudley2014uniform}
concerning measurable covers), we obtain the weak convergence result \eqref{eq: weak convergence}. The second claim follows by \cref{lem: support}.
\end{proof}

The following lemma will be used to prove the second claim of Theorem \ref{thm: limit theorem Sinkhorn null}.

\begin{lemma}
\label{lem: support Sinkhorn}
Consider the setting of \cref{thm: second H derivative Sinkhorn} and assume that $\bar{\mathsf{S}}_{c,\varepsilon}$ is nonnegative. Let $W$ be a tight random variable with values in $\ell^\infty(B^s) \times \ell^\infty(B^s)$ whose support is a cone $\ne \{ 0 \}$. Then, unless $\Delta_\mu$ is identically zero, the support of $\Delta_\mu(W)$ agrees with $[0,\infty)$. 
\end{lemma}

\begin{proof}[Proof of \cref{lem: support Sinkhorn}]
Consider the restriction of $\Delta_\mu$ on $\supp(W)$, which we denote by the same symbol. The restriction is still continuous and positively homogeneous of degree $2$ (the latter follows as $\supp(W)$ is a cone).
The functional $\Delta_\mu$ is nonnegative by construction. 
By positive homogeneity, $\Delta_\mu$ is either identically zero or onto $[0,\infty)$. 
In our setting, $\Delta_\mu$ is onto $[0,\infty)$. For every open interval $(a,b) \subset [0,\infty)$, the inverse image $\Delta_\mu^{-1}((a,b))$ is nonempty and open in $\supp(W)$, and hence by the definition of support, we have $\Prob (\Delta_\mu(W) \in (a,b))> 0$, which implies that $\supp(\Delta_\mu(W)) = [0,\infty)$.
\end{proof}

We are now ready to prove \cref{thm: limit theorem EOT potential}, \cref{cor: limit theorem EOT map}, \cref{prop: limit theorem Sinkhorn}, and \cref{thm: limit theorem Sinkhorn null}.

\begin{proof}[Proof of \cref{thm: limit theorem EOT potential}]
The theorem follows from the Hadamard differentiability result (\cref{thm: H derivative potential}) combined with the functional delta method. 
Let $s > d/2$, so that the weak convergence \eqref{eq: weak convergence} holds.
Since $\sqrt{n}(\hat{\bm{\mu}}_n - \bm{\mu}) \in \calP_{\mu^1} \times \calP_{\mu^2}$ a.s. (as $\supp(\hat{\mu}_n^i) \subset \supp(\mu^i)$ a.s.) and $\supp(\GG^{\bm{\mu}}) \subset  \overline{\mathfrak{M}_{\mu^1}}^{\ell^\infty(
B^s)} \times \overline{\mathfrak{M}_{\mu^2}}^{\ell^\infty(
B^s)} \subset \overline{\calM_{\mu^1}}^{\ell^\infty(
B^s)} \times \overline{\calM_{\mu^2}}^{\ell^\infty(
B^s)}$  (or the inclusion $\supp(\GG^{\bm{\mu}}) \subset \overline{\calM_{\mu^1}}^{\ell^\infty(
B^s)} \times \overline{\calM_{\mu^2}}^{\ell^\infty(
B^s)}$
follows from the portmanteau theorem), we may apply the functional delta method (\cref{lem: functional delta method}) to conclude that 
\[
\sqrt{n}\big (\hat{\bm{\varphi}}_n -\bm{\varphi}^{\bm{\mu}} \big) \stackrel{d}{\to} [\bm{\varphi}^{\bm{\mu}}]'(\GG^{\bm{\mu}}) \quad \text{in} \ \calC^s(
\calX
) \times \calC^s(
\calX
). 
\]
Also, since $\overline{\mathfrak{M}_{\mu^1}}^{\ell^\infty(
B^s
)} \times \overline{\mathfrak{M}_{\mu^2}}^{\ell^\infty(
B^s
)}$ is a closed subspace of $\ell^\infty(
B^s
) \times \ell^\infty(
B^s
)$ and the restriction of $[\bm{\varphi}^{\bm{\mu}}]'$ to $\overline{\mathfrak{M}_{\mu^1}}^{\ell^\infty(
B^s
)} \times \overline{\mathfrak{M}_{\mu^2}}^{\ell^\infty(
B^s
)}$ is a continuous linear operator, we see that $[\bm{\varphi}^{\bm{\mu}}]'(\GG^{\bm{\mu}})$ is a zero-mean Gaussian random variable with values in $\calC^s(
\calX
) \times \calC^s(
\calX
)$. The result for the $s \le d/2$ case follows by the fact the inclusion map $f \mapsto f, \calC^s(
\calX
) \to \calC^{s'}(
\calX
)$ with $s' < s$ is continuous.
\end{proof}

\begin{proof}[Proof of \cref{cor: limit theorem EOT map}]
Since the map $f \mapsto \nabla f, \calC^s(
\calX
) \to \calC^{s-1}(
\calX
;\R^d)$ is continuous and linear, we have $\sqrt{n}(\hat{T}_n - T^{\bm{\mu}}) = - \nabla \sqrt{n}(\hat{\varphi}^n_{1} - \varphi_1^{\bm{\mu}}) \stackrel{d}{\to} -\nabla G_1^{\bm{\mu}}$ in $ \calC^{s-1}(
\calX
;\R^d)$ and the limit $-\nabla G_1^{\bm{\mu}}$ is again a zero-mean Gaussian variable. 
\end{proof}

\begin{proof}[Proof of \cref{prop: limit theorem Sinkhorn}]
Given the weak convergence and Hadamard differentiablity results (Lemmas \ref{lem: weak convergence} and \ref{lem: H derivative Sinkhorn}), the first claim follows by applying the functional delta method.

For the second claim, 
assume $\bar{\mathsf{S}}_{c,\varepsilon}(\mu^1,\mu^2) \ne 0$ and $\supp (\mu^1) \cap \supp(\mu^2) \ne \varnothing$. For simplicity of notation, assume without loss of generality that $\varepsilon=1$ and $x_1^\circ = x_2^\circ$, and define the shorthands $\varphi_i^{\bm{\mu}} = \varphi_i$ and $\varphi_i^{(\mu^i,\mu^i)} = \psi_i$, for $i=1,2$. Suppose on the contrary that $\sigma_{\bm{\mu}}^2 = 0$, which entails $\Var_{\mu^i}(\varphi_i - \psi_i) = 0$, so $\psi_i = \varphi_i + a_i$ $\mu^i$-a.e. for some $a_i \in \R$ for $i=1,2$.
The Schr\"{o}dinger system \eqref{eq: FOC} implies that 
\[
1 = \int e^{\varphi_1 \oplus \varphi_2 -c} \, d\mu^2 = e^{-a_1-a_2} \int e^{\psi_1\oplus \psi_2 -c} \, d\mu^2 = e^{-a_1-a_2 + \psi_1 - \psi_2}
\]
where we used the fact that $\varphi_1^{(\mu^2,\mu^2)} = \varphi_2^{(\mu^2,\mu^2)} = \psi_2$. The equality above holds $\mu^1$-a.e., so $\psi_1 - \psi_2 = a_1+a_2$ $\mu^1$-a.e. By symmetry, we also have $\psi_2-\psi_1 = a_1+a_2$ $\mu^2$-a.e.
Since $\psi_1$ and $\psi_2$ are continuous, we have $\psi_1-\psi_2=a_1+a_2$ on $\supp (\mu^1)$ and $\psi_2-\psi_1=a_1+a_2$ on $\supp (\mu^2)$, and as $\supp (\mu^1) \cap \supp(\mu^2) \ne \varnothing$, we must have $a_1 + a_2 = 0$. However, the duality formula \eqref{eq: Sinkhorn dual} then entails $\bar{\mathsf{S}}_{c,\varepsilon} (\mu^1,\mu^2) = 0$, which is a contradiction. 
\end{proof}

\begin{proof}[Proof of \cref{thm: limit theorem Sinkhorn null}]
Given the second-order Hadamard differentiability result (\cref{thm: second H derivative Sinkhorn}), the first claim of the theorem follows by applying the second-order functional delta method (\cref{lem: second order functional delta method}), 
$n \bar{\mathsf{S}}_{c,\varepsilon} (\hat{\mu}_n^1,\hat{\mu}_n^2) \stackrel{d}{\to} \Delta_\mu(\GG^{(\mu,\mu)})/2$.
The second claim follows by \cref{lem: support Sinkhorn}, upon noting that $\supp (\GG^{(\mu,\mu)})$ is a vector subspace of $\ell^\infty(B^s) \times \ell^\infty(B^s)$ (cf. the proof of \cref{lem: support} or Lemma 5.1 in \cite{van2008reproducing}).
\end{proof}

\subsection{Proof of \cref{prop: independence test}}
The proof uses techniques from $U$-processes. We refer to \cite{PeGi1999} as an excellent reference on $U$-processes. 
Assume $\pi = \pi^V \otimes \pi^W$.
For notational convenience, set $\hat{\pi}_n^\circ = \hat{\pi}_n^V \otimes \hat{\pi}_n^W$.
Let $B^s$ be the unit ball in $\calC^s(\calX)$ with $s > 2d$ (recall that $d=d_1+d_2$ and $\calX = \mathcal V \times \mathcal W$).
We will derive a joint limit distribution for $\sqrt{n}(\hat{\pi}_n - \pi)$ and $\sqrt{n}(\hat{\pi}_n^\circ - \pi)$ in $\ell^\infty (B^s) \times \ell^\infty (B^s)$. 
As in \cite{albert2015bootstrap},  define for $f \in B^s$, 
\begin{equation}
\label{eq: kernel_symmetrization}
\begin{split}
    h_f ( x_1, x_2 ) &= h_f((v_1,w_1),(v_2,w_2)) \\
    &:= f(v_1,w_1) + f(v_2,w_2) - f(v_1,w_2) - f(v_2,w_1).
    \end{split}
\end{equation}
The function $h_f$ is symmetric  with $\| h_f \|_{\calC^s(\calX \times \calX)} \le 4$. As $V_i$ and $W_i$ are independent, the mean of $h_f(X_1,X_2)$ is zero: $\E[h_f(X_1,X_2)] = 0$. Consider the $U$-statistic with kernel $h_f$:
\[
U_n(h_f) = \frac{1}{n(n-1)} \sum_{1 \le i\neq j \le n} h_f(X_i, X_j).
\]
Then, keeping in mind $X_i = (V_i,W_i)$, we can expand $U_n(h_f)$ as 
\begin{align*}
    U_n(h_f) &= \frac{2}{n} \sum_{i=1}^n f(X_i) - \frac{2}{n(n-1)} \sum_{1 \le i \neq j \le n} f(V_i,W_j)\\
    &=  \frac{2}{n} \sum_{i=1}^n f(X_i) - \frac{2}{n^2} \sum_{i,j=1}^n f(V_i,W_j) \\
    &\quad + \underbrace{\frac{2}{n^3 - n^2} \sum_{1 \le i \neq j \le n} f(V_i,W_j) + \frac{2}{n^2} \sum_{i=1}^n f(X_i)}_{=: A_n(f)} \\
    &= 2\big( \hat{\pi}_n (f) -  \hat{\pi}_n^\circ (f) \big) + A_n(f),
\end{align*}
that is,
\[
\hat{\pi}_n^\circ (f) = \hat{\pi}_n (f) - \frac{1}{2} U_n (h_f) + \frac{1}{2}A_n(f). 
\]
Since $|f| \le 1$, we have $|A_n(f)| \le 4/n$. 
We will approximate $U_n(h_f)$ by the Haj\'{e}k process. Apply the Hoeffding decomposition to $h_f$:
\[
h_f^{(1)} (x) = \int h_f (x,x') \, d\pi(x'), \ h_{f}^{(2)} (x,x') = h_f(x,x') - h_{f}^{(1)}(x) - h_{f}^{(1)}(x').
\]
Then, we can decompose $U_n(h_f)$ as $U_n (h_f) = 2 \hat{\pi}_n(h_f^{(1)}) + U_n(h_f^{(2)})$.

Now, since $\{ h_f : f \in B^s \} \subset \{ h \in \calC^s (\calX \times \calX) : \| h \|_{\calC^s(\calX \times \calX)} \le 4 \}$ and the $t$-entropy number of the latter function class  w.r.t. $\| \cdot \|_{\infty,\calX \times \calX}$ is of order $t^{-2d/s}$ as $t \downarrow 0$ by Theorem 2.7.1 in \cite{van1996weak}, applying Corollary 5.6 in \cite{chen2020jackknife} with $k=2$, we have 
\[
\E[\| U_n(h_f^{(2)}) \|_{\infty,B^s}] \le O(n^{-1}) \times \int_0^1 t^{-2d/s} dt = O(n^{-1}). 
\]
Summarizing, we have $\E\big[\big \| \sqrt{n} \big(\hat{\pi}_n^{\circ}(f) - \hat{\pi}_n(f-h_f^{(1)}) \big) \big\|_{\infty,B^s}\big] \to 0$.

For notational convenience, set $\GG_n = \sqrt{n}(\hat{\pi}_n - \pi)$ and $\tilde{h}_f = f-h_f^{(1)}$. We will show that $\big ( \big(\GG_n(f)\big)_{f \in B^s}, \big(\GG_n(\tilde{h}_f) \big)_{f \in B^s} \big)$ converges in distribution a tight limit in $\ell^\infty(B^s) \times \ell^\infty(B^s)$. To this end, we will apply \cref{lem: product weak limit}. Again, by Theorem 2.7.1 in \cite{van1996weak}, $B^s$ and $\tilde{\calH} := \{ \tilde{h}_f : f \in B^s\}$ are both $\pi$-Donsker. Finite-dimensional convergence follows trivially, so by \cref{lem: product weak limit}, 
\[
\big ( \big(\GG_n(f)\big)_{f \in B^s}, \big(\GG_n(h) \big)_{h \in \tilde{\mathcal{H}}} \big) \stackrel{d}{\to} \bm{G}^\pi \quad \text{in} \ \ell^\infty (B^s) \times \ell^\infty(\tilde{\calH}) 
\]
for some tight limit $\bm{G}^\pi$.
Finally, since the map $\iota : \ell^\infty (B^s) \times \ell^\infty(\tilde{\calH}) \to \ell^\infty (B^s) \times \ell^\infty (B^s)$ defined by
\[
(\iota \bm{z})_1 (f) = z_1 (f) \quad \text{and} \quad  (\iota \bm{z})_2 (f) = z_2 (\tilde{h}_f), \quad \bm{z}=(z_1,z_2) \in \ell^\infty (B^s) \times \ell^\infty(\tilde{\calH})
\]
is continuous, we conclude that 
\[
\big ( \big(\GG_n(f)\big)_{f \in B^s}, \big(\GG_n(\tilde{h}_f) \big)_{f \in B^s} \big) \stackrel{d}{\to}  \iota \bm{G}^\pi =: \tilde{\bm{G}}^{\pi} \quad \text{in} \ \ell^\infty (B^s) \times \ell^\infty (B^s).
\]
The limit variable $\tilde{\bm{G}}^{\pi}= (\tilde G_1^\pi,\tilde G_2^\pi)$ is a two-dimensional Gaussian process with mean zero and covariance structure
\[
\begin{split}
&\Cov (\tilde G_1^\pi(f),\tilde G_1^\pi(g)) = \Cov_\pi (f,g), \ \Cov (\tilde G_2^\pi(f),\tilde G_2^\pi(g)) = \Cov_\pi (\tilde h_f,\tilde h_g), \\
&\text{and} \quad \Cov (\tilde G_1^\pi(f),\tilde G_2^\pi(g)) = \Cov_{\pi}(f,\tilde{h}_g).
\end{split}
\]
In view of the discussion above, the same limit holds for $\big (\sqrt{n}(\hat{\pi}_n-\pi),\sqrt{n}(\hat{\pi}_n^\circ-\pi) \big)$. Now, combining \cref{thm: second order H derivative} and the second-order functional delta method (\cref{lem: second order functional delta method}), we conclude that 
\[
nD_n = n \bar{\mathsf{S}}_{c,\varepsilon}(\hat{\pi}_n,\hat{\pi}_n^\circ) \stackrel{d}{\to} \frac{1}{2}\Delta_\pi (\tilde{\bm{G}}^{\pi}).
\]
Finally, the second claim of the proposition follows by \cref{lem: support Sinkhorn}, upon noting that the support of $\tilde{\bm{G}}^\pi$ is a vector subspace of $\ell^\infty(B^s) \times \ell^\infty(B^s)$; cf. Lemma 5.1 in \cite{van2008reproducing}.
This completes the proof.
\qed

\subsection{Proof of \cref{prop: asymptotic efficiency}}
(i). 
We first note that, as $hd\mu^2 \in \overline{\calM_{\mu^2}}^{\ell^\infty(B^s)}
$, by Hadamard differentiability, $\kappa_n$ satisfies 
\[
\sqrt{n}\big(\kappa_n(h) - \kappa_n(0)\big) \to \delta_{\mu^2}'(hd\mu^2) =: \dot \kappa (h). 
\]
Since $h \mapsto hd\mu^2$ is linear and continuous from $H$ into 
$\ell^\infty(B^s)$
(cf. Example 1.5.10 in \cite{van1996weak}), $\dot \kappa: H \to \BB$ is a continuous linear map, establishing regularity of the parameter sequence $\kappa_n(h)$. 

Next, we shall verify regularity of the empirical EOT map $\tilde{T}_n$. Hadamard differentiability enables finding limit distributions under local alternatives.
By the second claim of the functional delta method (\cref{lem: functional delta method}), we have
\[
\sqrt{n} \big(\tilde{T}_n - T^{\bm{\mu}}\big) - \delta_{\mu^2}'(\sqrt{n}(\hat{\mu}_n^2-\mu^2)) \to 0
\]
in probability under $P_{n,0}$. From the proof of Theorem 3.10.12 in \cite{van1996weak}, 
\[
\Big ( \sqrt{n}(\hat{\mu}_n^2 - \mu^2), \frac{dP_{n,h}}{dP_{n,0}} \Big ) \stackrel{d}{\to} (\GG_{2}^{\mu^2},\Lambda) \quad \text{in} \ 
\ell^\infty(B^s)
\times \R
\]
under $P_{n,0}$, and the law $L$ on 
$\ell^\infty(B^s)$
defined by $L(A) = \E[\ind_{A}(\GG_{2}^{\mu^2})\Lambda]$ agrees with the law of $\GG_{2}^{\mu^2} + hd\mu^2$ (note: $hd\mu^2$ should be understood as an element of 
$\ell^\infty(B^s)$
). Combining these two displays above yields that 
\[
\Big (\sqrt{n} (\tilde{T}_n - T^{\bm{\mu}}), \frac{dP_{n,h}}{dP_{n,0}} \Big ) \stackrel{d}{\to} (\delta_{\mu^2}'(\GG_{2}^{\mu^2}),\Lambda) \quad \text{in} \ 
\ell^\infty(B^s)
\times \R
\]
under $P_{n,0}$. Hence, by Le Cam's third lemma (Theorem 3.10.7 in \cite{van1996weak}), $\sqrt{n} (\tilde{T}_n - T^{\bm{\mu}}) \stackrel{d}{\to} \delta_{\mu^2}'(\GG_{2}^{\mu^2} + hd\mu^2)$ under $P_{n,h}$. Note that $\GG_{2}^{\mu^2} + hd\mu^2 \in \overline{\mathfrak{M}_{\mu^2}}^{\ell^\infty(
B^s
)}$. By linearity of the derivative, we have $\delta_{\mu^2}'(\GG_{2}^{\mu^2} + hd\mu^2) = \delta_{\mu^2}'(\GG_{2}^{\mu^2}) + \delta_{\mu^2}'(hd\mu^2)$.
Conclude that
\[
\sqrt{n} (\tilde{T}_n - T^{\bm{\mu}}) \stackrel{d}{\to} \delta_{\mu^2}'(\GG_{2}^{\mu^2}) + \delta_{\mu^2}'(hd\mu^2) \quad \text{in} \ \BB
\]
under $P_{n,h}$. This immediately implies that $\tilde{T}_n$ is regular, i.e., $\sqrt{n}(\tilde{T}_n - T^{(\mu^1,\mu_{n,h}^2)}) \stackrel{d}{\to} \delta_{\mu^2}'(\GG_{2}^{\mu^2})$ under $P_{n,h}$.

Now, in view of Theorems 3.11.2 in \cite{van1996weak}, the last claim follows by verifying that  $\delta_{\mu^2}'(\GG_2^{\mu^2}) \stackrel{d}{=} G$, where $G$ is a Gaussian random variable in $\BB$ such that for every $b^* \in \BB^*$, $b^*G$ is Gaussian with mean zero and variance 
\[
\sup_{h \in H: \| h \|_{L^2(\mu^2)}=1} (b^* \dot \kappa(h))^2 = \sup_{h \in H: \| h \|_{L^2(\mu^2)}=1} \big(b^* \delta_{\mu^2}'(hd\mu^2) \big)^2.
\]
By construction $\delta_{\mu^2,b^*}':=b^* \circ \delta_{\mu^2}'$ is a continuous linear functional on $\overline{\mathfrak{M}_{\mu^2}}^{\ell^\infty(
B^s
)} \supset \supp(\GG_2^{\mu^2})$ with $\mathfrak{M}_{\mu^2} = \{ hd\mu^2 : h \in H \}$ (\cref{lem: support}). Note that 
$\overline{\mathfrak{M}_{\mu^2}}^{\ell^\infty(B^s)}$
is a closed subspace of 
$\ell^\infty(B^s)$
. Arguing as in the proof of Proposition 2 in \cite{goldfeld2022statistical}, we obtain 
\[
\Var (b^* G) = \sup_{h \in H: \| h \|_{L^2(\mu^2)}=1} \big(\delta_{\mu^2,b^*}'(hd\mu^2) \big)^2 = \Var \big ( \delta_{\mu^2,b^*}' (\GG_2^{\mu^2}) \big ) = \Var (b^*\delta_{\mu^2}'(\GG_2^{\mu^2})). 
\]
Conclude that $\delta_{\mu^2}'(\GG_2^{\mu^2}) \stackrel{d}{=} G$.

\medskip

(ii). Given the result of Part (i), Part (ii) follows directly from Theorem 3.11.5 in \cite{van1996weak}.
\qed

\section{Discussions}
\label{sec: conclusion}

We have established limit distributions for the EOT potentials, map, and Sinkhorn divergence. The main ingredient of our proofs was Hadamard differentiability of the relevant maps. Importantly, the Hadamard differentiability results yield not only limit distributions but also bootstrap consistency and asymptotic efficiency. Regarding the Sinkhorn divergence, our main contribution is the derivation of the null limit distribution for compactly supported distributions. When the population measures agree, the first Hadamard derivative of the Sinkhorn divergence vanishes, which necessitates looking into higher-order Hadamard derivatives. For this, Hadamard differentiability of the EOT map in (sufficiently regular) H\"{o}lder spaces plays an important role. As another contribution, we have derived the null limit distribution of the Sinkhorn independence test statistic and determined the precise order of the test statistic under the null, which was not available before. We end this paper with discussions of two possible extensions. 

\subsection{Unbounded supports}
In the present paper, we have assumed that the marginals $\mu^1$ and $\mu^2$ are compactly supported, which excludes, for instance, Gaussian distributions (one exception is \cref{prop: limit theorem Sinkhorn}, where the compactness assumption can be relaxed to a sub-Gaussian condition when $c$ is quadratic; see \cite{del2022central,goldfeld2022statistical}). Indeed, the compactness assumption is essential to formulate the Hadamard differentiability results, where we first regard the dual potentials as a mapping into the H\"{o}lder space with arbitrary smoothness level (cf. \cref{lem: EOT potential} (ii)), and then embed each marginal into the topological dual of the H\"{o}lder space (cf. \cref{rem: tangent cone}). The second step is natural in view of the Schr\"{o}dinger system, since whenever $\varphi_1,\varphi_2 \in \calC^s(\calX)$, the exponentiated functions
\begin{equation}
e^{\frac{\varphi_1(x_1) + \varphi_2(\cdot) - c(x_1,\cdot)}{\varepsilon}} \quad \text{and} \quad e^{\frac{\varphi_1(\cdot) + \varphi_2(x_2) - c(\cdot,x_2)}{\varepsilon}}
\label{eq: exponential}
\end{equation}
both lie in $\calC^s(\calX)$ for every $x_1,x_2 \in \calX$. 

For a smooth cost, even when the marginals are not compactly supported, as long as they have sufficiently light tails, we may uniformly upper bound derivatives of EOT potentials by a polynomial of $1+\| \cdot \|$ (cf. \cite{mena2019statistical}), so a natural idea would be to replace the H\"{o}lder space with a certain weighted H\"{o}lder space (cf. \cite{nickl2007bracketing}). However, even when the dual potentials lie in a weighted H\"{o}lder space, the exponentiated functions in \eqref{eq: exponential} need not lie in the same weighted H\"{o}lder space, so it is nontrivial to find a suitable normed space into the marginals are embedded when the marginals are not compactly supported. Such an extension would be highly technical and hence is left for future research. It is worth mentioning that \cite{hundrieser2021limit} derive a Hadamard derivative for the EOT plan w.r.t. the marginals when the supports are \textit{countable}. They use a curvature-type condition on the cost to allow for unbounded supports (see Remark 5.9 in \cite{hundrieser2021limit}; see also \cite{harchaoui2020asymptotics} for a similar condition), but their proof technique is restricted to the countable support case where we may naturally embed the marginals into a sequence space, and does not directly extend to a more general case.

\subsection{Multimarginal EOT}
The results of the present paper extend to the multimarginal case, as long as we deal with compactly supported marginals. Given a smooth cost function $c: (\R^d)^N \to \R_+$ and marginals $\mu^1,\dots,\mu^N \in \calP(\calX)$ with $N \ge 2$ arbitrary, the multimarginal EOT problem reads as
\begin{equation}
\inf_{\pi \in \Pi (\mu^1,\dots,\mu^N)} \int c \, d\pi + \varepsilon \mathsf{D}_{\mathsf{KL}} ( \pi \| m ), \quad m = \mu^1 \otimes \cdots \otimes \mu^N,
\label{eq: EOT multi}
\end{equation}
where $\Pi(\mu^1,\dots,\mu^N)$ is the set of Borel probability measures on $\calX^N$ with marginals $\mu^1,\dots,\mu^N$. To simplify notation, we set $\varepsilon = 1$. The corresponding dual problem is
\[
\sup_{\bm{\varphi} = (\varphi_1,\dots,\varphi_N) \in \prod_{i=1}^N L^\infty (\mu^i)} \sum_{i=1}^N \int \varphi_i \, d\mu^i - \int e^{\sum_{i=1}^N\varphi_i(x_i) -c(x_1,\dots,x_N)} dm(x_1,\dots,x_N)+ 1.
\]
Bounded functions $\bm{\varphi} = (\varphi_1,\dots,\varphi_N)$ solve the dual problem if and only if they satisfy the Schr\"{o}dinger system, i.e., 
\begin{equation}
\int e^{\sum_{j =1}^N \varphi_j(x_j) -c(x_1,\dots,x_N)} dm^{-i}(x_{-i}) -1=0 \quad \text{for} \ \mu^i\text{-a.e.} \ x_i, \ i=1,\dots,N, 
\label{eq: Schrodinger multi}
\end{equation}
where $m^{-i} = \otimes_{j \ne i} \mu^j$ and $x_{-i}=(x_1,\dots,x_{i-1},x_{i+1},\dots,x_N)$. See \cite{carlier2020differential}. 
Theorem 4.3 in \cite{carlier2020differential} shows that the Schr\"{o}dinger system \eqref{eq: Schrodinger multi} admits a solution $\bm{\varphi} = (\varphi_1,\dots,\varphi_N) \in \prod_{i=1}^N L^\infty (\mu^i)$, which is unique in the sense that if $\tilde {\bm{\varphi}} = (\tilde \varphi_1,\dots,\tilde \varphi_N) \in \prod_{i=1}^N L^\infty (\mu^i)$ is another solution to the Schr\"{o}dinger system \eqref{eq: Schrodinger multi}, then $\sum_{i=1}^N \varphi(x_i) = \sum_{i=1}^N \tilde \varphi_i(x_i)$ for $m$-a.e. $(x_1,\dots,x_N)$, i.e., there exist constants $a_1,\dots,a_N$ that sum to zero such that $(\tilde{\varphi}_1,\dots,\tilde{\varphi}_N) = (\varphi_1+a_1,\dots,\varphi_N+a_N)$ for $m$-a.e. $(x_1,\dots,x_N)$. We call $\bm{\varphi}$ EOT potentials. Then, the (unique) optimal solution to the multimarginal EOT problem \eqref{eq: EOT multi} is given by
\[
d\pi^{\star} (x_1,\dots,x_N)=e^{\sum_{i=1}^N\varphi_i(x_i) -c(x_1,\dots,x_N)} dm(x_1,\dots,x_N). 
\]

Given these preparations, it is not difficult to see that \cref{lem: EOT potential} naturally extends to the multimarginal setting. For simplicity, we choose a common reference point $x^\circ \in \calX$. 
Then, for every $\bm{\mu} = (\mu^1,\dots,\mu^N) \in \prod_{i=1}^N \calP(\calX)$, there exists a unique set of functions $\bm{\varphi}^{\bm{\mu}} = (\varphi_1^{\bm{\mu}},\dots,\varphi_{N}^{\bm{\mu}}) \in \prod_{i=1}^N \calC(\calX)$ satisfying the Schr\"{o}dinger system \eqref{eq: Schrodinger multi} for every $x_i \in \calX$ for each $i=1,\dots,N$, and such that $\varphi_1^{\bm{\mu}} (x^\circ) = \cdots = \varphi_N^{\bm{\mu}} (x^\circ)$. Furthermore, for every $s \in \NN$, there exists a constant $R_s$ such that $\max_{1 \le i \le N} \| \varphi_i^{\bm{\mu}} \|_{\calC^s(\calX)} \le R_s$ for all $\bm{\mu}  \in \prod_{i=1}^N \calP(\calX)$. Then, Theorem \ref{thm: H derivative potential} extends to the multimarginal setting as follows.

\begin{theorem}
For every $s \in \NN$ and $\bm{\mu} = (\mu^1,\dots,\mu^N) \in \prod_{i=1}^N \calP(\calX)$, the map $\bm{\nu} \mapsto \bm{\varphi}^{\bm{\nu}}, \prod_{i=1}^N \calP_{\mu^i} \subset \prod_{i=1}^N \ell^\infty (B^s) \to \prod_{i=1}^N\calC^s(\calX)$ is Hadamard differentiable tangentially to $\prod_{i=1}^N \overline{\calM_{\mu^i}}^{\ell^\infty(B^s)}$. 
\end{theorem}

The proof is similar to the two-marginal case, so omitted for brevity. Note that the proof of Lemma \ref{lem: Psi functional} relies on the results of \cite{carcamo2020directional}, but their results cover the multimarginal case. Other results extend similarly. For example, let $\hat{\mu}_n^i$ denote the empirical distribution of $n$ i.i.d. data from $\mu^i$ and assume the samples from different marginal distributions are independent.  Then, for $\hat{\bm{\varphi}}_n = \bm{\varphi}^{\hat{\bm{\mu}}_n}$, we have that $\sqrt{n}(\hat{\bm{\varphi}}_n  - \bm{\varphi}^{\bm{\mu}})$ converges in distribution to a zero-mean Gaussian random variable in $\prod_{i=1}^N \calC^s(\calX)$.


\appendix

\section{$m$-out-of-$n$ bootstrap for Sinkhorn null limit}
\label{sec: subsampling}

We consider estimating  the Sinkhorn null limit distribution in Theorem \ref{thm: limit theorem Sinkhorn null} by the two-sample $m$-out-of-$n$ bootstrap. 
Let $\bm{\mu} = (\mu^1,\mu^2) \in \calP(\calX) \times \calP(\calX)$ be arbitrary. For each $i=1,2$, let $X_1^i,\dots,X_n^i$ be i.i.d. data from $\mu^i$ with $\hat{\mu}_n^i = n^{-1}\sum_{j=1}^n \delta_{X_j^i}$. Consider the pooled empirical distribution $\hat{\rho}_{n} =(2n)^{-1}\sum_{j=1}^n (\delta_{X_j^1} + \delta_{X_j^2})$, and let $Z_1^B,\dots,Z_{2m}^B$ be an independent sample from $\hat{\rho}_n$, where $m=m_n \to \infty$. Set
\[
\hat{\rho}_{m,n}^{1,B} = \frac{1}{m}\sum_{j=1}^m \delta_{Z_j^B} \quad \text{and} \quad \hat{\rho}_{m,n}^{2,B} = \frac{1}{m}\sum_{j=m+1}^{2m}\delta_{Z_j^B}.
\]
The following proposition shows that the $m$-out-of-$n$ bootstrap can consistently estimate the null limit law in Theorem \ref{thm: limit theorem Sinkhorn null} when $\mu^1=\mu^2=\mu$.
Recall that $\mathsf{BL}_1(\R)$ denotes the collection of $1$-Lipschitz functions $g: \R \to [-1,1]$.
\begin{proposition}
Consider the above setting and set $\rho = (\mu^1+\mu^2)/2$. Assume $m=o(n)$. Then, we have
\begin{equation}
\sup_{g \in \mathsf{BL}_1(\R)}\Big| \E^B \Big[ g\big(m\bar{\mathsf{S}}_{c,\varepsilon}(\hat{\rho}_{m,n}^{1,B},\hat{\rho}_{m,n}^{2,B})\big)\Big] - \E\big[g(\chi_\rho
)\big] \Big| \to 0
\label{eq: subsampling}
\end{equation}
in probability,  where $\E^B$ denotes the conditional expectation given the sample and $\chi_\rho$ follows the limit law in Theorem \ref{thm: limit theorem Sinkhorn null} with $\mu$ replaced by $\rho$.
\end{proposition}
\begin{remark}
Since $\mathsf{BL}_1(\R)$ is compact w.r.t. the topology of locally uniform convergence by the Ascoli-Arzel\`{a} theorem, the left-hand side of \eqref{eq: subsampling} reduces to the supremum over a countable subcollection of $\mathsf{BL}_1(\R)$ and hence is a proper random variable. 
\end{remark}
\begin{proof}
Fix $s > d/2$. 
Set $\GG_{m,n}^{i,B} = \sqrt{m} (\hat{\rho}_{m,n}^{i,B} - \hat{\rho}_n)$ for $i=1,2$. Given the sample, $(\GG_{m,n}^{1,B},\GG_{m,n}^{2,B}\big)$ has only finitely many possible values, so induces the conditional distribution defined on the Borel $\sigma$-field on $\ell^\infty(B^s) \times \ell^\infty(B^s)$. Then, since $B^s$ is uniformly bounded and Donsker for $\mu^1$ and $\mu^2$, arguing as in the proof of Theorem 3.7.1 in \cite{van1996weak}, we have 
\[
(\GG_{m,n}^{1,B},\GG_{m,n}^{2,B}\big) \stackrel{d}{\to} (\GG^\rho_1,\GG^\rho_2) \quad \text{in} \quad \ell^\infty(B^s) \times \ell^\infty(B^s)
\]
given almost every sequence $X_1^1,X_2^1,\dots,X_1^2,X_2^2,\dots$, where $\GG^\rho_1$ and $\GG^\rho_2$ are independent tight $\rho$-Brownian bridges in $\ell^\infty (B^s)$. Furthermore, $\E[ \| \sqrt{m}(\hat{\rho}_n-\rho) \|_{\infty,B^s} ] = O(\sqrt{m/n}) = o(1)$ by Theorem 2.14.1 in \cite{van1996weak} and the assumption that $m=o(n)$, so $\| \sqrt{m}(\hat{\rho}_n-\rho) \|_{\infty,B^s} \to 0$ in probability by Markov's inequality.  Pick any subsequence $n'$ of $n$ and choose a further subsequence $n''$ for which $\| \sqrt{m''}
(\hat{\rho}_{n''}-\rho) \|_{\infty,B^s} \to 0$ a.s. with $m''=m_{n''}$. Now, since $\sqrt{m}(\hat{\rho}_{m,n}^{i,B} - \rho) = \GG_{m,n}^{i,B} + \sqrt{m}(\hat{\rho}_n-\rho)$,
we have
\[
\big(\sqrt{m''}(\hat{\rho}_{m'',n''}^{1,B} - \rho),\sqrt{m''}(\hat{\rho}_{m'',n''}^{2,B} - \rho)\big) \stackrel{d}{\to} (\GG^\rho_1,\GG^\rho_2) \quad \text{in} \quad  \ell^\infty(B^s) \times \ell^\infty(B^s)
\]
given almost every sequence $X_1^1,X_2^1,\dots,X_1^2,X_2^2,\dots$ Then, by the second-order Hadamard differentiability of $\bar{\mathsf{S}}_{c,\varepsilon}$ (Theorem \ref{thm: second H derivative Sinkhorn}) and the second-order functional delta method (Lemma \ref{lem: second order functional delta method}), we have 
\[
m'' \bar{\mathsf{S}}_{c,\varepsilon}(\hat{\rho}_{m'',n''}^{1,B},\hat{\rho}_{m'',n''}^{2,B}) \stackrel{d}{\to} \frac{1}{2} \Delta_{\rho} \big(\GG^\rho_1,\GG^\rho_2\big)\stackrel{d}{=} \chi_\rho
\]
given almost every sequence $X_1^1,X_2^1,\dots,X_1^2,X_2^2,\dots$ Since the limit is independent of the choice of subsequence $n'$, we obtain the result. 
\end{proof}

\subsection{Numerical experiments}
\label{subsec: Numerical Experiments}

We present numerical experiments to assess the scaling for the Sinkhorn null limit  in \cref{thm: limit theorem Sinkhorn null} as well as the finite sample performance of the two-sample $m$-out-of-$n$ bootstrap procedure. 

\begin{figure}[!htb]
    \centering  \includegraphics[width=\textwidth]{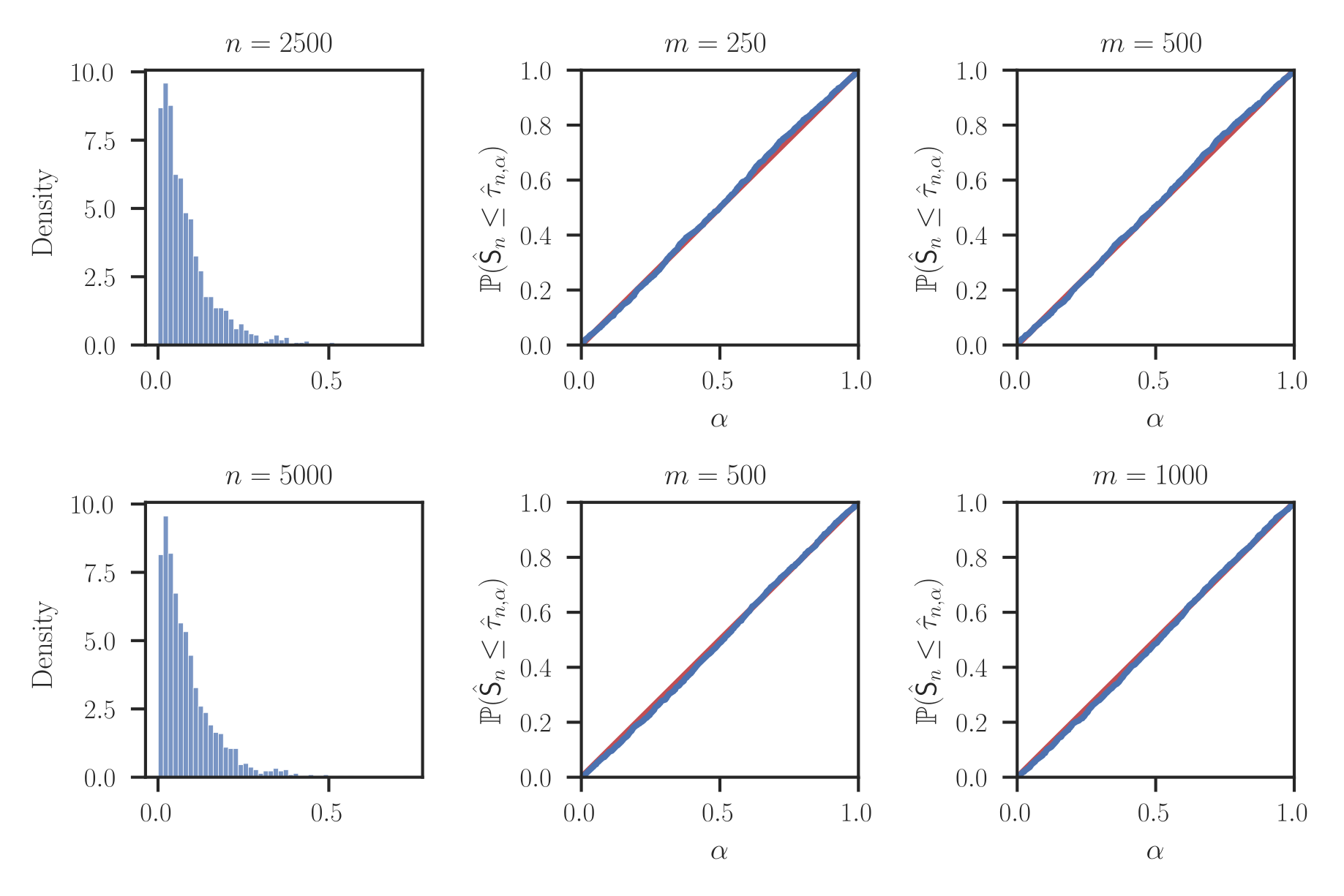}
    \caption{\small The top row consists (from left to right) of a histogram of  $\hat{\mathsf S}_n =n\bar{\mathsf S}_{c,\varepsilon}(\hat \mu_n^{1},\hat \mu_n^{2})$ for $n=2500$, and P-P plots of $\hat{\mathsf S}_n$ generated using the methodology described in \cref{subsec: Numerical Experiments} with $m=0.1n$ and $m=0.2 n$ respectively along with a red $45^{\circ}$ reference line.  The bottom row compiles analogous plots for $n=5000$.}
    \label{fig:SinkhornNull}
\end{figure}

Throughout, the cost function $c$ is the squared Euclidean distance, $\varepsilon=1$, and $\mu^1=\mu^2$ is taken to be the uniform distribution on $(0,1/2)^2$.
Figure \ref{fig:SinkhornNull} consists of two experiments. The first (top row) involves plotting a histogram of $\hat{\mathsf S}_n=n\bar{\mathsf S}_{c,\varepsilon}(\hat \mu_{n}^{1},\hat \mu_{n}^{2})$ for $n=2500$ based on $1500$ repetitions along with corresponding P-P plots for $\hat{\mathsf S}_n$. As the limit distribution $\chi_{\mu}$ in \cref{thm: limit theorem Sinkhorn null} is not defined explicitly, we compute the coverage probabilities, $\mathbb P(\hat{\mathsf S}_n\leq \hat \tau_{n,\alpha})$ for $\alpha\in(0,1)$, where $\hat{\tau}_{n,\alpha}$ is the $\alpha$-quantile of the subsampled distribution. Precisely, for each of the $1500$ repetitions, we construct the subsampled distribution function $\hat F_{m,n}^B$ of $m\bar{\mathsf S}_{c,\varepsilon}(\hat \rho_{m,n}^{1,B},\hat \rho_{m,n}^{2,B})$ for $m=0.1n$ and $m=0.2n$ based on $1000$ repetitions; for each such repetition, we compute the rank of $\hat{\mathsf S}_n$ w.r.t. $\hat F_{m,n}^B$ (i.e., $\hat F_{m,n}^B(\hat{\mathsf S}_n)$), and then approximate $\mathbb P(\hat{\mathsf S}_n\leq \hat \tau_{n,\alpha})$ by  the number of ranks with value less than $\alpha$. The second experiment (bottom row) is performed similarly, but with $n=5000$ rather than $2500$ to see how these results  vary with increasing $n$.

From these experiments, one can see that the empirical distribution of $\hat {\mathsf S}_n$ is observed to be reasonably stable in the finite sample regime, which is consistent with \cref{thm: limit theorem Sinkhorn null}, and the coverage probabilities  $\mathbb P(\hat{\mathsf S}_n\leq \hat \tau_{n,\alpha})$ are close to the $45^\circ$ line uniformly over $\alpha \in (0,1)$, so the $m$-out-of-$n$ bootstrap aproximates well the sampling distribution of $\hat{\mathsf{S}}_n$.

\section{Auxiliary proofs}
\label{sec: auxiliary proofs}

\subsection{Proof of \cref{lem: EOT potential}}
The results are standard (except possibly (iii)), but we include the proof for completeness.

(i). The argument is similar to Proposition 6 in \cite{mena2019statistical}. Pick any pair of EOT potentials $(\varphi_1^\circ,\varphi_2^\circ)$. Update $(\varphi_1^\circ,\varphi_2^\circ)$ as 
\[
\begin{split}
&\varphi_1 (x_1)= -\varepsilon \log \int e^{\frac{\varphi_2^\circ(x_2) - c(x_1,x_2)}{\varepsilon}} \, d\mu^2(x_2), \ x_1 \in 
\calX
, \\
&\varphi_2 (x_2)= -\varepsilon \log \int e^{\frac{\varphi_1(x_1) - c(x_1,x_2)}{\varepsilon}} \, d\mu^1(x_1), \ x_2 \in 
\calX.
\end{split}
\]
The functions $(\varphi_1,\varphi_2)$ are well-defined pointwise by Jensen's inequality.
By construction, $\int e^{\frac{\varphi_1 (x_1') + \varphi_2(x_2) - c(x_1',x_2)}{\varepsilon}} \, d\mu^1(x_1') - 1 = 0$ for all $x_2 \in 
\calX
$. Also, by Jensen's inequality, 
\[
\begin{split}
&\int (\varphi_1 - \varphi_1^\circ) \, d\mu^1 + \int (\varphi_2 - \varphi_2^\circ) \, d\mu^2 \\
&=-\varepsilon \int \log e^{\frac{\varphi_1^\circ - \varphi_1}{\varepsilon}} \, d\mu^1 -\varepsilon \int \log e^{\frac{\varphi_2^\circ - \varphi_2}{\varepsilon}} \, d\mu^2 \\
&\ge -\varepsilon \log \int e^{\frac{\varphi_1^\circ- \varphi_1}{\varepsilon}} \, d\mu^1 -\varepsilon\log  \int e^{\frac{\varphi_2^\circ - \varphi_2}{\varepsilon}} \, d\mu^2  \\
&=-\varepsilon \log \int e^{\frac{\varphi_1^\circ \oplus\varphi_2^\circ-c}{\varepsilon}} d(\mu^1 \otimes \mu^2) - \varepsilon \log \int e^{\frac{\varphi_1 \oplus\varphi_2^\circ-c}{\varepsilon}} d(\mu^1 \otimes \mu^2) \\
&=0, 
\end{split}
\]
so that $(\varphi_1,\varphi_2)$ is a pair of EOT potentials and the inequality above is an equality. In particular, $\int \log e^{\frac{\varphi_2^\circ - \varphi_2}{\varepsilon}} \, d\mu^2 = \log  \int e^{\frac{\varphi_2^\circ - \varphi_2}{\varepsilon}} \, d\mu^2$, and by strict concavity of the logarithm, we have $e^{\frac{\varphi_2^\circ - \varphi_2}{\varepsilon}} = 1$, i.e., $\varphi_2^\circ = \varphi_2$ $\mu^2$-a.e. Thus, 
\[
\int e^{\frac{\varphi_1(x_1) + \varphi_2(x_2') - c(x_1,x_2')}{\varepsilon}} \, d\mu^2(x_2')  = \int e^{\frac{\varphi_1(x_1) + \varphi_2^\circ (x_2') - c(x_1,x_2')}{\varepsilon}} \, d\mu^2(x_2') = 1
\]
for all $x_1 \in 
\calX
$. The other claims are straightforward.

(ii). This follows from the expressions
\[
\begin{split}
&\varphi_1^{\bm{\mu}} (x_1)= -\varepsilon \log \int e^{\frac{\varphi_2^{\bm{\mu}}(x_2) - c(x_1,x_2)}{\varepsilon}} \, d\mu^2(x_2), \ x_1 \in 
\calX
, \\
&\varphi_2^{\bm{\mu}} (x_2)= -\varepsilon \log \int e^{\frac{\varphi_1^{\bm{\mu}}(x_1) - c(x_1,x_2)}{\varepsilon}} \, d\mu^1(x_1), \ x_2 \in 
\calX
.
\end{split}
\]
First, by Lemma 2.1 in \cite{nutz2021entropic}, there exists a version of EOT potentials $(\varphi_1,\varphi_2)$ satisfying \eqref{eq: FOC} with $\| \varphi_1\|_{\infty,\calX} \vee \| \varphi_2 \|_{\infty,\calX} \le 2 \| c \|_{\infty,\calX \times \calX}$. By uniqueness, $\varphi_1^{\bm{\mu}} = \varphi_1 - \frac{1}{2}(\varphi_1(x_1^\circ)-\varphi_2(x_2^\circ))$ and $\varphi_2^{\bm{\mu}} = \varphi_2 + \frac{1}{2}(\varphi_1(x_1^\circ)-\varphi_2(x_2^\circ))$, so $\| \varphi_1^{\bm{\mu}} \|_{\infty,\calX} \vee \| \varphi_2^{\bm{\mu}} \|_{\infty,\calX} \le 3 \| c \|_{\infty,\calX \times \calX}$. 
Derivatives of $\varphi_1^{\bm{\mu}}$ and $\varphi_{2}^{\bm{\mu}}$ can be evaluated by interchanging differentiation and integration, which is guaranteed under the current assumption. See \cite{genevay2019sample,mena2019statistical} for similar arguments.

(iii). For notational convenience, let $\bm{\varphi}_n = \bm{\varphi}^{\bm{\mu}_n}$ and $\bm{\varphi} = \bm{\varphi}^{\bm{\mu}}$. By Part (ii), the Ascoli-Arzel\`{a} theorem, and the diagonal argument, for every subsequence $n'$, there exists a further subsequence $n''$ along which the derivatives $D^k \varphi_{n,i}$ converge in 
$\calC(
\calX
)$ for all $i=1,2$ and $k \in \NN_{0}^d$ with $|k| \le s$, which implies that $\varphi_{n'',i}$ is Cauchy in $\calC^s(
\calX
)$ for $i=1,2$. By completeness of $\calC^s(
\calX
)$, we have $\varphi_{n'',i} \to \bar{\varphi}_i$ in $\calC^s(
\calX
)$ for $i=1,2$. By assumption, $\mu_n^i$ converges weakly to $\mu^i$, which implies that 
$\sup_{f \in \mathrm{BL}_1(
\calX
)}  | \int fd(\mu_n^i - \mu^i)  | \to 0$, where $\mathrm{BL}_1(
\calX
)$ is the class of $1$-Lipschitz functions $f: 
\calX
\to [-1,1]$ (cf. Chapter 1.12 in \cite{van1996weak}). Again, by Part (ii), we see that 
\[
\sup_{n \in \NN} \sup_{(x_1,x_2) \in 
\calX\times\calX}
\left \| \nabla_{x_i} \big (e^{\frac{\varphi_{n,1} (x_1)+ \varphi_{n,2}(x_2) - c(x_1,x_2)}{\varepsilon}}\big) \right \| < \infty.
\]
Hence, for each fixed $x_1 \in 
\calX
$, 
\[
\begin{split}
 \int e^{\frac{\varphi_{n'',1}(x_1) + \varphi_{n'',2}(x_2) - c(x_1,x_2)}{\varepsilon}} \, d\mu^{2}_{n''} (x_2) &= \int e^{\frac{\varphi_{n'',1}(x_1) + \varphi_{n'',2}(x_2) - c(x_1,x_2)}{\varepsilon}} \, d\mu^{2} (x_2) + o(1) \\
&= \int e^{\frac{\bar{\varphi}_1(x_1) + \bar{\varphi}_2(x_2) - c(x_1,x_2)}{\varepsilon}} \, d\mu^{2} (x_2) + o(1),
\end{split}
\]
where the second equality follows as $\| \varphi_{n'',i} - \bar{\varphi}_i\|_{\infty,
\calX
} \to 0$. Since the left-hand side is $\equiv 1$, we conclude that $\int e^{\frac{\bar{\varphi}_1\oplus \bar{\varphi}_2 - c}{\varepsilon}} \, d\mu^{2} \equiv 1$. By symmetry, we also have $\int e^{\frac{\bar{\varphi}_1\oplus \bar{\varphi}_2 - c}{\varepsilon}} \, d\mu^1 \equiv 1$. By construction, $\bar{\varphi}_1 (x_1^\circ) = \bar{\varphi}_2(x_2^\circ)$, so that $\bar{\bm{\varphi}} = \bm{\varphi}$, i.e., $\bm{\varphi}_{n''} \to \bm{\varphi}$ in $\calC^s(
\calX
) \times \calC^s(
\calX
)$. Since the limit does not depend on the choice of subsequence, we have $\bm{\varphi}_n \to \bm{\varphi}$ in $\calC^s(
\calX
) \times \calC^s(
\calX
)$.
\qed

\subsection{Proof of \cref{lem: Psi functional}}
The proof relies on the results from \cite{carlier2020differential}.
We will use the following observation throughout the proof: 
for any continuous functions $f,g$ on $S_i=
\supp(\mu^i)
$, if $f=g$ $\mu^i$-a.e., then $f\equiv g$. Indeed, for every $A \subset 
S_i
$ with $\mu^i$-measure 1, its closure $\overline{A}$ agrees with 
$S_i$, since otherwise 
$
S_i
\setminus \overline{A}$ is a nonempty open set, so $\mu^i (
S_i
\setminus \overline{A}) > 0$, 
which contradicts the assumption that $A$ has $\mu^i$-measure 1. Also, for any continuous function $f$ on $
S_i
$, we have $\| f \|_{L^\infty(\mu^i)} = \| f \|_{\infty,
S_i
}$.

(i). Consider the map $\calT: L^\infty (\mu^1) \times L^\infty(\mu^2) \to L^\infty (\mu^1) \times L^\infty(\mu^2)$ defined by
\[
\calT(\bm{\varphi}) = \left ( \int e^{\frac{\varphi_1 \oplus \varphi_2 - c}{\varepsilon}} \, d\mu^2, \int e^{\frac{\varphi_1 \oplus \varphi_2 - c}{\varepsilon}} \, d\mu^1 \right ).
\]
By Theorem 4.3 in \cite{carlier2020differential}, $\calT$ is injective in the sense that, if $\calT(\bm{\varphi}) = \calT(\tilde{\bm{\varphi}})$ $(\mu^1 \otimes \mu^2)$-a.e., then there exists a constant $a \in \R$ such that $(\tilde{\varphi}_1,\tilde{\varphi}_2) = (\varphi_1+a,\varphi_2-a)$ $(\mu^1 \otimes \mu^2)$-a.e.
If $\bm{\varphi},\tilde{\bm{\varphi}} \in 
\Theta^s
$, then $(\tilde{\varphi}_1,\tilde{\varphi}_2) \equiv (\varphi_1+a,\varphi_2-a)$ on $S$, but because of the normalization $\varphi_1(x_1^\circ) = \varphi_2(x_2^\circ)$ and $\tilde{\varphi}_1(x_1^\circ) =\tilde{\varphi}_2(x_2^\circ)$, we have $a=0$, i.e., $\tilde{\bm{\varphi}} \equiv \bm{\varphi}$.  This shows injectivity of $\Psi_{\bm{\mu}}$. 

To show that the inverse of $\Psi_{\bm{\mu}}$ is continuous at $0$, it suffices to show that
\[
\bm{\varphi}_n \in \Theta^s, \| \Psi_{\bm{\mu}}(\bm{\varphi}_n)  \|_{\DD} \to 0 \Rightarrow \| \bm{\varphi}_n - \bm{\varphi}^{\bm{\mu}}\vert_{S} \|_{\DD} \to 0. 
\]
Recall $\Psi_{\bm{\mu}}(\bm{\varphi}^{\bm{\mu}}\vert_{S}) \equiv 0$ on $S$. Observe that the map $\Psi_{\bm{\mu}}$ makes sense on $\DD$ and is continuous from $\DD$ into $\DD$. Since $\bm{\varphi}_n \in \Theta^s$, by the Ascoli-Arzel\`a theorem, for any subsequence $n'$, there exists a further subsequence $n''$ along which $\bm{\varphi}_{n''} \to \bar{\bm{\varphi}}$ in $\DD$ for some $\bar{\bm{\varphi}} \in \DD$. Since $\Psi_{\bm{\mu}}$ is continuous from $\DD$ into $\DD$, we have $\Psi_{\bm{\mu}}(\bar{\bm{\varphi}}) \equiv 0$. By construction, $\bar{\varphi}_1(x_1^\circ) = \bar{\varphi}_2(x_2^\circ)$, so by \cref{lem: EOT potential} (i), we have $\bar{\bm{\varphi}} \equiv \bm{\varphi}^{\bm{\mu}}\vert_{S}$. Since the limit $\bm{\varphi}^{\bm{\mu}}\vert_{S}$ is independent of the choice of subsequence, we have $\bm{\varphi}_n \to \bm{\varphi}^{\bm{\mu}}\vert_{S}$ in $\DD$.  

(ii). The first claim is straightforward. To show the second claim, it suffices to show that 
\[
\inf_{\bm{h} \in \lin (\Theta^s), \| \bm{h} \|_{\DD} = 1} \| \dot \Psi_{\bm{\mu}} (\bm{h}) \|_{\DD} > 0. 
\]
Consider the map $\dot \calT_{\bm{\mu}}: L^\infty (\mu^1) \times L^\infty (\mu^2) \to L^\infty (\mu^1) \times L^\infty (\mu^2)$ defined by 
\[
\dot \calT_{\bm{\mu}} (\bm{h}) = \left ( \varepsilon^{-1}\int e^{\frac{\varphi_1^{\bm{\mu}} \oplus \varphi_2^{\bm{\mu}} - c}{\varepsilon}} (h_1 \oplus h_2) \, d\mu^2, \varepsilon^{-1}\int e^{\frac{\varphi_1^{\bm{\mu}} \oplus \varphi_2^{\bm{\mu}} - c}{\varepsilon}} (h_1 \oplus h_2) \, d\mu^1  \right ).
\]
Equip $L^\infty (\mu^1) \times L^\infty (\mu^2)$ with a product norm $\| \bm{h} \|_{L^\infty (\bm{\mu})} := \| h_1 \|_{L^\infty(\mu^1)} \vee \| h_2 \|_{L^\infty (\mu^2)}$.  
By Proposition 3.1 in \cite{carlier2020differential},
\[
\inf_{\| \bm{h} \|_{_{L^\infty (\bm{\mu})}}=1, \int h_1 \, d\mu^1 = 0} \| \dot \calT_{\bm{\mu}} (\bm{h}) \|_{L^\infty (\bm{\mu})} > 0.
\]
For $\bm{h} \in \lin (\Theta^s)$, define $\bar{\bm{h}} = (h_1-\int h_1 \, d\mu^1,h_2+\int h_1 \, d\mu^1)$.  Then  $\| \dot \calT_{\bm{\mu}} (\bar{\bm{h}}) \|_{L^\infty (\bm{\mu})} = \| \dot \calT_{\bm{\mu}} (\bm{h}) \|_{L^\infty (\bm{\mu})} =  \| \dot{\Psi}_{\bm{\mu}}(\bm{h}) \|_{\DD}$. It remains to show that 
\[
\inf_{\bm{h} \in \lin (\Theta^s), \| \bm{h} \|_{\DD}=1} \| \bar{\bm{h}} \|_{\DD} > 0. 
\]
Suppose on the contrary that $\inf_{\bm{h} \in \lin (\Theta^s), \| \bm{h} \|_{\DD}=1} \| \bar{ \bm{h}} \|_{\DD}=0$. 
Then, there exists a sequence $\bm{h}_{n}  \in \lin (\Theta^s)$ with $\| \bm{h}_{n} \|_{\DD} = 1$ such that $\| \bar{\bm{h}}_{n} \|_{\DD} \to 0$. Since $\int h_{n,1} \, d\mu^1$ is bounded, there exists a subsequence along which $\int h_{n,1} \, d\mu^1 \to a$ for some $a \in \R$. Along the subsequence, $\|  \bm{h}_{n} - (a,-a) \|_{\DD}\to 0$. However, since $h_{n,1} (x_1^\circ) = h_{n,2}(x_2^\circ)$ by construction, we must have $a=0$, i.e., $\| \bm{h}_{n} \|_{\DD} \to 0$, which contradicts the assumption that $\| \bm{h}_{n} \|_{\DD}=1$. 
\qed

\section{Technical tools}

\subsection{Hadamard differentiability and functional delta method}
\label{sec: functional delta}

In this appendix, we review concepts of Hadamard differentiability and the functional delta methods. Our exposition mostly follows \cite{romisch2004}. Other standard references are \cite{van1996weak,vanderVaart1998asymptotic}.

Let $\mathfrak{D},\mathfrak{E}$ be normed spaces and $\phi: \Theta \subset \mathfrak{D} \to \mathfrak{E}$ be a map. We say that $\phi$ is 
\textit{Hadamard directionally differentiable}  at $\theta \in \Theta$  if there exists a map $\phi'_{\theta}: \mathfrak{T}_{\Theta}(\theta) \to \mathfrak{E}$ such that 
\begin{equation}
\lim_{t \downarrow 0} \frac{\phi(\theta_t) - \phi(\theta)}{t} = \phi_{\theta}'(h)
\label{eq: H derivative}
\end{equation}
for any sequence $(\theta_t)_{t>0} \subset \Theta$ with $t^{-1}(\theta_t -\theta) \to h$ as $t \downarrow 0$, where $\mathfrak{T}_{\Theta}(\theta)$ is the tangent (or adjacent) cone to $\Theta$ at $\theta$,
\[
\mathfrak{T}_{\Theta}(\theta) = \left \{ h \in \mathfrak{D}: h = \lim_{t \downarrow 0} \frac{\theta_t -\theta}{t} \ \text{for some $\theta_t \to \theta$ in $\Theta$}, t \downarrow 0  \right \}.
\]
The derivative $\phi_{\theta}'$ is continuous (cf. Proposition 3.1 in \cite{shapiro1990}) and positively homogeneous (but need not be linear). Furthermore, the tangent cone $\mathfrak{T}_{\Theta}(\theta)$ is closed, and: (i) if $\Theta$ is open, then $\mathfrak{T}_{\Theta}(\theta)$ agrees with $\mathfrak{D}$; and (ii) if $\Theta$ is convex, then $\mathfrak{T}_{\Theta}(\theta)$ agrees with $\overline{\{ t(\vartheta - \theta) : \vartheta \in \Theta, t > 0\}}^{\mathfrak{D}}$ (cf. Chapter 4 in \cite{aubin2009set}).\footnote{Indeed, \cite{romisch2004} defines the Hadamard derivative on the (Bouligand) contingent cone, which is in general slightly bigger than the adjacent cone. For our purpose, this difference is immaterial. Note that both concepts agree when $\Theta$ is convex.}

If \eqref{eq: H derivative} only holds for $h \in \mathfrak{D}_0$ for a subset $\mathfrak{D}_0 \subset \mathfrak{T}_{\Theta}(\theta)$, then we say that $\phi$ is 
Hadamard directionally differentiable  at $\theta \in \Theta$ \textit{tangentially} to $\mathfrak{D}_0$. In that case, the derivative $\phi_{\theta}'$ is defined only on $\mathfrak{D}_0$. Finally, if the derivative $\phi_{\theta}'$ is linear, then we say that $\phi$ is \textit{Hadamard differentiable}  at $\theta$ (tangentially to $\mathfrak{D}_0$ if $\phi_{\theta}'$ is defined only on $\mathfrak{D}_0$). The tangent set $\mathfrak{D}_0$ need not be a vector subspace of $\mathfrak{D}$, so by linearity, we mean that, for any $\alpha_1,\dots,\alpha_J \in \R$ and $h_1,\dots,h_J \in \mathfrak{D}_0$, whenever $\sum_{j=1}^J\alpha_j h_j =0$, it holds that $ \sum_{j=1}^J \alpha_j \phi_{\theta_0}'(h_j) = 0$, which is equivalent to $\phi_{\theta_0}'$ admitting a linear extension to the linear hull of $\mathfrak{D}_0$ by Lemma 2.5.3 in \cite{dudley2014uniform}.

\begin{lemma}[Functional delta method; \cite{romisch2004}]
\label{lem: functional delta method}
Let $\mathfrak{D},\mathfrak{E}$ be normed spaces and $\phi: \Theta \subset \mathfrak{D} \to \mathfrak{E}$ be a map that is Hadamard directionally differentiable at $\theta \in \Theta$ tangentially to a set $\mathfrak{D}_0 \subset \mathfrak{T}_{\Theta}(\theta)$. Let $T_n: \Omega \to \Theta$ be maps such that $r_n (T_n - \theta) \stackrel{d}{\to} T$ for some $r_n \to \infty$ and Borel measurable map $T: \Omega \to \mathfrak{D}$ with values in a separable subset of $\mathfrak{D}_0$. Then the following hold:
(i) $r_n \big(\phi(T_n) - \phi(\theta)\big) \stackrel{d}{\to} \phi_{\theta}'(T)$; 
(ii) If in addition $\Theta$ is convex and $\mathfrak{D}_0 = \mathfrak{T}_{\Theta}(\theta)$, then $r_n \big(\phi(T_n) - \phi(\theta)\big) - \phi_{\theta}'(r_n(T_n-\theta)) \to 0$ in outer probability. 
\end{lemma}

\begin{remark}
\label{rem: Hadamard differentiability}
Our definition of Hadamard differentiability is slightly different from \cite{van1996weak,vanderVaart1998asymptotic}, in that those references do not require the tangent set $\mathfrak{D}_0$ to be a subset of $\mathfrak{T}_{\Theta}(\theta)$. Our modification is made to be consistent with the definition of Hadamard directional differentability in \cite{romisch2004}. However, this modification is innocuous since, for any sequence $(\theta_t)_{t > 0} \subset \Theta$ with $h := \lim_{t \downarrow 0}t^{-1}(\theta_t-\theta)$ (if exists), we must have $h \in \mathfrak{T}_{\Theta}(\theta)$.
\end{remark}

The proof of \cref{thm: limit theorem Sinkhorn null} relies on the second-order functional delta method, which we describe next. 
We say that a map $\phi: \Theta \subset \mathfrak{D} \to \mathfrak{E}$ (with $\Theta$ being convex) is \textit{second-order Hadamard directionally differentiable} at $\theta \in \Theta$ if it is (first-order) Hadamard directionally differentiable at $\theta$ and there exists a map $\phi_{\theta}'': \mathfrak{T}_{\Theta}(\theta) \to \mathfrak{E}$ such that
\[
\lim_{t \downarrow 0} \frac{\phi (\theta_t) - \phi(\theta) - t\phi_{\theta}'(h_t)}{t^2/2} = \phi_{\theta}''(h)
\]
for any sequence $(\theta_t)_{t>0} \subset \Theta$ with $h_t:=t^{-1}(\theta_t -\theta) \to h$ as $t \downarrow 0$ (note: as $\Theta$ is convex, $h_t \in \mathfrak{T}_{\Theta}(\theta)$, so that $\phi_{\theta}'(h_t)$ is well-defined). The map $\phi_{\theta}''$ is continuous and positively homogeneous of degree $2$. 

\begin{lemma}[Second-order chain rule]
\label{lem:SecondOrderChainRule}
Let $\mathfrak{D},\mathfrak{E},\mathfrak{F}$ be normed spaces, $\Theta \subset \mathfrak{D}$, $\Xi\subset \mathfrak{E}$ be convex, $\phi: \Theta  \to \Xi$ be a map that is twice  Hadamard directionally differentiable at $\theta \in \Theta$, and $\psi:\Xi\to\mathfrak{F}$ be twice  Hadamard directionally differentiable at $\phi(\theta)$ with linear first-order derivative at $\phi(\theta)$ defined on the linear hull of $\mathfrak{T}_{\Xi}(\phi(\theta))$. Then, the composition $\psi\circ\phi:\Theta\to \mathfrak{F}$ is twice  Hadamard directionally differentiable at $\theta$ with 
\[
[\psi\circ \phi]''_{\theta}= \psi''_{\phi(\theta)}\circ\phi'_{\theta}+\psi'_{\phi(\theta)}\circ\phi''_{\theta}.
\]
\end{lemma}
\begin{proof}
Let $(\theta_t)_{t>0} \subset \Theta$ be such that $h_t:=t^{-1}(\theta_t -\theta) \to h$ in $\mathfrak D$ as $t \downarrow 0$. Observe that $\left(\phi(\theta_t)\right)_{t>0}\subset \Xi$ with $t^{-1}(\phi(\theta_t)-\phi(\theta))\to \phi'_{\theta}(h)$ as $t\downarrow 0$, so that
\[
    \psi(\phi(\theta_t))-\psi(\phi(\theta))=t\psi'_{\phi(\theta)}(t^{-1}(\phi(\theta_t)-\phi(\theta)))+\frac{t^2}{2}\psi''_{\phi(\theta)}(\phi'_{\theta}(h))+o(t^2).  
\]
By linearity  of $\psi'_{\phi(\theta)}$, 
\[
    t\psi'_{\phi(\theta)}(t^{-1}(\phi(\theta_t)-\phi(\theta)))-t\psi'_{\phi(\theta)}(\phi'_{\theta}(h_t))=\psi'_{\phi(\theta)}(\phi(\theta_t)-\phi(\theta)-t\phi'_{\theta}(h_t)). 
\]
It follows from continuity and positive homogeneity of $\psi'_{\phi(\theta)}$ that
\[
\lim_{t \downarrow 0} \frac{\psi(\phi (\theta_t)) - \psi(\phi(\theta)) - t\psi'_{\phi(\theta)}(\phi_{\theta}'(h_t))}{t^2/2} =\psi''_{\phi(\theta)}(\phi'_{\theta}(h))+\psi'_{\phi(\theta)}(\phi''_{\theta}(h)). 
\]
\end{proof}

\begin{lemma}[Second-order functional delta method; \cite{romisch2004}]
\label{lem: second order functional delta method}
Let $\mathfrak{D},\mathfrak{E}$ be normed spaces, $\Theta \subset \mathfrak{D}$ be convex,  and $\phi: \Theta  \to \mathfrak{E}$ be a map that is second-order  Hadamard directionally differentiable at $\theta \in \Theta$. Let $T_n: \Omega \to \Theta$ be maps such that $r_n (T_n - \theta) \stackrel{d}{\to} T$ for some $r_n \to \infty$ and  Borel measurable map $T: \Omega \to \mathfrak{D}$ with values in a separable subset of $\mathfrak{T}_{\Theta}(\theta)$. Then, $r_n^2 \big (\phi (T_n) - \phi(\theta) - \phi_{\theta}'(T_n-\theta)\big) \stackrel{d}{\to} \frac{1}{2} \phi_{\theta}''(T)$ and $r_n^2 \big (\phi (T_n) - \phi(\theta) - \phi_{\theta}'(T_n-\theta) -  \frac{1}{2} \phi_{\theta}''(T_n-\theta)\big) \to 0$ in outer probability. 
\end{lemma}

\subsection{Hadamard differentiability of $Z$-functional}
\label{sec: Z-functional}

Let $\DD$ be a Banach space, $\Theta$ be an arbitrary nonempty subset of $\DD$, and $\LL$ be another Banach space. Let $Z(\Theta,\LL)$ be the subset of $\ell^\infty(\Theta,\LL)$ consisting of all maps with at least one zero, and let $\phi: Z(\Theta,\LL) \to \Theta$ be a map that assigns one of its zeros of to each $z \in Z(\Theta,\LL)$. Following \cite{van1996weak}, we call $\phi$ the $Z$-functional. 
We say that a map $\Psi: \Theta \subset \D \to \LL$ is Fr\'{e}chet differentiable at $\theta_0 \in \Theta$ if there exists a bounded linear operator $\dot \Psi_{\theta_0}: \lin (\Theta) \to \LL$ such that 
\[
\lim_{\substack{\| h \|_{\D} \to 0\\ \theta_0 + h\in \Theta}} \frac{\| \Psi (\theta_0+h) - \Psi (\theta_0) - \dot\Psi_{\theta_0}(h)\|_{\LL}}{\| h \|_{\DD}} = 0.
\]
The following is taken from  Lemma 3.9.34 in \cite{van1996weak}, with a minor modification to adapt to our definition of Hadamard differentiability; cf. \cref{rem: Hadamard differentiability}.

\begin{lemma}[Hadamard derivative of $Z$-functional]
\label{lem: Z-functional}
Suppose that (i) $\Psi: \Theta \to \LL$ is uniformly norm-bounded, one-to-one, possesses a zero at $\theta_0 \in \Theta$, and has an inverse that is continuous at $0$, and (ii) 
$\Psi$ is Fr\'{e}chet differentiable at $\theta_0 \in \Theta$ with derivative $\dot \Psi_{\theta_0}$ that is one-to-one and such that its inverse is continuous on $\dot \Psi_{\theta_0}(\lin(\Theta))$. 
Then, $\phi$ is Hadamard differentiable at $\Psi$ tangentially to the set 
\[
\begin{split}
\mathcal{Z}_{\Psi} =&\Big \{ z \in \ell^\infty (\Theta,\LL) : z=\lim_{t \downarrow 0} \frac{z_t -\Psi}{t} \ \text{for some} \ z_t \to \Psi \ \text{in} \ Z(\Theta,\LL), \  t \downarrow  0 \Big\} \\
&\bigcap \Big \{ z \in \ell^\infty (\Theta,\LL) : \text{$z$ is continuous at $\theta_0$} \Big \}. 
\end{split}
\]
The derivative is given by $\phi_{\Psi}' (z) = -\dot \Psi_{\theta_0}^{-1} (z(\theta_0))$.
\end{lemma}

\begin{remark}
\label{rem: Z-functional}
It is \textit{a priori} not clear whether $z(\theta_0)$ is in the domain of $\dot \Psi_{\theta_0}^{-1}$, but this follows from the proof of Lemma 3.9.34 in \cite{van1996weak}. Indeed, let  $(z_t)_{t > 0} \subset \ell^\infty (\Theta,\LL)$ be such that $\Psi +tz_t \in Z(\Theta,\LL)$ for $t$ sufficiently small, $z_t \to z$ as $t \downarrow 0$, and  $z$ is continuous at $\theta_0$. Set $\theta_t = \phi (\Psi+tz_t)$. Arguing as in the proof of Lemma 3.9.34 in \cite{van1996weak}, we have $\Psi(\theta_t)=\dot \Psi_{\theta_0} (\theta_t - \theta_0) + R(t)$ with $\| R(t) \|_{\LL} = o(t)$ as $t \to 0$, so that $t^{-1}\{\Psi(\theta_t) - R(t)\} = \dot \Psi_{\theta_0} \big( t^{-1}(\theta_t-\theta_0)\big) \in \dot \Psi_{\theta_0} (\lin (\Theta))$ while $t^{-1}\{\Psi(\theta_t) - R(t)\} = -z_t (\theta_t) - t^{-1}R(t) \to -z (\theta_0)$ in $\LL$. Hence $z(\theta_0) \in \overline{\dot \Psi_{\theta_0} (\lin (\Theta))}^{\LL}$.
\end{remark}

The following lemma concerns the second-order Hadamard derivative for the $Z$-functional, which is used in the proof of \cref{thm: second order H derivative}.

\begin{lemma}[Second-order Hadamard derivative of $Z$-functional]
\label{lem: second order H derivative}
Consider the assumption of the preceding lemma. Assume further that $\Psi$ is twice Fr\'{e}chet differentiable at $\theta_0$, in the sense that there exists a continuous operator $\ddot \Psi_{\theta_0}: \overline{\lin (\Theta)}^{\DD} \to \LL$ positively homogeneous of degree 2 such that
\[
\lim_{\substack{\| h \|_{\D} \to 0 \\ \theta_0+h \in \Theta}} \frac{\| \Psi (\theta_0+h) - \Psi (\theta_0) - \dot\Psi_{\theta_0}(h) - \frac{1}{2} \ddot \Psi_{\theta_0} (h) \|_{\LL}}{\| h \|_{\DD}^2} = 0.
\]
Let $(z_t)_{t > 0} \in \ell^\infty (\Theta,\LL)$ be a sequence of maps such that (i) $\Psi + tz_t \in Z(\Theta,\LL)$ for $t$ sufficiently small, (ii) $z_t \to z$ as $t \downarrow 0$ for some $z$ that is continuous at $\theta_0$, (iii) $z_t (\theta_0) \in \overline{\dot \Psi_{\theta_0}(\lin (\Theta))}^{\LL}$ for $t$ sufficiently small, and (iv) there exists a limit $t^{-1}(z_t(\theta_t) - z_t(\theta_0)) \to \dot z_{\theta_0}$ in $\LL$ as $t \downarrow 0$ with $\theta_t = \phi (\Psi + tz_t)$. Then, we have
\[
\frac{\phi (\Psi+tz_t) - \phi(\Psi)-t\phi_{\Psi}' (z_t)}{t^2/2} \to  -\dot \Psi_{\theta_0}^{-1} \left ( 2 \dot z_{\theta_0} +  \ddot \Psi_{\theta_0} (\phi_{\Psi}'(z)) \right).
\]
\end{lemma}

The above definition of second Fr\'{e}chet derivative differs from the standard one (cf. \cite{zeidler2012applied}), but suffices for our purpose.

\begin{proof}
Recall $\theta_t = \phi (\Psi + t z_t)$, i.e., $\Psi (\theta_t) + tz_t (\theta_t) = 0$. Arguing as in the proof of Lemma 3.9.34 in \cite{van1996weak}, we have $\| \theta_t - \theta_0 \| = O(t)$. Thus,
\[
\underbrace{\Psi (\theta_t)}_{-tz_t(\theta_t)} - \underbrace{\Psi (\theta_0)}_{=0} = \dot \Psi_{\theta_0} (\theta_t - \theta_0) + \frac{1}{2} \ddot \Psi_{\theta_0}(\theta_t-\theta_0) + o(t^2).  
\]
Subtracting $-tz_t(\theta_0)$ from both sides, we have 
\[
-t(z_t (\theta_t) - z_t(\theta_0)) = \dot \Psi_{\theta_0} (\theta_t - \theta_0 - t\phi_{\Psi}'(z_t)) + \frac{1}{2}\ddot \Psi_{\theta_0}(\theta_t-\theta_0) + o(t^2).
\]
Since $t^{-1}(\theta_t -\theta_0) \to \phi_{\Psi}'(z)$ by the preceding lemma, we have 
\[
\ddot \Psi_{\theta_0}(\theta_t-\theta_0) = t^2 \ddot \Psi_{\theta_0} (\phi_{\Psi}'(z)) + o(t^2).
\]
Also, by assumption, 
\[
-t(z_t (\theta_t) - z_t(\theta_0)) 
= -t^2 \dot z_{\theta_0}  + o(t^2). 
\]
Conclude that
\[
\frac{\theta_t - \theta_0 - t\phi'_{\Psi}(z_t)}{t^2/2} \to  -\dot \Psi_{\theta_0}^{-1} \left ( 2 \dot z_{\theta_0} +  \ddot \Psi_{\theta_0} (\phi_{\Psi}'(z)) \right).
\]
\end{proof}

\section{Other auxiliary results}

\begin{lemma}[Convergence in $\ell^\infty(B^s)$ implies weak convergence]
\label{lem: convergence determining}
Let $\calX \subset \R^d$ be a compact set that agrees with the closure of its interior. Pick any $s \in \NN$ and set $B^s$ to be the unit ball in $\calC^s(\calX)$. For $\mu_n,\mu \in \calP(\calX)$, if $\mu_n \to \mu$ in $\ell^\infty (B^s)$, then $\mu_n \to \mu$ weakly. 
\end{lemma}

\begin{proof}
Pick any bounded $1$-Lipschitz function $f$ on $\calX$. By the Kirszbraun-McShane theorem, we may extend $f$ to a $1$-Lipschitz function on $\R^d$, which we denote by the same symbol $f$. Let $K: \R^d \to \R$ be a compactly supported smooth density function and approximate $f$ by $f_t = t^{-d} \int_{\R^d} f(y) K((y-\cdot)/t) dy = \int_{\R^d} f (\cdot + tz)K(z) dz$ for $t > 0$. As $f$ is $1$-Lipschitz, $f_t$ is smooth and $1$-Lipshitz with $\| f - f_t \|_{\infty,\R^d} \le t$. The restriction of $f_t$ to $\calX$ belongs to $\calC^s(\calX)$, so that for any $t > 0$,
\[
\int_{\calX} f \, d\mu_n \le \int_{\calX} f_t \, d\mu_n + t \le \int_{\calX} f \, d\mu + 2t + o(1), \quad n \to \infty. 
\]
The reverse inequality follows similarly, so that we have $\limsup_{n \to \infty}|\int_{\calX} f d(\mu_n-\mu) | \le 2t$. Sending $t \downarrow 0$, we have $\int_{\calX} f \, d\mu_n \to \int_{\calX} f\, d\mu$, implying $\mu_n \to \mu$ weakly. 
\end{proof}
\begin{lemma}[Support of Brownian bridge]
\label{lem: support}
Let $\mu$ be a probability measure on a measurable space $S$ and $\calF \subset L^2 (\mu)$ be a $\mu$-pre-Gaussian class, i.e., there exists a tight $\mu$-Brownian bridge $G_\mu$ in $\ell^\infty (\calF)$. Let 
\[
\mathfrak{M}_\mu =\Big \{ g \, d\mu : \text{$g$ is a bounded measurable function on $S$ with $\mu$-mean zero} \Big \}.
\]
Then $\supp (G_\mu) \subset \overline{\mathfrak{M}_\mu}^{\ell^\infty (\calF)}$.
\end{lemma}

\begin{proof}
Let $\calC_u (\calF)$ denote the space of uniformly continuous functions on $\calF$ relative to the pseudometric $d_{\mu}(f,g) = \sqrt{\Var_{\mu}(f-g)}$. Since $\calF$ is $\mu$-pre-Gaussian, $\calF$ is totally bounded for $d_\mu$ (so that $\calC_u (\calF)$ is a closed subspace of $\ell^\infty(\calF)$) and $G_\mu \in \calC_u(\calF)$ a.s. (cf. Example 1.5.10 in \cite{van1996weak}). For every $f \in \calF$ and  $g_1,g_2 \in L^2(\mu)$ with $\mu$-mean zero, 
\[
\int f (g_1-g_2) \, d\mu = \int (f- {\textstyle \int f\, d\mu}) (g_1-g_2) \, d\mu \le \sqrt{\Var_\mu(f)} d_{\mu}(g_1-g_2)
\]
by the Cauchy-Schwarz inequality. Since $\sup_{f \in \calF}\Var_\mu(f) < \infty$ by total boundedness of $\calF$ w.r.t. $d_\mu$, we have that $g \, d\mu \in \overline{\mathfrak{M}_\mu}^{\ell^\infty (\calF)}$ for every $g \in L^2(\mu)$ with mean zero. 

By Lemma 5.1 in \cite{van2008reproducing}, the support $\supp(G_\mu)$ agrees with the $\| \cdot \|_{\infty,\calF}$-closure of the reproducing kernel Hilbert space (RKHS) for $G_\mu$ (think of $G_\mu$ as a zero-mean Gaussian random variable in $\calC_u(\calF)$, which is a separable Banach space). There are two ways to define the RKHS for $G_\mu$; by viewing $G_\mu$ as a stochastic process or as a random variable with values in the Banach space $\calC_u(\calF)$. In this case, however, they both agree; see Theorem 2.1 in \cite{van2008reproducing}. With this in mind, any element of the RKHS for $G_\mu$ is of the form 
\[
f \mapsto \E[ G_\mu(f) X], \ X \in \overline{\lin \{ G_\mu(g) : g \in \calF \}}^{L^2(\Prob)}. 
\]
For $X = \sum_{i=1}^s \alpha_i G_\mu(g_i)$ with $\alpha_i \in \R$ and $g_i \in \calF$, we have 
\[
\begin{split}
\E[G_\mu(f)X] &= \sum_{i=1}^s \alpha_i \E[G_\mu(f)G_\mu(g_i)] = \sum_{i=1}^s \alpha_i \Cov_\mu(f,g_i) \\
&=   \int f \Big (\sum_{i=1}^s\alpha_i(g_i-{\textstyle \int g_i\, d\mu})\Big)\, d\mu,
\end{split}
\]
so that $\E[G_\mu(\cdot)X] \in \overline{\mathfrak{M}_\mu}^{\ell^\infty (\calF)}$. Furthermore, for any $X \in \overline{\lin \{ G_\mu(g) : g \in \calF \}}^{L^2(\Prob)}$, choose $X_n \in \lin \{ G_\mu(g) : g \in \calF \}$ such that $\E[|X_n-X|^2] \to 0$. Then, 
\[
\| \E[G_\mu(\cdot)X_n] - \E[G_\mu(\cdot)X]  \|_{\infty,\calF} \le \sqrt{\sup_{f \in \calF} \Var_\mu(f)} \sqrt{\E[|X_n-X|^2]} \to 0, 
\]
which shows that $\E[G_\mu(\cdot) X] \in \overline{\mathfrak{M}_\mu}^{\ell^\infty (\calF)}$. Conclude that $\supp(G_\mu) \subset \overline{\mathfrak{M}_\mu}^{\ell^\infty (\calF)}$.
\end{proof}

\begin{lemma}[Weak convergence in product space]
\label{lem: product weak limit}
Let $S,T$ be nonempty sets and let $X_n = (X_n(s))_{s \in S},Y_n = (Y_n(t))_{t \in T}$ be sequences of stochastic processes with bounded paths.  Suppose that, marginally, $X_n$ and $Y_n$ converge in distribution  to  tight random variables in $\ell^\infty (S)$ and $\ell^\infty (T)$, respectively. Then, if the finite-dimensional distributions of $(X_n,Y_n)$ converge weakly, i.e., for every $s_1,\dots,s_m \in S, t_1,\dots,t_\ell \in T$, $(X_n(s_1),\dots,X_n(s_m),Y_n(t_1),\dots,Y_n(t_\ell))$ jointly converges in distribution (in $\R^{m+\ell}$), then $(X_n,Y_n) \stackrel{d}{\to} (X,Y)$ in $\ell^\infty (S) \times \ell^\infty (T)$ for some tight limit $(X,Y)$.
\end{lemma}

\begin{proof}
The lemma follows from Prohorov's theorem, upon observing that a tight limit is uniquely determined by the finite-dimensional convergence. One way to show the latter is to apply Lemma 1.3.12 in \cite{van1996weak}. We present another more direct proof.
Let $(X,Y): \Omega \to \ell^\infty (S) \times \ell^\infty (T)$ be a tight (Borel measurable) random variable. Then, $X$ is tight in $\ell^\infty (S)$, so there exits a pseudometric $\rho_S$ on $S$ that makes $S$ totally bounded and such that $X \in \calC_u(S)$ a.s., where $\calC_u(S)$ is the space of $\rho_S$-uniformly continuous functions on $S$ equipped with the sup-norm $\| \cdot \|_{\infty,S}$; see Chapter 1.5 in \cite{van1996weak}. Consider the Borel $\sigma$-field on $\calC_u(S)$. Define $\calC_u(T)$ analogously. Since $\calC_u(S)$ and $\calC_u(T)$ are separable, the Borel $\sigma$-field on $\calC_u(S) \times \calC_u(T)$ (defined w.r.t. the product topology) agrees with the product $\sigma$-field. In turn, the Borel $\sigma$-field on $\calC_u(S)$ agrees with the cylinder $\sigma$-field (i.e., the smallest $\sigma$-field that makes every coordinate projection $f \mapsto f(s)$ measurable).
For $s_1,\dots,s_m \in S$, let $\pi_{s_1,\dots,s_m}^S: \calC_u (S) \to \R^m$ be the projection onto $s_1,\dots,s_m$, i.e., $\pi_{s_1,\dots,s_m}^S (f) = (f(s_1),\dots,f(s_m))$. Define $\pi_{t_1,\dots,t_\ell}^T$ analogously. Then, the collection of sets of the form 
\[
\begin{split}
&[\pi_{s_1,\dots,s_m}^S]^{-1}(A) \times [\pi_{t_1,\dots,t_\ell}^T]^{-1}(B) \subset \calC_u(S) \times \calC_u(T), \\ &s_i \in S, t_j \in T, \ A \subset \R^m, B \subset \R^\ell:  \text{Borel sets}
\end{split}
\]
is a $\pi$-system that generates the Borel $\sigma$-field on $\calC_u(S) \times \calC_u(T)$. Hence, the joint law of $(X,Y)$  is uniquely determined by the collection of the joint laws of random vectors of the form $(X(s_1),\dots,X(s_m),Y(t_1),\dots,Y(t_\ell))$. 

The rest of the proof is standard. Since $X_n$ and $Y_n$ are marginally asymptotically tight and asymptotically measurable (Lemma 1.3.8 in \cite{van1996weak}), $(X_n,Y_n)$ is jointly asymptotically tight and asymptotically measurable (Lemmas 1.4.3 and 1.4.4 in \cite{van1996weak}). By Prohorov's theorem (Theorem 1.3.9 in \cite{van1996weak}), every subsequence has a further subsequence weakly convergent  to a tight law. By finite-dimensional convergence, the weak limit is unique. Hence $(X_n,Y_n) \stackrel{d}{\to} (X,Y)$ in $\ell^\infty (S) \times \ell^\infty (T)$ for some tight limit $(X,Y)$.
\end{proof} 

\bibliographystyle{alpha}
\bibliography{ref}
\end{document}